\pdfoutput=1
\documentclass[10pt]{article}
\usepackage{latexsym,amsfonts,amsthm,amsmath,amscd,amssymb,color}
\usepackage[normalem]{ulem}

\usepackage[all]{xy}

\usepackage{tikz-cd}
\theoremstyle{plain}
\newtheorem{lemma}{Lemma}[section]
\newtheorem{theorem}[lemma]{Theorem}
\newtheorem{proposition}[lemma]{Proposition}
\newtheorem{corollary}[lemma]{Corollary}

\theoremstyle{definition}
\newtheorem{definition}[lemma]{Definition}
\newtheorem{example}[lemma]{Example}
\newtheorem{remark}[lemma]{Remark}

\theoremstyle{remark}

\newcommand{\PBVBG}[3]{{#2}^* #1}

\newcommand{\Ker}{\operatorname{Ker}}

\newcommand{\half}{\frac 1 2}

\def\Z{{\mathbb Z}}
\def\N{{\mathbb N}}

\def\g{\mathfrak{g}}

\def\Nu{{\cal V}}

\newcommand{\bea}{\begin{eqnarray}}
\newcommand{\eea}{\end{eqnarray}}

\def\BaseGroupoid{{\Gamma}}
\newcommand{\BaseAlgebroid}{A}

\def\toto{{\rightrightarrows}}
\newcommand{\gauge}[1]{\ensuremath{\overrightarrow{#1}-\overleftarrow{#1}}}

\newcommand{\longiso}{\stackrel{\textstyle\sim}{\longrightarrow}}

\newcommand{\gpoidmm}{\Gamma\toto M}
\newcommand{\gpoidmmi}{\Gamma_1\toto M_1}
\newcommand{\gpoidmmm}{\Gamma_2\toto M_2}

\newcommand{\quasip}{\big(\gpoidmm, \Pi, \Lambda  \big)}

\newcommand{\smalcirc}{\mbox{\tiny{$\circ $}}}
\newcommand{\pixx}{{\cal P}}
\newcommand{\zz}{\mathbb{Z}}
\newcommand{\frakg}{\mathfrak{g}}
\newcommand{\pro}{proj}

\newcommand{\vb}[4]{\xymatrix{
{#1} \ar@<-.5ex>[d]\ar@<.5ex>[d]\ar[r] &
{#3} \ar@<-.5ex>[d]\ar@<.5ex>[d]\\
{#2}\ar[r]^\phi & {#4} }}

\newcommand{\ec}{{\cal E}}

\DeclareMathAlphabet{\mathpzc}{OT1}{pzc}{m}{it}

\newcommand{\XX}{\mathfrak{X}}

\newcommand{\leftgla}{\mathfrak{A}}
\newcommand{\rightgla}{\mathfrak G}
\newcommand{\assogla}{\mathcal V}
\newcommand{\diff}{\mathrm d}
\newcommand{\LeftRight}{\mathrm d}
\newcommand{\hSec}[2]{\Sigma^{#1} (#2)}

\newcommand{\polymult}[2]{{\mathcal T}_{mult}^{#1}{#2}}
\newcommand{\super}{_{[1]}}

\newcommand{\twist}{T}
\newcommand{\dgla}{dgla}

\newcommand{\hordis}{\nabla}

\newcommand{\RRR}{\psi}

\newcommand{\superG}{graded groupoid}

\newcommand{\horlift}{{\lambda_{\nabla}}}
\newcommand{\projj}[1]{\bar{#1}}

\newcommand{\strict}{{$\zz$-graded  Lie $2$-algebra}}

\newcommand{\glaPoint}{$\zz$-graded Lie algebra.}
\newcommand{\Lieh}{\text{Lie}_2}
\newcommand{\Poiss}{{Pois}}

\newcommand{\proj}{\text{pr}}

\newcommand{\Groupoid}{\underline{Gr}}

\newcommand{\vbdiag}[4]{\xymatrix{ #1 \ar[d] \ar@<-.5ex>[r] \ar@<.5ex>[r] & #2 \ar[d] \\ #3 \ar@<-.5ex>[r] \ar@<.5ex>[r] & #4 }}

\newcommand{\id}{\text{id}}

\newcommand{\VB}[4]{\xymatrix{ #1 \ar[d] \ar@<-.5ex>[r] \ar@<.5ex>[r] & #2 \ar[d] \\ #3 \ar@<-.5ex>[r] \ar@<.5ex>[r] & #4 }}

\newcommand{\Hom}{\text{Hom}}
\let\Vec=\overrightarrow
\let\ceV=\overleftarrow

\begin{document}
\title{Shifted  Poisson Structures on Differentiable  Stacks}

\author{Francesco Bonechi \footnote{\small INFN Sezione di Firenze, email: francesco.bonechi@fi.infn.it}, Nicola Ciccoli \footnote{\small Dipartimento di Matematica e Informatica, Universit\`a di Perugia, email: nicola.ciccoli@unipg.it}, Camille Laurent-Gengoux \footnote{\small Institut Elie Cartan de Lorraine (IECL), UMR 7502, Universit\'e de Lorraine, Metz, email: camille.laurent-gengoux@univ-lorraine.fr}, Ping Xu \footnote{\small Department of Mathematics, Pennsylvania State University, email: ping@math.psu.edu} }

\date{}

\maketitle

\abstract{The purpose of this paper is to investigate  $(+1)$-shifted
Poisson structures in the context of differential geometry.
The  relevant notion is that of $(+1)$-shifted  Poisson structures on differentiable stacks.
More precisely,  we develop the notion of the Morita equivalence  of
 quasi-Poisson groupoids. Thus isomorphism classes of $(+1)$-shifted
 Poisson stacks correspond to Morita equivalence classes of 
quasi-Poisson groupoids. 
 In the process,   we carry out the following
program  which is of independent interest:

(1) We introduce   a $\zz$-graded Lie 2-algebra of polyvector fields on a
 given Lie groupoid and prove that its homotopy equivalence class is 
invariant under the Morita equivalence of Lie groupoids, and
 thus they can be considered to be  polyvector fields on the corresponding differentiable
stack $\XX$. It turns out that  $(+1)$-shifted  Poisson structures on $\XX$
correspond exactly to  elements of the Maurer-Cartan  moduli set 
of the corresponding dgla.

(2) We introduce  the notion of the
 tangent complex $T_\XX$  and the cotangent complex
$L_\XX$ of a differentiable stack $\XX$ in terms of any Lie groupoid
$\Gamma\toto M$ representing $\XX$. They correspond to   
a homotopy class of 2-term  
homotopy $\Gamma$-modules 
$A[1]\rightarrow TM$ and $T^\vee M\rightarrow A^\vee[-1]$, respectively.
Relying on the  tools  of theory of VB-groupoids including
homotopy and Morita equivalence of VB-groupoids,
we prove that a $(+1)$-shifted Poisson structure on a differentiable
 stack $\XX$ defines a morphism ${L_\XX}[1]\to {T_\XX}$.
}

\section{Introduction}
\newtheorem{introthm}{Theorem}
\newtheorem{intropro}[introthm]{Problem}
\newcommand{\ra}{\rangle}

Derived algebraic
 geometry for shifted symplectic structures
and Poisson structures on moduli spaces
  have proved to be 
 important in understanding several theories   including
 Donaldson--Thomas invariants \cite{PTVV}
and quantum field theory \cite{CG}. 
 The symplectic case was addressed first in \cite{PTVV} and later shifted
 Poisson structures were developed
 (see \cite{CPTVV, MeSa1, MeSa2, Pr1, Pr2, Saf}).

Although a  very powerful post-Grothendieck machinery has been developed
in the context of algebraic geometry to deal with both
derived and stacky   singularities, 
we believe it is valuable to develop a purely differential geometric
 approach to issues pertaining to symplectic and Poisson geometry that
 are specific to the $C^\infty$-context. 
Both derived and stacky  singularities occur 
 in problems in  classical  symplectic and Poisson geometry \cite{SL}. 
There, many existing tools from differential geometry can be used and results can be sharpened. 
In this paper, we will thus focus on $(+1)$-shifted  Poisson structures on differentiable
Artin 1-stacks (what we shall call  differentiable stacks for short). 

Classical Poisson manifolds and $(+1)$-shifted  Poisson stacks are
  different in nature. 
While classical Poisson manifolds arise as phase spaces of Hamiltonian
 systems in classical mechanics,
$(+1)$-shifted  Poisson stacks are abstract mathematical constructions
 capturing the symmetry of Hamiltonian systems possessing momentum maps. 
The word `momentum' denotes a quantity whose conservation under the time 
evolution of a physical system is related to some symmetry of the system. 
The $(+1)$-shifted  symplectic stack $[\frakg^*/G]$ was (perhaps)
 the first instance (albeit in a hidden form) of a $(+1)$-shifted  Poisson
 stack encountered in the study of Hamiltonian systems. 
It can be credited to Mikami--Weinstein \cite{MW} who showed that the usual Hamiltonian momentum map theory 
can in fact be reformulated as a symplectic action of the symplectic groupoid 
 $\frakg^*\rtimes G \toto \frakg^*$, which
 is indeed a presentation of the $(+1)$-shifted  symplectic stack $[\frakg^*/G]$.

In the late 1980's, Weinstein introduced the notion of Poisson groupoid
\cite{Wei:Poiss}
in order to unify Drinfeld's theory of Poisson groups \cite{Drinfeld} 
with the theory of symplectic groupoids \cite{Wei:symp}.
The introduction of Poisson groupoids has led to many new developments
 in Poisson geometry in the last three decades, 
in particular the theory of quasi-Poisson groupoids which was
 developed in \cite{IPLGX}.
Roughly speaking, a quasi-Poisson groupoid is a Lie groupoid endowed with
 a multiplicative bivector field 
whose Schouten bracket with itself is `homotopic to zero.' 
In the present paper, we adopt the viewpoint that quasi-Poisson groupoids
ought to be understood as \emph{$(+1)$-shifted  differentiable Poisson stacks},
 which we should introduce. 

It is well known that isomorphism classes of differentiable stacks
can be constructed as Morita equivalence classes of Lie 
groupoids \cite{BehrendXu}.
Hence it  is natural to define $(+1)$-shifted  differentiable Poisson stacks
as \emph{Morita equivalence classes of quasi-Poisson groupoids}. 

This immediately raises the following  problems: 
\begin{intropro}\label{London}
\begin{itemize}
\item What is Morita equivalence for quasi-Poisson groupoids?
\item Given a quasi-Poisson structure on a Lie groupoid, 
is it possible to transfer it to any  Morita equivalent Lie groupoid?
\end{itemize}
\end{intropro}

While the notion of Morita equivalence of Lie groupoids was easily
 extended to quasi-symplectic groupoids \cite{X}, 
it does not admit a straightforward extension to quasi-Poisson groupoids. 
Indeed, unlike differential forms, 
the pull back of polyvector fields  is not well defined.
To overcome this difficulty, we show that
  quasi-Poisson structures on a given Lie groupoid $\Gamma\toto M$ 
are  Maurer--Cartan elements of the dgla
 determined by a $\Z$-graded Lie 2-algebra 
$\hSec{\bullet}{\BaseAlgebroid}\stackrel{\diff}{\mapsto} \polymult{\bullet}{\BaseGroupoid}$
constructed in a canonical way from the groupoid $\Gamma\toto M$.
This construction is explained in Section \ref{section2} and the Appendix.


This re-characterization of quasi-Poisson structures is closely related
 to the following  question, which is of independent interest: 
\begin{intropro}\label{Rome}
 What  are polyvector fields on a differentiable stack,
 and how can we describe them efficiently?
\end{intropro}

Berwick-Evans and Lerman \cite{Berwick-Evans-Lerman} proved that, 
given a presentation of a differentiable stack $\XX$ by a Lie groupoid $\Gamma\toto M$, 
the vector fields on $\XX$ can be understood in terms of a Lie 2-algebra 
consisting of the multiplicative vector fields \cite{MX:3} on $\Gamma$ and 
the sections of the Lie algebroid $A$ associated with the Lie groupoid
 $\Gamma\toto M$. 
This Lie 2-algebra has appeared in a disguised form in \cite{Mehta}
(see \cite[Propositions 3.2.25 and  3.2.27]{Mehta}).

Inspired by \cite{Berwick-Evans-Lerman}, 
we associate a $\Z$-graded Lie 2-algebra 
$\hSec{\bullet}{\BaseAlgebroid}\xrightarrow{\diff}\polymult{\bullet}{\BaseGroupoid}$
of `polyvector fields' with every Lie groupoid $\Gamma\toto M$. 
Here $\polymult{\bullet}{\BaseGroupoid}$ denotes the space of multiplicative polyvector fields on $\BaseGroupoid$ 
and $\hSec{\bullet}{\BaseAlgebroid}$ denotes the space of sections of the exterior powers of the Lie algebroid $\BaseAlgebroid$.
We prove that the $\Z$-graded Lie 2-algebras
 associated in this way to Morita equivalent Lie groupoids 
are  homotopy equivalent. 
Consequently, we define the space of polyvector fields on
 a differentiable stack $\XX$ 
to be the homotopy equivalence class of the $\Z$-graded Lie 2-algebras
 associated with any Lie groupoid representing the  differentiable stack $\XX$. 
A $(+1)$-shifted  Poisson structure on a differentiable 
stack $\XX$ is then simply an element of the Maurer--Cartan moduli
 set of the dgla 
determined by the homotopy equivalence class of $\Z$-graded Lie 2-algebras corresponding to $\XX$.
The choice of a presentation of the stack $\XX$ by
 a Lie groupoid $\Gamma\toto M$ 
identifies the $(+1)$-shifted  Poisson structures on $\XX$ with gauge equivalence classes of quasi-Poisson structures on $\Gamma\toto M$. 
Such a  gauge equivalence class of quasi-Poisson structures can be passed 
along from one Lie groupoid to any other Morita equivalent Lie groupoid. 
We thus obtain a satisfying definition of 
the Morita equivalence of quasi-Poisson groupoids.

The second  goal of the paper is   to explore where  the degree 
shifting comes from for a quasi-Poisson groupoid,
and to introduce the rank of $(+1)$-shifted Poisson stacks
 and the meaning of  their non-degeneracy.
Our construction is certainly inspired by derived algebraic
 geometry, but is of a different nature. 
Recall that, in classical Poisson geometry, a Poisson structure $\pi$ on
 a smooth manifold $X$ 
determines a morphism $\pi^\sharp: T^\vee_X\to T_X$
from the cotangent bundle $T^\vee_X$ to the tangent bundle $T_X$.
One expects that  an analogue statement holds for $(+1)$-shifted 
 Poisson stacks. 
Before one can attempt to address this issue, one must first investigate the 
following questions.
\begin{intropro}
\label{Brussels}
What are the analogues of the tangent and cotangent bundles for differentiable stacks?
\end{intropro}

In (derived) algebraic geometry, the definition
of the  cotangent complex  requires enormous preparation work
 \cite{Iussi}. 
This seems neither practical nor necessary when dealing with differentiable stacks. 
Here we introduce  this notion in terms of presentations of the  
differentiable stack by Lie groupoids. 
The following short answer was suggested to us by Behrend
(private communication; see also \cite[Introduction]{Kai}): 
the \emph{tangent complex} $T_\XX$ of a differentiable
stack $\XX$ admitting a presentation by a Lie groupoid $\Gamma\toto M$ 
is  the \emph{homotopy equivalence class of the
 homotopy $\Gamma$-module $\rho: A[1]\to TM$}
\cite{AriasAbadCrainic, ELW, Gracia-Saz-Mehta},  
where $A$ designates once again the Lie algebroid of $\Gamma\toto M$ and $\rho$ denotes its anchor map. 
Its dual, \emph{the cotangent complex $L_\XX$ of $\XX$}, is the 
\emph{homotopy equivalence class of the homotopy $\Gamma$-module
 $\rho^\vee: T^\vee M \to A^\vee[-1]$}.
Homotopy $\Gamma$-modules were independently introduced by Gracia-Saz and Mehta,
 who called them ``flat superconnection" \cite{Gracia-Saz-Mehta0, Gracia-Saz-Mehta}, and by  Abad and Crainic \cite{AriasAbadCrainic},
 who  called them ``representations up to homotopy". Both of  them were
 inspired by the work of Evens, Lu, and Weinstein \cite{ELW}.

Of course, one must  justify  that the  tangent  complex and
the  cotangent complex are well defined by investigating 
the  relation between the homotopy $\Gamma$-modules arising from different
 presentations of the differentiable stack $\XX$, i.e.\ 
different Morita equivalent  Lie groupoids $\Gamma\toto M$. 


Homotopy $\Gamma$-modules have been studied extensively in the literature.
%
In their pioneering work \cite{Gracia-Saz-Mehta}, Gracia-Saz and Mehta
established a dictionary between VB-groupoids over a fixed Lie groupoid $\Gamma\toto M$, 
and 2-term homotopy $\Gamma$-modules.
Here we enrich the dictionary by investigating Morita equivalence.
Note that Morita equivalence of VB-groupoids has been studied
in \cite{OrtizDelHoyo}. In this paper, however, we
will take a different approach more relevant to our situation,
and we will relate Morita and homotopy equivalence of VB-groupoids and investigate how this reflects on maps between them. 
 VB-groupoids $V_1\toto E_1$ and $V_2\toto E_2$ over $\Gamma_1\toto M_1$ 
and $\Gamma_2\toto M_2$, respectively, 
are Morita equivalent if and only if there exists  a
$\Gamma_1$-$\Gamma_2$-bitorsor
 $M_1\stackrel{\varphi_1}{\leftarrow} X \stackrel{\varphi_1}{\rightarrow} M_2$
such that the pullback VB-groupoids $V_1[\varphi^*_1 E_1]$ and $V_2[\varphi_2^* E_2]$ are homotopy equivalent.
Making use of the dictionary of Gracia-Saz and Mehta \cite{Gracia-Saz-Mehta},
this definition can be transposed to homotopy $\Gamma$-modules:
a homotopy $\Gamma_1$-module $\ec_1$ is Morita equivalent to a
 homotopy $\Gamma_2$-module $\ec_2$
if and only if there exists a $\Gamma_1$-$\Gamma_2$-bitorsor 
 $M_1\stackrel{\varphi_1}{\leftarrow} X \stackrel{\varphi_1}{\rightarrow} M_2$
and a homotopy equivalence of homotopy
 $\Gamma_1[X] (\cong \Gamma_2[X])$-modules
from $\ec_1[X]$ to $\ec_2[X]$.

If $\Gamma_1\toto M_1$ and $\Gamma_2 \toto M_2$ are Morita equivalent
 Lie groupoids, 
then $T\Gamma_1\toto TM_1$ and $T\Gamma_2\toto TM_2$ are Morita equivalent VB-groupoids 
and, similarly, $T^\vee\Gamma_1\toto A_1^\vee$ and $T^\vee\Gamma_2\toto A_2^\vee$ are Morita equivalent VB-groupoids.
It immediately follows that the homotopy $\Gamma_1$-module $A_1[1]\to TM_1$
is Morita equivalent to the homotopy $\Gamma_2$-module $A_2[1]\to TM_2$,
while the homotopy $\Gamma_1$-module $T^\vee M_1 \to A^\vee _1[-1]$ 
is Morita equivalent to the homotopy $\Gamma_2$-module  $T^\vee M_2 \to A^\vee_2[-1]$. 
This justifies  our definition of the tangent and cotangent complexes $T_\XX$ and $L_\XX$ of a differentiable stack $\XX$. 


Given a  quasi-Poisson groupoid $(\Gamma,\Pi,\Lambda)$, the associated map 
$\Pi^\sharp: T^\vee\Gamma\rightarrow T\Gamma$ is a VB-groupoid morphism. 
Moreover, if $(\Gamma_1,\Pi_1,\Lambda_1)$ and $(\Gamma_2,\Pi_2,\Lambda_2)$ are
Morita equivalent quasi-Poisson groupoids, then
the associated VB-groupoid morphisms $\Pi_1^\sharp: T^\vee\Gamma_1\rightarrow T\Gamma_1$ 
and $\Pi^\sharp_2:T^\vee\Gamma_2\rightarrow T\Gamma_2$ are equivalent as generalized VB-groupoid morphisms.
As an immediate consequence, we prove that a $(+1)$-shifted Poisson structure on a differentiable stack $\XX$  indeed
determines a morphism $\Pi^\sharp: {L_\XX}[1]\to {T_\XX}$ of $2$-term 
complexes from the $(+1)$-shifted cotangent complex to the tangent complex.
This, in turn, allows us to introduce the rank of a $(+1)$-shifted  Poisson stack $\XX$ as an integer-valued map 
defined on its coarse moduli space $|\XX|$ \cite{BehrendXu}.
We are thus led to a natural definition of non-degenerate $(+1)$-shifted
   Poisson stacks.


We conclude this introduction with a few remarks.
As was proved in \cite{X}, quasi-symplectic structures on a
Lie groupoid transfer to Morita equivalent Lie groupoids.
Therefore one can speak quasi-symplectic structures
on a  Morita equivalent class of  Lie groupoids.
  This readily provides a notion of  $(+1)$-shifted
 symplectic structures on a differentiable stack.
 It is natural to expect that
non-degenerate  $(+1)$-shifted  Poisson stacks   are isomorphic  to
$(+1)$-shifted  symplectic stacks.
Indeed, in a forthcoming paper \cite{BCLX2}, 
we establish an explicit one-one correspondence between
non-degenerate  $(+1)$-shifted  Poisson stacks
and  $(+1)$-shifted  symplectic stacks,
 and study its application  to the momentum map theory.
\color{black} 
In particular, we prove that the momentum map theory of quasi-Poisson 
groupoids in~\cite{IPLGX} is stacky in nature 
and that Hamiltonian reductions can be carried out, which agrees with the 
derived symplectic  geometry principle that 
the derived intersection of coisotropics of a $(+1)$-shifted  Poisson stack
 gives rise to a Poisson structure \cite{ToenICM}. 
It also enables us to merge the quasi-Hamiltonian momentum map theory of Alekseev-Malkin-Meinrenken \cite{AMM} 
with the quasi-Poisson theory of Alekseev, Kosmann-Schwarzbach and Meinrenken \cite{AKKS, AKM}.
Finally, we refer the reader to \cite{Pr2} for an  explanation
of  the  relationship
 between various concepts introduced in the present
paper and those  in the algebraic geometry setting   \cite{CPTVV, Pr1}.
See also  Remark \ref{rk:1}, Remark \ref{rk:2}, Remark \ref{rk:3}, and
Remark \ref{rk:4}.
In  the second version  of \cite{Saf},  the author claims
 that  $(+1)$-shifted Poisson structures introduced
in \cite{CPTVV} 
  are  equivalent to  $(+1)$-shifted Poisson structures in our sense
(see Theorem 3.29 in  version 2 of \cite{Saf}).
 However,  note that
the author  considers source-connected groupoids in the context of smooth affine
 schemes, while we deal with any  $C^\infty$-groupoids.

One of the authors announced some of the results set forth
 in the present paper at the conference 
\emph{Derived algebraic geometry with a focus on derived symplectic techniques} held at the University of Warwick in April 2015.
He wishes to thank the organizers for providing him the opportunity to disseminate the results of our work.

\subsection*{Acknowledgements}

We would like to thank Ruggero Bandiera, Kai Behrend, Domenico Fiorenza,
Yvette Kosmann-Schwarzbach, Jonathan  Pridham, Mathieu Sti\'enon,
  Gabriele Vezzosi for fruitful discussions and useful comments.
We are grateful to the anonymous referees for useful comments and suggestions.
 The research project was, at different stages, hosted by INFN-Sezione di Firenze, Department of Mathematics, Pennsylvania State University and Dipartimento di Matematica e Informatica, Universit\`a di Perugia. We thank all the involved institutions for providing support and excellent working conditions.
 
 N.C. is partially supported by Indam-GNSAGA, PRIN \emph{Variet\`a reali e complesse: geometria, topologia e analisi armonica} and Research project \emph{Geometry of quantization}- Universit\`a di Perugia. 
P.X. is  partially supported by  NSF grants DMS-1406668 and DMS-1707545.

\section{Polyvector fields on a differentiable stack}

\label{section2}

The purpose of  this section  is to introduce the notion of
 polyvector fields on a differentiable stack. 
 They can be represented by a dgla
  when the differentiable stack is represented by a Lie groupoid.
 Different Lie groupoids representing the isomorphic stack  give rise  to
homotopy equivalent dglas. In fact, more precisely,
they are represented by  homotopy equivalence classes of {\strict}s.

\subsection{The {\strict} of polyvector fields on a Lie groupoid}

We first recall a few basic facts 
concerning  multiplicative polyvector fields 
on a Lie groupoid~\cite{IPLGX}.

Let $\BaseGroupoid \toto M$ be a Lie groupoid with
the source map $s$ and the target map  $t$,  respectively. Let 
 $\BaseAlgebroid$ be its Lie algebroid, with anchor map $\rho:A\to TM$. 
Denote the graph of the multiplication  by
$$
{\rm graph}(\BaseGroupoid) = \{(g_1,g_2,g_1 g_2)|\ g_1,g_2\in \BaseGroupoid,\
 s(g_1)=t(g_2)\} \subset \BaseGroupoid\times \BaseGroupoid\times \BaseGroupoid.
$$ 
We say that a $k$-vector field $P\in\Gamma(\wedge^k T\BaseGroupoid)$
on $\BaseGroupoid$  is {\it multiplicative}
if ${\rm graph}(\BaseGroupoid)$ is coisotropic with respect to 
$P\oplus P\oplus (-1)^{k+1}P$. 
We denote by $\polymult{k}{\BaseGroupoid}$, for any $k\geq 0$,
 the space of multiplicative 
$(k+1)$-vector fields on $\BaseGroupoid$, 
and  by  $\polymult{\bullet}{\BaseGroupoid} = \bigoplus_{k \geq -1} \polymult{k}{\BaseGroupoid}$   the $\zz$-graded vector space of multiplicative polyvector
 fields on  $\BaseGroupoid$.

For any $k \geq -1$, we denote by $\hSec{k}{\BaseAlgebroid}$ the
 space $\Gamma(\wedge^{k+1} \BaseAlgebroid)$ of sections of
the exterior vector bundle  $\wedge^{k+1} \BaseAlgebroid \to M$.
When equipped with the Schouten-Nijenhuis bracket, $\hSec{\bullet}{\BaseAlgebroid} := \bigoplus_{k \geq -1} \hSec{k}{\BaseAlgebroid} $ is a {\glaPoint}

For every $a\in\hSec{k}{\BaseAlgebroid}$, we denote the corresponding right 
and left invariant $(k+1)$-vector fields by $\overrightarrow{a}$
 and $\overleftarrow{a}$,
respectively. 

 It is easy to check that $\gauge{a}$
is a multiplicative $(k+1)$-vector field \cite{IPLGX}, called an {\it exact multiplicative $(k+1)$-vector field}.
We recall from \cite{IPLGX}  some 
well-known facts concerning multiplicative polyvector fields.

\smallskip
\begin{lemma}
\label{lem:rightAndleft}
\begin{enumerate}
 \item[(i)] The space of multiplicative polyvector fields
 $\polymult{\bullet}\BaseGroupoid$ on $\BaseGroupoid\toto M$
is closed under the Schouten-Nijenhuis bracket, and is therefore a {\glaPoint}
\item[(ii)] The map ${\LeftRight}:\hSec{\bullet}{\BaseAlgebroid}\rightarrow\polymult{\bullet}{\BaseGroupoid}$, 
\begin{equation}\label{twist}
{\LeftRight}(a)=\gauge{a}
 \end{equation}
is an homomorphism of $\zz$-graded Lie algebras.
\item[(iii)] For each $P\in\polymult{k}{\BaseGroupoid}$ and $a\in\hSec{l}{\BaseAlgebroid}$, 
there exists an unique section $\delta_P (a) \in \hSec{k+l}{\BaseAlgebroid}$ such that
$$
\overrightarrow{\delta_P (a)} = [P,\overrightarrow{a}]\;.
$$
Moreover, the correspondence $P\mapsto \delta_P$ is
 a $\zz$-graded Lie algebra morphism
from  $\polymult{\bullet}{\BaseGroupoid}$ to $\mathrm{Der}^\bullet(\Sigma(A))$.
The action satisfies the following properties:
\begin{itemize}
 \item[$1$)]  ${\LeftRight}(\delta_P (a))= \, [P,{\LeftRight}(a)]$; 
 \item[$2$)] $\delta_{{\LeftRight}(a)}(b) = [a,b]$,
\end{itemize}
for any  $P\in\polymult{\bullet}{\BaseGroupoid}$ and $a,b\in\hSec{\bullet}{\BaseAlgebroid}$.
\end{enumerate}
\end{lemma}

The following proposition  follows  immediately.

\begin{proposition}
\label{prop:strict}
 Let $\BaseGroupoid\toto M$ be a Lie groupoid. Then $\hSec{\bullet}{\BaseAlgebroid}\stackrel{\diff}{\mapsto} \polymult{\bullet}{\BaseGroupoid}$ together with the Lie brackets 
and actions described in Lemma \ref{lem:rightAndleft}  is a {\strict}.
\end{proposition}

In other words,  
$\hSec{\bullet}{\BaseAlgebroid}\stackrel{\diff}{\mapsto} 
\polymult{\bullet}{\BaseGroupoid}$
is a crossed module of $\zz$-graded Lie algebras.  There is an associated
   dgla 
$\Nu^\bullet (\BaseGroupoid )  := \bigoplus_{k\ge -2}\Nu^k (\BaseGroupoid )$,
where
  $$ \Nu^k (\BaseGroupoid ) =    \hSec{k+1}{\BaseAlgebroid}\oplus \polymult{k}{\BaseGroupoid} .$$
See Appendix \ref{AppendixA1} for details. 
The  associated {\dgla}  $\Nu^\bullet (\BaseGroupoid )$
 is  called the \emph{{\dgla} of polyvector fields on the Lie groupoid
 $\BaseGroupoid \toto M$}.

\begin{remark}
\label{rk:florence}
Recall that, for any {\dgla}  ${\Nu}^\bullet$, its  cohomology
 $H^\bullet(\Nu )$ is a $\zz$-graded Lie algebra. It is easy to
see that
$$H^k(\Nu (\BaseGroupoid)) \cong 
\hSec{k+1}{\BaseAlgebroid}^\BaseGroupoid 
\oplus\frac{\polymult{k}\BaseGroupoid}{\{\gauge{a}|a\in \hSec{k}{\BaseAlgebroid}\} },$$
where $\hSec{k+1}{\BaseAlgebroid}^\BaseGroupoid$ denotes 
the space of $\Gamma$-invariant sections of $\wedge^{k+1} \BaseAlgebroid$.
\end{remark}

\subsection{Morita equivalence}

In this section, we discuss how 
the  ${\mathbb Z}$-graded   2-term 
complex\footnote{A ${\mathbb Z}$-graded 2-term complex is a 2-term 
complex in the category of ${\mathbb Z}$-graded vector spaces.
More explicitly, a ${\mathbb Z}$-graded   2-term complex consists of
$\zz$-graded vector spaces $A$ and $B$  and a graded linear map
$d: A \to B$ of degree \emph{zero} (with respect to the gradings of $A$
 and $B$). 
A ${\mathbb Z}$-graded 2-term complex morphism
from  $A \stackrel{d}{\mapsto} B$ to  $A' \stackrel{d'}{\mapsto} B'$
 is a pair of chain maps $A \mapsto A'$ and $B \mapsto B'$ 
of degree $0$.
Homotopies between morphisms are  usual homotopy maps 
  $B\to A'$, which are again assumed to be of degree $0$.}
$\hSec{\bullet}{\BaseAlgebroid}\stackrel{\diff}{\mapsto} \polymult{\bullet}{\BaseGroupoid}$ changes  
under Morita equivalence of Lie groupoids.
Note that  if $\BaseGroupoid_i \toto M_i$, $i=1,2$, are two Lie groupoids with respective Lie algebroids 
$\BaseAlgebroid_i$, and $\phi: \BaseGroupoid_1 \to \BaseGroupoid_2 $ is a
 Lie groupoid morphism over $\varphi : M_1 \to M_2$, in general,
 there is \emph{no} natural  chain map
from  $\hSec{\bullet}{\BaseAlgebroid_1}\stackrel{\diff_1}{\mapsto} \polymult{\bullet}{\BaseGroupoid_1}$
 to  $\hSec{\bullet}{\BaseAlgebroid_2}\stackrel{\diff_2}{\mapsto} \polymult{\bullet}{\BaseGroupoid_2}$.
However, we will prove that when $\phi$ is a Morita morphism of Lie groupoids,
these ${\mathbb Z}$-graded 2-term  complexes
are indeed homotopy equivalent.

Assume that $\BaseGroupoid[X]\toto X$
 is the pull-back groupoid of the Lie groupoid $\BaseGroupoid\toto M$
under a  surjective submersion $\varphi:X \to M$, where
 $\BaseGroupoid[X] =X\times_{M, t}\Gamma\times_{s, M} X$.		 
Let $\phi :  \BaseGroupoid[X] \to \BaseGroupoid $ be the natural projection,
which is a Morita morphism.
 Let $\phi_A:A[X]\to A$ be the corresponding Lie algebroid morphism. 
By 
$\hSec{\bullet}{\BaseAlgebroid} \stackrel{\diff}{\mapsto}
\polymult{\bullet}{\BaseGroupoid}$
and 
$\hSec{\bullet}{\BaseAlgebroid[X]} \stackrel{\diff'}{\mapsto}
 \polymult{\bullet}{\BaseGroupoid[X]}$,
we denote   the $\zz$-graded Lie 2-algebras as in
Proposition \ref{prop:strict}
associated to the Lie groupoids $\BaseGroupoid\toto M$ and
$\BaseGroupoid[X]\toto X$, respectively.
 As in \cite{OrtizWaldron},   we  consider the following spaces:

\begin{enumerate}
 \item[(i)] By $\polymult{\bullet}{\BaseGroupoid[X]}_{\pro}$,
we denote the subspace of $\polymult{\bullet}{\BaseGroupoid[X]}$ 
consisting of \emph{projectable multiplicative polyvector fields on
 $\BaseGroupoid[X] $},
 namely those   $P\in  \polymult{\bullet}{\BaseGroupoid[X]}$ 
  such that 
there exists $\projj{P} \in  \polymult{\bullet}{\BaseGroupoid} $ 
satisfying
 $ \phi_* (P)  = \projj{P}$.
 \item[(ii)]  By $\hSec{\bullet}{\BaseAlgebroid[X]}_{\pro}$, we  denote
 the subspace  of $\hSec{\bullet}{\BaseAlgebroid[X]}$
consisting  of \emph{projectable sections} in  $\Gamma (X; \wedge^{\bullet+1}
 \BaseAlgebroid[X] )$,
 namely those sections  $a \in \hSec{\bullet}{\BaseAlgebroid[X]}$
 such that there exists $\projj{a} \in  \hSec{\bullet}{\BaseAlgebroid}$
satisfying $\phi_A (a)  = \projj{a}$.
\end{enumerate}

There are projection maps:
\begin{eqnarray}
&\proj: &\polymult{\bullet}{\BaseGroupoid[X]}_{\pro}\to \polymult{\bullet}{\BaseGroupoid}, \ \ P\mapsto \projj{P}\, , \label{eq:Paris1}\\
&\proj: &\hSec{\bullet}{\BaseAlgebroid[X]}_{\pro} \to \hSec{\bullet}{\BaseAlgebroid},
\ \ a\mapsto \projj{a}\, \label{eq:Paris2}.
\end{eqnarray}

\begin{proposition}
\label{prop:strict2}
Assume that $\BaseGroupoid\toto M$
is a Lie groupoid, $\varphi: X \to M$
a  surjective submersion.
Let $\phi: \BaseGroupoid[X] \to \BaseGroupoid $ be the corresponding
Morita   morphism of Lie groupoids. Then
\begin{enumerate}
 \item[(i)] 
$\hSec{\bullet}{\BaseAlgebroid[X]}_{\pro} \stackrel{\diff'}{\mapsto} \polymult{\bullet}{\BaseGroupoid[X]}_{\pro}$ 
is a $\zz$-graded Lie 2-subalgebra of
 $\hSec{\bullet}{\BaseAlgebroid[X]} \stackrel{\diff'}{\mapsto}
 \polymult{\bullet}{\BaseGroupoid[X]}$; 
 \item[(ii)]  the projection  $\proj$ in Equations \eqref{eq:Paris1}-\eqref{eq:Paris2}
 is a   morphism of $\zz$-graded Lie 2-algebras
 from 
$\hSec{\bullet}{\BaseAlgebroid[X]}_{\pro} \stackrel{\diff'}{\mapsto} \polymult{\bullet}{\BaseGroupoid[X]}_{\pro}$ 
to $\hSec{\bullet}{\BaseAlgebroid} \stackrel{\diff}{\mapsto} 
\polymult{\bullet}{\BaseGroupoid}$. 
\end{enumerate}
\end{proposition} 
Proposition \ref{prop:strict2} means that both horizontal maps
 in the diagram below  are  morphisms of $\zz$-graded Lie 2-algebras, 
where ${\mathfrak i}$ stands for the inclusion maps:

\begin{equation}
 \label{eq:purpose0}
 \xymatrix{  \polymult{\bullet}{\BaseGroupoid} & \polymult{\bullet}{\BaseGroupoid[X]}_{\pro} &  \polymult{\bullet}{\BaseGroupoid[X]} \\ 
& \ar@{}[r]^(.25){}="a"^(.75){}="b" \ar@{^{(}->}^{\mathfrak i} "a";"b"   \ar@{}[l]^(.25){}="a"^(.75){}="b" \ar@{->}_{\proj} "a";"b"  & \\ 
\ar[uu]^{\diff} \hSec{\bullet}{\BaseAlgebroid} & \hSec{\bullet}{\BaseAlgebroid[X]}_{\pro} \ar[uu]^{\diff'}& \hSec{\bullet}{\BaseAlgebroid [X]} \ar[uu]^{\diff'} } 
\end{equation}

We  now define horizontal lifts.
By  an \emph{Ehresmann connection} $\hordis$  for a surjective
 submersion $\varphi : X \to M$, we mean   a subbundle 
$H^\nabla\subset TX$ such that $TX \cong H^\nabla	\oplus \ker(T\varphi )$
as vector bundles over $X$. An Ehresmann connection $\hordis$ induces
an injective  map of vector bundles, denoted by the same symbol,
$\nabla:\varphi^*TM\hookrightarrow  TX$.
The subbundle $H^\nabla\subseteq TX$ is also called an \emph{horizontal lift}.

For any $x,y \in X$ and $\gamma \in \BaseGroupoid$ with $\varphi(x) = t(\gamma)$ and $\varphi(y)=s(\gamma)$, 
the connection $\hordis$ induces a pair of natural injections:
\bea
&&\label{eq:Jusseu1}
 T_{\gamma} \Gamma  \hookrightarrow  T_{(x,\gamma,y)} (\BaseGroupoid[X]),
 \hbox{ and } \\
&&\label{eq:Jusseu2} 
\BaseAlgebroid_{\varphi(x)} \hookrightarrow \BaseAlgebroid[X]_{x} 
\eea

The map \eqref{eq:Jusseu1} is defined as follows:
\bea
\label{eq:horlift00}
  T_{\gamma} \Gamma  &&\to   (T_xX) 
\times_{(T_{\varphi(x)} M)}
  (T_{\gamma} \BaseGroupoid)  \times_{(T_{\varphi(y)} M)}  (T_y X) 
{\simeq} T_{(x,\gamma,y)} (\BaseGroupoid[X])  \\
  u &&\to \big( (\nabla \smalcirc t_{T\Gamma}) (u) , \ u,  \ 
(\nabla \smalcirc s_{T\Gamma}) (u) \big),  \nonumber
\eea
 where $s_{T\Gamma}(u)\in T_{s(\gamma)}M\hookrightarrow \varphi^*(TM)_y$ and $t_{T\Gamma}(u)\in T_{t(\gamma)}M\hookrightarrow \varphi^*(TM)_x$ 
are, respectively,  the
 source map and the target map of the tangent groupoid $T\Gamma\toto TM$.
The map \eqref{eq:Jusseu2} is  defined by
\bea
\label{eq:horlift01}
  A_{\varphi(x)} && \to (A_{\varphi(x)}) \times_{(T_{\varphi(x)}M)} (T_x X)   {\simeq} \BaseAlgebroid [X]_x  \\
  a &&\to \big(a , \  (\nabla \smalcirc \rho ) (a) \big). \nonumber
 \eea

By dualizing the maps (\ref{eq:horlift00}-\ref{eq:horlift01}),
 we obtain a pair of vector bundle morphisms:

\begin{equation}
\label{eq:aa}
\begin{tikzcd}
{T^\vee \BaseGroupoid[X]}  \arrow{d}  \arrow{r}{\Phi_\hordis} &
{T^\vee\BaseGroupoid} \arrow{d}  \\
{X} \arrow{r}{\varphi} & {M}
\end{tikzcd}
\end{equation}
and
\begin{equation}
\label{eq:aa'}
\begin{tikzcd}
{\BaseAlgebroid[X]^\vee} \arrow{d}  \arrow{r}{\phi_\hordis} &
{\BaseAlgebroid^\vee} \arrow{d} \\
{X} \arrow{r}{\varphi} & {M}
\end{tikzcd}
\end{equation}

These morphisms extend to exterior product bundles, and give rise to
 a pair of maps on the sections of their dual bundles, called
\emph{horizontal lifts} by abuse of notations:
\bea
\label{eq:hordis000}
&& \lambda_\hordis : \Gamma(\wedge T\BaseGroupoid ) \to \Gamma(\wedge T\BaseGroupoid[X] )  \hbox{ and } \\
&& \lambda_\hordis : \Gamma(\wedge \BaseAlgebroid ) \to   \Gamma(\wedge \BaseAlgebroid[X] ) \nonumber.
\eea  
Note that
$\polymult{\bullet}{\BaseGroupoid} \to
  \hSec{\bullet}{\BaseAlgebroid},
  \polymult{\bullet}{\BaseGroupoid[X]}_{\pro} \to
  \hSec{\bullet}{\BaseAlgebroid[X]}_{\pro}$, and
  $\polymult{\bullet}{\BaseGroupoid[X]} \to
  \hSec{\bullet}{\BaseAlgebroid[X]}$ are 
$\zz$-graded 2-term complexes.
 By forgetting, for the moment, their $\zz$-graded  Lie brackets, 
we have the following proposition, whose proof is
postponed to Appendix   \ref{App:B2}.

\begin{proposition}
\label{cor:superGroupoidsConclusion}
Let  $\BaseGroupoid \toto M$  be a Lie groupoid,
 and $\varphi : X \to M$ a surjective submersion. Choose  an Ehresmann connection $\hordis$ for
$\varphi$. Then
\begin{enumerate}
 \item[(i)]  the chain map $\proj $ is a left inverse of $\horlift$,
 and, moreover, there exists a  chain homotopy 
 $ h_{\lambda_\hordis}: \polymult{\bullet}{\BaseGroupoid[X]}_{\pro} \to
  \hSec{\bullet}{\BaseAlgebroid[X]}_{\pro}$ 
 between $ \horlift \circ \proj $ and the identity map:
$$ \xymatrix{  \polymult{\bullet}{\BaseGroupoid}   & 
\polymult{\bullet}{\BaseGroupoid[X]}_{\pro} \ar@/^/@[blue][dd]^{h_\horlift} 
 \\ 
\ar@{}[r]^(.25){}="a"^(.75){}="b" \ar@<4pt>@[red]@{^{(}->}^{\horlift} "a";"b"  & 
 \ar@{}[l]^(.25){}="a"^(.75){}="b" \ar@<4pt>@[red]@{_{}->}^{\proj} "a";"b"\\
\ar[uu]^{\diff} \hSec{\bullet}{\BaseAlgebroid}   & 
\hSec{\bullet}{\BaseAlgebroid[X]}_{\pro} \ar[uu]^{\diff'} } $$ 
 
 \item[(ii)] there exists a chain map $ \RRR$  and an homotopy $h_X: \polymult{\bullet}{\BaseGroupoid[X]}\rightarrow\hSec{\bullet}{\BaseAlgebroid[X]}$, 
 
 $$ \xymatrix{  \ar@/_1.0pc/@[blue][dd]_{h_{\lambda_\nabla}}  
\polymult{\bullet}{\BaseGroupoid[X]}_{\pro}  & \polymult{\bullet}{\BaseGroupoid[X]} 
 \ar@/^/@[blue][dd]^{h_X} 
 \\ 
\ar@{}[r]^(.25){}="a"^(.75){}="b" \ar@<4pt>@[red]@{^{(}->}^{{\mathfrak i}} "a";"b"  & 
 \ar@{}[l]^(.25){}="a"^(.75){}="b" \ar@<4pt>@[red]@{_{}->}^{\RRR} "a";"b"\\
\ar[uu]^{\diff'} \hSec{\bullet}{\BaseAlgebroid[X]}_{\pro}  & \hSec{\bullet}{\BaseAlgebroid[X]} \ar[uu]^{\diff'}, } $$ 
 such that  both $ \RRR \circ {\mathfrak i}$ 
and $  {\mathfrak i} \circ \RRR$ are homotopic 
 to the identity as chain maps.
\end{enumerate}  
\end{proposition}

\begin{remark}
In Proposition \ref{cor:superGroupoidsConclusion}, 
the  maps  $\RRR$, $h_X$ and $h_{\lambda_\nabla}$
can be described explicitly
 in terms of geometric data such as
 the connection $\nabla$ on $\varphi : X \to M$,
a partition of unity with respect to an open cover
 $(U_i)_{i \in I}$ of $M$,
 and  local sections $\sigma_i : U_i \to X$ of $\varphi$.
 Explicit formulas
can be derived from Equations (\ref{eq:I1}--\ref{eq:HX}).
\end{remark}


\subsection{Polyvector fields on a differentiable stack}

\label{sec:polyvDiffStack}

Let $\BaseGroupoid \toto M$ be a Lie groupoid and $\varphi: X \to M$
 a surjective submersion.
Let $\proj$ and ${\mathfrak i}$ be the  morphisms of $\zz$-graded 2-term 
complexes 
 as in Equation (\ref{eq:purpose0}).

Choose  an Ehresmann connection $\hordis$ for
$\varphi : X \to M$. 
According to  Proposition \ref{cor:superGroupoidsConclusion} (i),
 the horizontal lift $\horlift$ is an  homotopy inverse 
of $\proj$.
According to  Proposition \ref{cor:superGroupoidsConclusion} (ii),
there  exists a retraction $\RRR$
which is a homotopy inverse  of ${\mathfrak i}$. We summarize
all chain maps in the diagram below, where
all morphisms of graded $2$-term complexes
 pointing on the left are homotopy inverses of those pointing on the right:
 \begin{equation}
 \label{eq:purpose}
 \xymatrix{  \polymult{\bullet}{\BaseGroupoid} & \polymult{\bullet}{\BaseGroupoid[X]}_{\pro} \ar@/^/@[blue][dd]^{h_\horlift} &  \polymult{\bullet}{\BaseGroupoid[X]}  \ar@/^/@[blue][dd]^{h_X}\\ 
\ar@{}[r]^(.25){}="a"^(.75){}="b" \ar@<4pt>@[red]@{^{(}->}^{\horlift} "a";"b" & 
\ar@{}[r]^(.25){}="a"^(.75){}="b" \ar@<4pt>@[red]@{^{(}->}^{\mathfrak i} "a";"b"   
\ar@{}[l]^(.25){}="a"^(.75){}="b" \ar@<4pt>@[red]@{->}^{\proj} "a";"b"  & 
\ar@{}[l]^(.25){}="a"^(.75){}="b" \ar@<4pt>@[red]@{->}^{\RRR} "a";"b"   \\ 
\ar[uu]^{\diff} \hSec{\bullet}{\BaseAlgebroid} & \hSec{\bullet}{\BaseAlgebroid[X]}_{\pro} \ar[uu]^{\diff'}& \hSec{\bullet}{\BaseAlgebroid [X]} \ar[uu]^{\diff'} } .
\end{equation}

In addition to being morphisms of $\zz$-graded 2-term complexes, 
both $\proj$ and ${\mathfrak i}$ are strict morphisms of $\zz$-graded
 Lie $2$-algebras.
 However, in general,
  neither  $ \horlift$ nor $\RRR$ is a  strict morphism of
  $\zz$-graded Lie 2-algebras. Nevertheless, we have the following

\begin{proposition}
\label{th:linfty_morphism}
Let $\BaseGroupoid \toto M$ be a Lie groupoid, and
 $\varphi: X \to M$ a surjective submersion.
Choose an Ehresmann connection $\nabla$ for $\varphi$. Then,
\begin{enumerate}
\item[(i)] the morphism  of $\zz$-graded Lie 2-algebras $\proj$ in Equation
\eqref{eq:purpose}
  admits an  homotopy inverse, whose linear part is the horizontal
 lift $\horlift$ and whose quadratic part depends only on $\hordis$ and $h_{\horlift}$;
 \item[(ii)] the morphism  of $\zz$-graded Lie 2-algebras
 $ {\mathfrak i}$  in Equation \eqref{eq:purpose} admits an homotopy 
 inverse,
 whose linear part is the retraction 
$\RRR$ and  whose quadratic part depends only on $\RRR, h_X$ and $h_{\lambda_\nabla}$.
\end{enumerate}
\end{proposition}
\begin{proof}
To prove (i), we  apply Theorem \ref{th:invertingStrictLie2Morphisms} to
the morphisms  in Proposition \ref{cor:superGroupoidsConclusion} (i).
Recall the notations of  Theorem \ref{th:invertingStrictLie2Morphisms}: 
\begin{equation}
 \xymatrix{  \ar@[blue]@/_{16pt}/[dd]_{h} \rightgla & \rightgla' \ar@[blue]@/^{16pt}/[dd]^{h'}  \\
    \ar@{}[r]^(.25){}="a"^(.75){}="b" \ar@[red]@<4pt>@{->}^{\Phi_1} "a";"b"  &    \ar@{}[l]^(.25){}="a"^(.75){}="b" \ar@[red]@<4pt>@{->}^{\Psi_1} "a";"b"   \\
 \leftgla  \ar[uu]^{\diff}  & \leftgla'  \ar[uu]_{\diff'} 
}  
\end{equation}
Here we   take the following data:
 (1) $\leftgla' \stackrel{\diff'}{\to} \rightgla'$ is $\hSec{\bullet}{\BaseAlgebroid[X]}_{\pro} \stackrel{\diff'}{\mapsto}
 \polymult{\bullet}{\BaseGroupoid[X]}_{\pro}$;
 (2) $\leftgla \stackrel{\diff}{\to} \rightgla$ is $\hSec{\bullet}{\BaseAlgebroid}\stackrel{\diff}{\mapsto} \polymult{\bullet}{\BaseGroupoid}$;
 (3) $ \Psi_1$ is the projection $\proj$;
 (4) $\Phi_1$ is the horizontal lift $\horlift$;
 (5) $h=0$;
 and (6)  $h'$ is the homotopy 
$h_{\horlift}: \polymult{\bullet}{\BaseGroupoid[X]}_{\pro} \to \hSec{\bullet}{\BaseAlgebroid[X]}_{\pro}$ 
as in  Proposition \ref{cor:superGroupoidsConclusion} (i).
It is easy to check that all
 conditions in Theorem \ref{th:invertingStrictLie2Morphisms}
 are  satisfied, and therefore assertion (i) is proved.

Similarly, assertion  (ii) is proved  by
applying Theorem \ref{th:invertingStrictLie2Morphisms} 
to the maps appearing as in  Proposition \ref{cor:superGroupoidsConclusion} (ii).
\end{proof}

It follows from the previous proposition that the 
$\zz$-graded Lie 2-algebras
$\hSec{\bullet}{\BaseAlgebroid}\stackrel{\diff}{\mapsto} \polymult{\bullet}{\BaseGroupoid}$ 
and 
$\hSec{\bullet}{\BaseAlgebroid[X]}\stackrel{\diff'}{\mapsto} \polymult{\bullet}{\BaseGroupoid[X]}$
are homotopy equivalent.
 Moreover, there is a  canonical homotopy
equivalence   class of morphisms
 between these $\zz$-graded Lie 2-algebras,
 which  is the composition of the
homotopy inverse of ${\proj}$ 
 with the inclusion ${\mathfrak i}$.
   The following result extends Theorem 7.4 in \cite{OrtizWaldron}.

\begin{theorem}
\label{cor:MoritaEquivalent1}
Let $\gpoidmmi$ and $\gpoidmmm$ be Morita equivalent Lie groupoids.
Then any  $\Gamma_1$-$\Gamma_2$-bitorsor $M_1\leftarrow X\to M_2$ 
induces a homotopy equivalence between the  $\zz$-graded Lie 2-algebra
$\hSec{\bullet}{\BaseAlgebroid_2}\stackrel{\diff'}{\mapsto} \polymult{\bullet}{\BaseGroupoid_2}$
of polyvector fields on $\gpoidmmm$
and   the $\zz$-graded Lie 2-algebra
 $\hSec{\bullet}{\BaseAlgebroid_1}\stackrel{\diff}{\mapsto} \polymult{\bullet}{\BaseGroupoid_1}$
of polyvector fields on $\gpoidmmi$.
\end{theorem}

By construction,
 the assignment in Theorem \ref{cor:MoritaEquivalent1}
is functorial. More precisely,  let  $\Groupoid$
be the category  whose 
 objects are Lie groupoids, and arrows are Morita bitorsors
 (up to isomorphisms), and $\Lieh$ be the category whose
objects are {\strict}s, and  arrows are homotopy equivalence
classes of  morphisms of {\strict}s. In summary, we have the following

\begin{corollary}
\label{cor:MoritaEquivalent2}
The assignment in Theorem \ref{cor:MoritaEquivalent1}
is  a   functor from the category $\Groupoid$ to the category  $\Lieh$.
\end{corollary}

Such a functor  is called \emph{the polyvector field functor}.
Corollary \ref{cor:MoritaEquivalent2} justifies the following 

\begin{definition}
\label{eq:poly}
Let $\XX$ be a differentiable stack. The space of polyvector fields on
$\XX$ is defined to be the homotopy equivalence class of {\strict}s 
$\hSec{\bullet}{\BaseAlgebroid}\stackrel{\diff}{\mapsto} \polymult{\bullet}{\BaseGroupoid}$, where $\BaseGroupoid \toto M$
is any Lie groupoid representing $\XX$.
\end{definition}

\begin{remark}
\label{rk:1}
We expect that the associated dg Lie algebra of polyvector fields
  in Definition \ref{eq:poly} corresponds  to
a  2-term truncation of the dg Lie algebra of
polyvector fields $\text{Pol}(X, 1)$ in \cite[Section 3.1]{CPTVV}
and $\widehat{\text{Pol}}(A, 1)$ in \cite[Section 3.3.1]{Pr1}.
See   \cite[Section 4.2]{Pr2}.
\end{remark}

\begin{remark}
Note that, to any Morita morphism, a canonical
bitorsor  is associated. In the sequel, we will use both  of them
interchangeably.  Assume that $\phi$ is a Morita morphism
of Lie groupoids from $\Gamma_1\toto M_1$ to $\Gamma_2\toto M_2$.
It is easy to check that

$$
\begin{array}{ccccc}
\Gamma_1&&M_1\times_{M_{2},t_2}\Gamma_2&&\Gamma_2\\
\downarrow\downarrow&\stackrel{\sigma_1 }{\swarrow}&& \stackrel{\sigma_2 }{\searrow} &\downarrow\downarrow\\
M_1&&&&M_{2}\\
\end{array}
$$
is a $\Gamma_1$-$\Gamma_2$-bitorsor.
 Here $\sigma_1 (m,\gamma)=m$, $\sigma_2 (m,\gamma)=s_2 (\gamma)$,
  $\forall \  (m,\gamma) \in  M_1\times_{M_{2},t_2}\Gamma_2$.
The left action of $ \Gamma_1\toto M_1$ on
 $M_1\times_{M_{2},t_2}\Gamma_2$ is given by
$$\gamma_1\cdot (m_1,\gamma_2)=(t_1(\gamma_1),\phi(\gamma_1)\gamma_2),$$
while the right action of  $\Gamma_2\toto M_2$
on $M_1\times_{M_{2},t_2}\Gamma_2$ is given by
$$(m_1,\gamma_2)\cdot \gamma_2'=(m_1,\gamma_2\gamma_2'),$$
whenever composable.
\end{remark}

It follows from Theorem \ref{cor:MoritaEquivalent1} that,
for Morita equivalent Lie groupoids  $\gpoidmmi$ and $\gpoidmmm$,
the corresponding dglas 
${\Nu}^\bullet ({\BaseGroupoid_1})$ and ${\Nu}^\bullet ({\BaseGroupoid_2})$
are quasi-isomorphic as $L_\infty$-algebras.
At the level of cohomology,
 this induces an  isomorphism of $\mathbb Z$-graded Lie algebras.
 The following result  extends Corollary 7.2 in \cite{OrtizWaldron}
 to polyvector fields:

\begin{corollary}
Under the same hypothesis as in Theorem \ref{cor:MoritaEquivalent1},
 there is   an  isomorphism  of $\zz$-graded Lie algebras
$H^\bullet ({\Nu}({\BaseGroupoid_1}))\simeq  H^\bullet ({\Nu}({\BaseGroupoid_2}))$. 
\end{corollary}

\section{$(+1)$-shifted Poisson structures on differentiable stacks}

\subsection{Quasi-Poisson groupoids}

First, we recall the definition of quasi-Poisson groupoids \cite{IPLGX}.
We follow the notations  of  Lemma \ref{lem:rightAndleft}.

\begin{definition} [\cite{IPLGX}]
\label{def:quasiPoisson}
Let $\BaseGroupoid\toto M$ be a Lie groupoid.
\begin{enumerate}
\item[(i)] A quasi-Poisson structure on $\BaseGroupoid\toto M$ is a pair
 $(\Pi,\Lambda)$, with $\Pi \in \polymult{1}{\BaseGroupoid}$
a multiplicative bivector field on 
$\BaseGroupoid$ and $\Lambda \in  \hSec{2}{\BaseAlgebroid}$
satisfying
\begin{equation}
\label{q_poisson_compatibility}
\frac{1}{2} [\Pi,\Pi] = {\rm d}\Lambda\;,\;\;\; \delta_\Pi (\Lambda) = 0\;. 
\end{equation}
\item[(ii)] Quasi-Poisson structures $(\Pi_1,\Lambda_1)$ and $(\Pi_2,\Lambda_2)$ on
 $\gpoidmm$
 are said to be \emph{twist equivalent} if there exists a section 
${\twist} \in \hSec{1}{\BaseAlgebroid}$, called the \emph{twist}, such that 
 \begin{equation}
  \label{eq:quasiEquivQuasiPoisson}
  \Pi_2 = \Pi_1+ {\rm d} {\twist}, \ \ \ \Lambda_2 =\Lambda_1 -
 \delta_{\Pi_1} ({\twist}) - \frac{1}{2} [{\twist},{\twist}]_{}. 
 \end{equation} 
\end{enumerate}
\end{definition}

In the sequel, we  will
denote the  quasi-Poisson structure
 $(\Pi + {\rm d} {\twist},\Lambda - \delta_{\Pi} ({\twist}) -
 \frac{1}{2} [{\twist},{\twist}]_{ })$ by   $(\Pi_\twist,\Lambda_\twist)$.
Quasi-Poisson structures and twist equivalences can
 be described  completely in term of {\strict}s. See Appendix 
\ref{set:MC}.

\begin{proposition}
\label{prop:MC=quasiPoisson}
Let  $\gpoidmm$ be a Lie groupoid.
\begin{enumerate}
\item[(i)]  There is a one-one correspondence between
 quasi-Poisson structures on $\gpoidmm$  
 and  Maurer-Cartan elements of the {\strict}
$\hSec{\bullet}{\BaseAlgebroid}\stackrel{\diff}{\mapsto} \polymult{\bullet}{\BaseGroupoid}$.
\item[(ii)]  Quasi-Poisson structures on the same Lie  groupoid
 $\gpoidmm$
 are twist equivalent if and only if they  correspond to
gauge  equivalent Maurer-Cartan elements 
of  the {\strict} $\hSec{\bullet}{\BaseAlgebroid}\stackrel{\diff}{\mapsto} \polymult{\bullet}{\BaseGroupoid}$ with  the gauge element being
in $\hSec{1}{\BaseAlgebroid}$.
\end{enumerate}
\end{proposition}
\begin{proof}
The first assertion is quite obvious.  For (ii), see Proposition
 \ref{prop:gauge=twist}.
\end{proof}

\begin{remark}
\label{rk:2}
Proposition \ref{prop:MC=quasiPoisson} essentially
states that 
quasi-Poisson structures on a  Lie  groupoid $\Gamma \toto M$ 
moduli twists is in bijection to the   Maurer-Cartan moduli set of
 the dgla associated to the Lie 2-algebra $\hSec{\bullet}{\BaseAlgebroid}\stackrel{\diff}{\mapsto} \polymult{\bullet}{\BaseGroupoid}$ with 
 the gauge element being in $\hSec{1}{\BaseAlgebroid}$. 
 In spirit, this is   parallel  to
\cite[Definition 3.1.1]{CPTVV},
\cite[Definition 1.5]{Pr1} and \cite[Definition 2.5]{Pr2}.
\end{remark}

As a consequence, for a given Lie groupoid $\Gamma \toto M$,
the  Maurer-Cartan  moduli set (see Definition \ref{MC_moduli_set})
$\underline{MC( \hSec{\bullet}{\BaseAlgebroid}\stackrel{\diff}{\mapsto} \polymult{\bullet}{\BaseGroupoid} ) }$
 of  the {\strict}
$\hSec{\bullet}{\BaseAlgebroid}\stackrel{\diff}{\mapsto} \polymult{\bullet}{\BaseGroupoid}$
coincides with the set of twist equivalence classes  of quasi-Poisson structures on
$\gpoidmm$.
 The composition of the polyvector field
 functor $\Groupoid\to \Lieh$ (Corollary
 \ref{cor:MoritaEquivalent2}) with the Maurer-Cartan functor (see the end of Appendix A)
 is a functor from the category $\Groupoid$ to the category \emph{Sets}, called 
 the \emph{Poisson functor} and denoted $\Poiss$.

According to   Proposition \ref{prop:MC=quasiPoisson},
 the Poisson functor associates to a Lie groupoid $\Gamma\toto M$
its moduli set of quasi-Poisson structures up to twists
 $\Poiss (\BaseGroupoid) := \underline{MC( \hSec{\bullet}{\BaseAlgebroid}
\stackrel{\diff}{\mapsto} \polymult{\bullet}{\BaseGroupoid})}$,
and to a Morita equivalence of Lie groupoids
 the induced bijection between the corresponding moduli sets.
 We denote by $\underline{\Lambda\oplus\Pi}$ the class in $\Poiss(\BaseGroupoid)$ of a quasi-Poisson
groupoid $( \BaseGroupoid \toto M, \Pi, \Lambda )$.

\begin{lemma}
\label{lem:QuasiPoissonUpToTwistsFunctor}
Let  $(\gpoidmmi, \Pi_1, \Lambda_1)$ and $(\gpoidmmm, \Pi_2, \Lambda_2)$
 be quasi-Poisson groupoids.
 Let $\phi$ be a Morita morphism from $\gpoidmmi$ to $\gpoidmmm$.
The following  statements are equivalent:
\begin{enumerate}
 \item[(i)]  Under the Poisson functor
$\Poiss(\phi): \Poiss (\BaseGroupoid_1) \simeq   \Poiss (\BaseGroupoid_2)$,
 the class  $\underline{(\Lambda_1\oplus\Pi_1)} \in 
\Poiss (\BaseGroupoid_1)$ corresponds to
 $\underline{(\Lambda_2\oplus\Pi_2)} \in \Poiss (\BaseGroupoid_2)$;
 \item[(ii)] The following relation holds
 $$\underline{MC}(\phi )^{-1}  \left( \underline{(\Lambda_2\oplus\Pi_2)} \right) =  \underline{MC}({\mathfrak i})^{-1}
 \left( \underline{(\Lambda_1\oplus\Pi_1)} \right) . $$
\end{enumerate}
\end{lemma}
\begin{proof}
In order to compare with Proposition \ref{th:linfty_morphism},
here we denote $\BaseGroupoid_2\toto M_2$ by $\Gamma\toto M$,
and   $M_1$ by $X$. Then  we  identify  $\BaseGroupoid_1\toto M_1$
with $\BaseGroupoid[X]\toto X$, and 
  $\phi: \BaseGroupoid_1\to \BaseGroupoid_2$  with the
    projection map $\proj :\BaseGroupoid[X]\to \BaseGroupoid$.

The polyvector field functor  assigns to
 the Morita morphism $\phi: \Gamma_2\to \Gamma_1$
 an  homotopy  equivalent class of $\zz$-graded
 Lie 2-algebra morphisms  from polyvector fields on $ \BaseGroupoid_2\toto M_2$ 
to polyvector fields on $ \BaseGroupoid_1\toto M_1$.
The latter  can be represented   by the
 composition $ {\proj} \circ {\mathfrak i}^{-1} $,
 where $\proj$ and $ {\mathfrak i}$ are as in Diagram (\ref{eq:purpose0}) and ${\mathfrak i}^{-1}$ is an homotopy inverse as in Proposition \ref{th:linfty_morphism} ($2$). Then our result follows immediately by
  functoriality of the Maurer-Cartan functor.
\end{proof}

It is  standard  that  given a  {\dgla}
$(\frakg, {\rm d}, [\cdot,\cdot])$ and  a Maurer-Cartan
element $\lambda \in \frakg^1$, the triple
$(\frakg, {\rm d}+[ \lambda, \cdot], [\cdot,\cdot])$ is again a {\dgla},
called the tangent {\dgla} \cite{Kontsevich}.
In our case, for a given quasi-Poisson structure $(\Pi,\Lambda)$ on 
$\gpoidmm$, since $(\Pi,\Lambda)$ is a Maurer-Cartan
element in $\Nu^\bullet (\BaseGroupoid)$,
 the resulting twisted  differential is given as follows:
 $$
 \begin{array}{rrcl}
 {\rm d}_{\Pi,\Lambda}: &    \Nu^{k}(\BaseGroupoid) & \mapsto &   \Nu^{k+1}(\BaseGroupoid) \\
 & a\oplus P  & \to & (- \delta_\Pi (a) - \delta_P (\Lambda) )\oplus( [\Pi,P] +  {\rm d} a )   , 
 \end{array}
 $$
where  $P \in  \polymult{k}{\BaseGroupoid}$ and
 $a \in \hSec{k+1}{\BaseAlgebroid}$.
As in  classical Poisson geometry, we introduce the following

\begin{definition}
Let  $(\gpoidmm, \Pi,\Lambda)$ be a quasi-Poisson groupoid.
 The complex $\big(\Nu^{\bullet}(\BaseGroupoid), {\rm d}_{\Pi,\Lambda}\big)$
is called the \emph{Lichnerowicz-Poisson (LP) cochain complex} of
 the quasi-Poisson structure $(\Pi,\Lambda)$, and  its cohomology is
called   the \emph{Lichnerowicz-Poisson cohomology} of $(\Pi,\Lambda)$,
denoted by $H^\bullet_{LP} (\gpoidmm, (\Pi, \Lambda))$
\end{definition}

Since twist equivalent quasi-Poisson structures are gauge equivalent
 according to  Proposition \ref{prop:MC=quasiPoisson},  the following proposition is immediate.

\begin{proposition}
\label{prop:chainiso}
If   quasi-Poisson structures on a Lie groupoid
$\Gamma \toto M$ are twist equivalent, their corresponding
 Lichnerowicz-Poisson cohomologies are  isomorphic.
\end{proposition}

\subsection{Morita equivalence and $(+1)$-shifted Poisson differentiable  
stacks}

\begin{definition}
\label{def:MoritaMorphQuasiPoisson}
Let  $(\gpoidmmi, \Pi_1, \Lambda_1)$ and $(\gpoidmmm, \Pi_2,\Lambda_2)$ 
 be quasi-Poisson groupoids.
 By a Morita morphism
 of quasi-Poisson groupoids from
  $(\gpoidmmi, \Pi_1, \Lambda_1)$ to $(\gpoidmmm, \Pi_2, \Lambda_2)$, 
we mean a  Morita morphism of Lie groupoids
\begin{equation}
\label{eq:Moritamorphism}
\begin{tikzcd}
{\Gamma_1} \arrow[xshift=-3pt]{d} \arrow[xshift=3pt]{d} \arrow{r}{\phi} &
{\Gamma_2} \arrow[xshift=-3pt]{d} \arrow[xshift=3pt]{d} \\
{M_1} \arrow{r}{\varphi} & {M_2}
\end{tikzcd}
\end{equation}
such that 
\begin{enumerate}
 \item[(i)] there exists a twist $T \in \Sigma^1 (A_1)$
 such that  $e^{T}  \cdot (\Lambda_1\oplus\Pi_1) $ is a projectable
 quasi-Poisson structure on $\BaseGroupoid_1 \toto M_1$;
 \item[(ii)]  $ \phi_* (e^{T}\cdot (\Lambda_1\oplus\Pi_1)) = 
\Lambda_2\oplus\Pi_2$, i.e.
$ (\phi_*)(\Pi_{1})_T =\Pi_2$ and $(\phi_*) (\Lambda_{1})_T=\Lambda_2$.
\end{enumerate}
\end{definition}

\begin{lemma}
\label{lem:MoritaMorphQuasiPoisson}
Let  $(\gpoidmmi, \Pi_1, \Lambda_1)$ and $(\gpoidmmm, \Pi_2, \Lambda_2)$ be
  quasi-Poisson groupoids, and
$$
\begin{tikzcd}
{\Gamma_1} \arrow[xshift=-3pt]{d} \arrow[xshift=3pt]{d} \arrow{r}{\phi} &
{\Gamma_2} \arrow[xshift=-3pt]{d} \arrow[xshift=3pt]{d} \\
{M_1} \arrow{r}{\varphi} & {M_2}
\end{tikzcd}
$$ 
 a Morita morphism of Lie groupoids. Then the following
statements are equivalent.
\begin{enumerate}
\item[(i)] $\phi$ is a Morita morphism of quasi-Poisson groupoids;
 \item[(ii)] There exists a twist $T_1\in \Sigma^1 (A_1)$
 such that  $e^{T_1}  \cdot (\Lambda_1\oplus\Pi_1)$ is  projectable,
 and $\phi_*(e^{T_1}(\Lambda_1\oplus\Pi_1)) =
 e^{T_2} \cdot (\Lambda_2\oplus\Pi_2 )$ for some $T_2\in \Sigma^1 (A_2)$.
 \item[(iii)] The  relation
 $ \Poiss (\phi) (\underline{ \Lambda_1\oplus\Pi_1})=\underline{\Lambda_2\oplus\Pi_2} $ holds;
\item[(iv)] The relation $\underline{MC}(\phi ) \smalcirc
 \underline{MC}({\mathfrak i})^{-1} \left(\underline{ \Lambda_1\oplus\Pi_1}\right)= \underline{(\Lambda_2\oplus\Pi_2)} )$ holds.
\end{enumerate}
\end{lemma} 
\begin{proof}
To be consistent with the notations introduced earlier,
let us denote $\BaseGroupoid_2\toto M_2$ by $\Gamma\toto M$,
and   $M_1$ by $X$. Then  we can identify  $\BaseGroupoid_1\toto M_1$
with $\BaseGroupoid[X]\toto X$, and
thus  $\phi: \BaseGroupoid_1\to \BaseGroupoid_2$ is simply the
    projection map $\proj :\BaseGroupoid[X]\to \BaseGroupoid$.

First, we prove the equivalence of (i) and (ii).
It is obvious that  (ii)  holds 
if $\phi$ is a Morita morphism  of  quasi-Poisson groupoids
as defined in Definition~\ref{def:MoritaMorphQuasiPoisson}.

Conversely, assume that  (ii) holds.
Let $T' \in \Sigma^1 (\BaseAlgebroid[X])_{proj}$ be
any projectable section such that $\phi_* (T') = \proj (T') = T_2$.
For instance, choose an Ehresmann connection on $\varphi:X\rightarrow M$, and
 take $T'=\lambda_\hordis (T_2)$,  where $\lambda_\hordis$ is as in (\ref{eq:purpose}).
It is simple to check that
$$\phi_* \left(  e^{T_1-T'}(\Lambda_1\oplus\Pi_1) \right)=
\Lambda_2\oplus\Pi_2 \, .$$
Therefore, $\phi$ is indeed a Morita morphism of quasi-Poisson groupoids.

Next, we prove the equivalence between (ii) and (iv).
 Let $\tilde{\Lambda}\oplus \tilde{\Pi}$ be any representative
 of $\underline{MC}({\mathfrak i})^{-1} (\underline{ \Lambda_1\oplus \Pi_1 })$.
 By definition, $\tilde{\Lambda} \oplus \tilde{\Pi}$ is projectable,
 and is twist equivalent to
 $\Lambda_1 \oplus \Pi_1$. Moreover
 $(\BaseGroupoid[X]\toto X, \tilde{\Pi} ,  \tilde{\Lambda})$
is a quasi-Poisson groupoid. The condition 
$\underline{MC}({\proj})   \left(\underline{ \tilde{\Lambda}\oplus \tilde{\Pi} }
\right)= \underline{(\Lambda_2 \oplus \Pi_2)}$  is then equivalent to
 $\proj_* ( \tilde{\Lambda}\oplus \tilde{\Pi}) = e^{T_2}
\cdot(\Lambda_2 \oplus \Pi_2 )$ for some $T_2\in \Sigma^1 (A_2)$.
Therefore (ii) and (iv) are  indeed equivalent.	

Finally, Lemma \ref{lem:QuasiPoissonUpToTwistsFunctor} implies that (iii)
  and (iv)  are equivalent.
This concludes the proof of the lemma.
\end{proof}

We are now ready to introduce the Morita equivalence of
 quasi-Poisson groupoids.
 
\begin{definition}
\label{def:morita}
Quasi-Poisson groupoids 
$(\gpoidmmi, \Pi_1, \Lambda_1)$ and $(\gpoidmmm, \Pi_2, \Lambda_2)$ 
are Morita equivalent
if  there exists a  third quasi-Poisson groupoid
$(\Xi\toto X, \Pi_X, \Lambda_X)$
 and  Morita morphisms of quasi-Poisson groupoids $(\Xi\toto X, \Pi_X, \Lambda_X)\to 
(\gpoidmmi, \Pi_1, \Lambda_1)$
and $(\Xi\toto X, \Pi_X, \Lambda_X)  \to (\gpoidmmm, \Pi_2, \Lambda_2)$.
\end{definition}

In order to prove that this is indeed an equivalence relation,
we need to describe 
Morita equivalence in terms of the Poisson functor $\Poiss$.
Recall that $\underline{\Lambda \oplus \Pi}$ stands for the class
of $\Lambda \oplus \Pi$ in the moduli set $Pois(\BaseGroupoid) :=
 \underline{MC(  \hSec{\bullet}{\BaseAlgebroid}\stackrel{\diff}{\mapsto} 
\polymult{\bullet}{\BaseGroupoid} )}$.

\begin{proposition}
\label{prop:equiv.def.MEquasiPoisson}
Quasi-Poisson groupoids 
$(\gpoidmmi, \Pi_1, \Lambda_1)$ and $(\gpoidmmm, \Pi_2, \Lambda_2)$ 
are Morita equivalent
if and only if  there exists 
a  bitorsor $M_1\leftarrow X\rightarrow M_2$ between 
  $\gpoidmmi$ and $\gpoidmmm$  such that
\begin{equation}
\label{eq:poissonMorph}
\Poiss (M_1\leftarrow X\rightarrow M_2) (\underline{  \Pi_1 \oplus \Lambda_1} ) = \underline{ 
 \Pi_2 \oplus \Lambda_2} .
\end{equation}
\end{proposition}
\begin{proof}
Assume that $(\gpoidmmi, \Pi_1, \Lambda_1)$ and $(\gpoidmmm, \Pi_2, \Lambda_2)$
are Morita equivalent quasi-Poisson groupoids. 
 By definition, there exists a third quasi-Poisson groupoid
$(\Xi\toto X, \Pi_X, \Lambda_X)$
 and  Morita morphisms of quasi-Poisson groupoids $\phi_{1}:(\Xi\toto X, \Pi_X, \Lambda_X)\to 
(\gpoidmmi, \Pi_1, \Lambda_1)$
and $\phi_2:(\Xi\toto X, \Pi_X, \Lambda_X)  \to (\gpoidmmm, \Pi_2, \Lambda_2)$.
According to  Lemma \ref{lem:MoritaMorphQuasiPoisson}, we have

$$ \Poiss (\phi_{1}) (\underline{ \Lambda_X \oplus \Pi_X })=\underline{\Lambda_1 \oplus \Pi_1} \hbox{ and }
\Poiss (\phi_{2}) (\underline{ \Lambda_X \oplus\Pi_X })=\underline{\Lambda_2 \oplus \Pi_2}\, .$$
 This implies that:
  $$ \Poiss (\phi_{2}) \circ \Poiss (\phi_{1})^{-1} \left( \underline{\Lambda_1 \oplus \Pi_1} \right) = \underline{\Lambda_2 \oplus \Pi_2}\, . $$
  But the composition $\Poiss (\phi_{2}) \circ \Poiss (\phi_{1})^{-1}$ is exactly $\Poiss (M_1\leftarrow X\rightarrow M_2)$, since $\Poiss$ is a functor.

Conversely, assume that $M_1\leftarrow X\rightarrow M_2$ is
 a bitorsor between the Lie groupoids
 $\gpoidmmi$ and $\gpoidmmm$, 
  and the  quasi-Poisson structures $(\Pi_i,\Lambda_i)$
  on $\BaseGroupoid_i\toto M_i$, $i=1,2$ are related to each other
by the following condition:
    \begin{equation}
    \label{eq:ifClassCorrespond} 
 \Poiss (M_1\leftarrow X\rightarrow M_2)  \left( \underline{\Lambda_1\oplus\Pi_1} \right) = \underline{\Lambda_2\oplus\Pi_2} .
    \end{equation}

 Let $\Gamma_1[X] \toto X$ and $\Gamma_2[X] \toto X$
 be the
 pull-back Lie groupoids of 
$\Gamma_1 \toto M_1$ and $\Gamma_2 \toto M_2$
 via
the surjective submersions $X\to M_1$ and $X\to M_2$, respectively.
Since $M_1\leftarrow X\rightarrow M_2$ is a bitorsor, 
 $\Gamma_1 [X] \toto X$ is canonically isomorphic to 
$\Gamma_2[X] \toto X$.
Denote the projections  from $\Gamma_1 [X] \cong \Gamma_2[X]\toto X$ 
 to $\Gamma_1 \toto M_1 $,
 and  to $\Gamma_2 \toto M_2$ by $\phi_{1} $ and $\phi_{2}$,
 respectively.
 By functoriality, we have 
 \[
 \Poiss (M_1\leftarrow X\rightarrow M_2) =\Poiss (\phi_{2}) \circ 
\Poiss (\phi_{1})^{-1}\, .
\]
Then  Equation (\ref{eq:ifClassCorrespond}) implies that 
\begin{equation}
\label{eq:FCO} 
  \Poiss (\phi_{1})^{-1} \left( \underline{\Lambda_1\oplus\Pi_1} \right) = \Poiss (\phi_{2})^{-1} \left( \underline{\Lambda_2\oplus\Pi_2}\right)\, .
\end{equation}
 Let $(\Pi_X, \Lambda_X)$ be any quasi-Poisson structure on $\Gamma_1 [X]\cong
\Gamma_2 [X]\toto X$
 representing    the class \eqref{eq:FCO}.
  By construction, we have
 $$ \Poiss (\phi_{1}) (\underline{ \Lambda_X\oplus\Pi_X })=
\underline{\Lambda_1\oplus\Pi_1} \ \ \  \hbox{ and }   
\Poiss (\phi_{2}) (\underline{ \Lambda_X\oplus\Pi_X })=
\underline{\Lambda_2\oplus\Pi_2}\,  .$$
 According to  Lemma \ref{lem:MoritaMorphQuasiPoisson},
 both $\phi_{1}$ and $\phi_{2}$ are Morita morphisms of
 quasi-Poisson groupoids. 
As a consequence,  the quasi-Poisson groupoids $(\gpoidmmi,\Pi_1,\Lambda_1)$ 
and $(\gpoidmmm,\Pi_2,\Lambda_2)$ are Morita equivalent.
\end{proof}

\begin{corollary}
Morita equivalence in Definition \eqref{def:morita} is  indeed an equivalence relation among quasi-Poisson Lie groupoids.
\end{corollary} 
 \begin{proof}
This follows immediately from Proposition
 \ref{prop:equiv.def.MEquasiPoisson}, together with the fact that  $\Poiss$ is a functor. 
 \end{proof}

\begin{theorem}
\label{prop:MEandQuasiPoisson}
Let $(\gpoidmmi, \Pi_1, \Lambda_1)$ be  a quasi-Poisson groupoid.
Assume that $\Gamma_2\toto M_2$ is any Lie groupoid Morita equivalent
to $\Gamma_1\toto M_1$ as Lie groupoids. Then
there exists a quasi-Poisson structure $(\Pi_2, \Lambda_2)$, unique up
to twists,
on $\gpoidmmm$ such that $(\gpoidmmm, \Pi_2, \Lambda_2)$ and
 $(\gpoidmmi, \Pi_1, \Lambda_1)$
are Morita  equivalent quasi-Poisson groupoids. 
\end{theorem}
\begin{proof}
 This is an immediate  consequence of Proposition
 \ref{prop:equiv.def.MEquasiPoisson}. 
\end{proof}

We are now ready to introduce

\begin{definition}
\label{poiss.stack}
A $(+1)$-shifted Poisson differentiable stack, up to isomorphisms,
 is a Morita  equivalence class of quasi-Poisson groupoids.
\end{definition}

We will use the notation $( \XX, \pixx)$ to
denote a $(+1)$-shifted Poisson differentiable  stack.

The following lemma follows from the general fact
concerning  tangent cohomology of a ${\dgla}$
at Maurer-Cartan elements \cite{Kontsevich}.

\begin{lemma}
Assume that $\phi$ is a  Morita morphism
of quasi-Poisson groupoids from
 $(\gpoidmmi, \Pi_1, \Lambda_1)$ to $(\gpoidmmm, \Pi_2, \Lambda_2)$.
 Then $\phi$ induces
an  isomorphism of the Lichnerowicz-Poisson cohomology
$$\phi_* :H^\bullet_{LP} (\gpoidmmi, (\Pi_1, \Lambda_1))\longiso 
H^\bullet_{LP} (\gpoidmmm, (\Pi_2, \Lambda_2)). $$
\end{lemma}

Since Lichnerowicz-Poisson cohomology of  quasi-Poisson groupoids
is invariant under Morita equivalence, the following definition is well-posed.

\begin{definition}
Let $( \XX, \pixx)$ be  a $(+1)$-shifted Poisson differentiable  stack. Its Lichnerowicz-Poisson cohomology is
$$H^\bullet_{LP} ( \XX, \pixx): =H^\bullet_{LP} (\gpoidmm, (\Pi, \Lambda)),$$
where  $(\gpoidmm, \Pi, \Lambda)$ is any  quasi-Poisson groupoid
representing $( \XX, \pixx)$.
\end{definition} 

\bigskip

\section{Homotopy and Morita equivalence of VB-groupoids}

Since the multiplicative bivector fields associated to quasi-Poisson groupoids
 induce VB-groupoid morphisms between the tangent and cotangent bundles,
 it is natural to ask how Morita equivalence of quasi-Poisson groupoids is
 reflected on relations between these  VB-groupoid morphisms.
 The purpose of this section is to study the framework needed to understand this relation. 
Our main result is  singled out in a separate subsection  at the end,
and shows that  Morita equivalent quasi-Poisson groupoids indeed
 induce the correct notion of equivalence of the underlying VB-groupoid morphisms.

\subsection{Homotopy equivalence of VB-groupoids}\label{sec:homot_equiv}

This section is devoted to the study of homotopy equivalence of VB-groupoids.
We  first recall some basic notions and results about VB-groupoids, following \cite{Gracia-Saz-Mehta,Mackenzie, MackenzieII}. 
We  will  introduce the definition of  homotopies of
VB-groupoid morphisms.
Our main examples of VB-groupoids are the tangent and cotangent groupoids.
Morphisms from the cotangent groupoids
to the  tangent groupoids  induced by multiplicative bivector fields
on Lie groupoids  will be our main examples
 of VB-groupoid morphisms, 
while  those corresponding to  their twists will provide our main examples of
 homotopies.

Recall that a \emph{VB-groupoid} is a groupoid object in the category of
vector bundles. 
In more concrete terms, a VB-groupoid is a pair  of Lie groupoids
 $V\toto E$ and $\Gamma\toto M$,  where $V\rightarrow\Gamma$
 and $E\rightarrow M$   are vector bundles,
 satisfying a list of compatibility
conditions (see  \cite{Mackenzie, MackenzieII}).
 A VB-groupoid is either denoted by the diagram: 
\begin{equation}\label{eq:florence}
\VB{V}{E}{\Gamma}{M}
\end{equation}
or simply by $V$ for short. The source and target maps of
 $V \toto E$ are denoted by $s_V,t_V$, while $s,t$
 stand for the source and target maps of $\Gamma \toto M$. 
For each $\gamma\in\Gamma$,  we denote by $0_\gamma^V$ the zero element of
 the fiber $V_\gamma$.
The {\it core} 
 $$C := \Ker ( s_V : V|_M \mapsto E)$$
 is a vector bundle over $M$.
The \emph{core-anchor} 
\begin{equation}
\label{eq:core_anchor}
\rho_V : C \mapsto E
\end{equation}
 is defined to be
the restriction of $t_V : V|_M \mapsto E$ to the core.
There are natural embeddings  
$R_V:t^*C\rightarrow V$ and $L_V:s^*C\rightarrow V$ defined by
\begin{equation}
\label{core_embedding}
L_V(c) =- 0_\gamma^V\cdot c^{-1}\;,\;\;\;
R_V(c') = c'\cdot 0_\gamma^V\;,
\end{equation}
for all $\gamma\in\Gamma$, $c \in C_{s(\gamma)}$  and $c' \in C_{t(\gamma)}$.
Here   the dot and the upscript
 stand for the groupoid multiplication and the inverse
of $V\toto E$. The embedding $R_V$ fits into the following exact sequence of vector bundles over $\Gamma$:
\begin{equation}
\label{VB_exact_sequence}
0\to t^*C\stackrel{R_V}{\to} V \stackrel{s_V}{\to} s^*E \to 0 . 
\end{equation}
There is an analogous short exact sequence for $L_V$. 
 The restriction of the short exact sequence (\ref{VB_exact_sequence}) to $M$ admits a canonical splitting given by the unit map
of $V \toto E$.
A splitting of \eqref{VB_exact_sequence} that coincides with such a
canonical splitting when restricted to $M$ is called a {\it right decomposition}. Every right decomposition induces a vector bundle isomorphism 
$\pi:V\simeq t^*C\oplus s^*E$ over $\Gamma$. By transporting the VB-groupoid
 structure on $V$ to the latter, we obtain a VB-groupoid on $t^* C\oplus s^*E$,
 called  {\it split VB-groupoid}. See  \cite{Gracia-Saz-Mehta} for explicit structure maps.

\begin{example} \label{eq:TG} For any Lie groupoid $\Gamma\toto M$,
 the \emph{tangent groupoid} is a VB-groupoid
\begin{equation}\label{eq:paris}\xymatrix{ T\Gamma\ar[d] \ar@<-.5ex>[r] \ar@<.5ex>[r] & TM \ar[d] \\ \Gamma \ar@<-.5ex>[r] \ar@<.5ex>[r] & M }.\end{equation}
The structure maps of $T\Gamma \toto TM$ are the 
tangent maps of the structure maps of $\Gamma \toto M$, {\it e.g.} $s_{T\Gamma} = Ts$, $t_{T\Gamma}=Tt$, and so on.
The core is the Lie algebroid $A \to M$ of $\Gamma \toto M$ and the embeddings (\ref{core_embedding}) are  the right and left groupoid
translations, respectively.
\end{example}

Given a VB-groupoid as in (\ref{eq:florence}), the dual bundle  $V^\vee\rightarrow\Gamma$ inherits a VB-groupoid structure called the \emph{dual VB-groupoid} \cite{Mackenzie, MackenzieII}
\begin{equation} \label{eq:perugia}\xymatrix{ V^\vee \ar[d] \ar@<-.5ex>[r] \ar@<.5ex>[r] & C^\vee \ar[d] \\ \Gamma \ar@<-.5ex>[r] \ar@<.5ex>[r] & M }
\end{equation}
where the source and target maps $s_{V^\vee},t_{V^\vee}:V^\vee\rightarrow C^\vee$ are defined, respectively,  by 
\begin{equation}
\label{dual_VB_maps}
\langle s_{V^\vee}(\eta),c\rangle = -\langle \eta, 0_\gamma^V\cdot c^{-1} \rangle\;\;,\langle t_{V^\vee}(\eta),c'\rangle = \langle \eta, c'\cdot 0_\gamma^V \rangle\;
\end{equation}
for all
 $c\in C_{s(\gamma)}$, $c'\in C_{t(\gamma)}$ and $\eta\in V^\vee_\gamma$.
 In particular, one has 
\begin{equation}\label{dual_VB_maps_translate}
R_V = t_{V^\vee}^\vee:t^*C\rightarrow V\;,\;\; L_V = s_{V^\vee}^\vee:s^*C\rightarrow V\;.
\end{equation}
The core of $V^\vee$ is $E^\vee \to M$. Note that the dual VB-groupoid  of
 (\ref{eq:perugia})
is  canonically isomorphic to the VB-groupoid $V$ itself.

\begin{remark}
Let $\Omega\subset V^\vee\times V^\vee\times V^\vee$
be the graph of the multiplication of the groupoid $V^\vee \toto C^\vee$.
Then $\bar{\Omega}=\{(\xi, \eta, -\gamma)| (\xi, \eta, \gamma)\in \Omega\}$
 is the annihilator of the graph of the multiplication of the 
groupoid $V\toto E$.
\end{remark}

\begin{example} \label{eq:TGdual} 
The dual VB-groupoid of the tangent groupoid in Example \ref{eq:TG} is the \emph{cotangent groupoid }:
\begin{equation}\label{eq:paris_dual}\xymatrix{ T^\vee\Gamma\ar[d] \ar@<-.5ex>[r] \ar@<.5ex>[r] & A^\vee \ar[d] \\ \Gamma \ar@<-.5ex>[r] \ar@<.5ex>[r] & M \, \, .}\end{equation}
Its core is $T^\vee M$,
 and the embeddings (\ref{core_embedding}) are  the dual maps of $Ts: T\Gamma \to TM$ and $Tt: T\Gamma \to TM$.
\end{example}

\emph{VB-groupoid morphisms} are both Lie groupoid morphisms and vector bundle morphisms \cite{Mackenzie, MackenzieII}.
The following proposition is  standard   \cite{MackenzieX:1994}:

\begin{proposition}
\label{pro:quasiP}
Let  $(\gpoidmm, \Pi, \Lambda)$ be  a quasi-Poisson groupoid. Then
$\Pi^\# : T^\vee\Gamma\to T\Gamma$ induces a morphism of VB-groupoids from the cotangent VB-groupoid \eqref{eq:paris_dual}  to the tangent VB-groupoid \eqref{eq:paris} .
\end{proposition}

We now introduce the notion of homotopy of VB-groupoid morphisms. 
 Consider VB-groupoids over the same base groupoid $\Gamma \toto M$:
 \begin{equation}
 \label{eq:VB:samebase}
 \VB{V_1}{E_1}{\Gamma}{M}  \hbox{ \begin{tabular}{c} \\ \\  and \end{tabular} } \VB{V_2}{E_2}{\Gamma}{M}  
 \end{equation}

It is simple to see that the space of
VB-groupoid morphisms from $V_1$ to $V_2$ 
over the identity map of $\Gamma$ is a vector space, denoted
by $\Hom_{\Gamma}(V_1, V_2)$.
 
  Denote the cores of $V_1$ and $V_2$ by $C_1$ and $C_2$,
 respectively. 
For any vector bundle morphism
$h: E_1\to C_2$ over the identity map of $M$, we define 
a vector bundle morphism over the identity map on $\Gamma$ by
\begin{equation}\label{eq:JofH}
 \begin{array}{rcl}J_h : V_1 & \mapsto & V_2\\  v & \to &  
L_{V_2}\smalcirc h  \smalcirc  s_{V_1} (v)+ R_{V_2}\smalcirc h \smalcirc  t_{V_1} (v).\end{array} 
\end{equation}

\begin{remark}
Using  (\ref{core_embedding}), we can rewrite
  $J_h$  as follows 
\begin{equation}
\label{eq:homotopy3}
J_h(v_\gamma) = 0_\gamma\cdot h(s_{V_1}(v_\gamma))^{-1} + h(t_{V_1}(v_\gamma))\cdot 0_\gamma\;.
\end{equation}
\end{remark}

The following can be easily verified.

\begin{lemma}
The map $J_h$ is a VB-groupoid morphism from $V_1$ to $V_2$ over the
 identity map on $\Gamma$.
\end{lemma}

\begin{definition}\label{defn:VB_homotopy}
 Let $V_1$ and $V_2$ be VB-groupoids as in
 \eqref{eq:VB:samebase}.
Let $\Phi$ and $\Psi\in  \Hom_{\Gamma}(V_1, V_2)$. 
We say that $\Phi$ is homotopic
to $\Psi$ if there exists a vector bundle morphism
$h: E_1\to C_2$ over the identity map on $M$, where $C_2$ is the core of $V_2$, such that the following relation holds
\begin{equation}
\label{eq:homotopy}
\Phi-\Psi = J_h.
\end{equation}
\end{definition}
We call $J_h$ the {\it VB-homotopy} defined by $h:E_1 \to C_2$. 

\begin{example} \label{prop:TGtwist}
Let  $(\Pi, \Lambda)$ and $(\Pi_T,\Lambda_T) $ be twist equivalent quasi-Poisson structures on a Lie groupoid $ \gpoidmm$. Then
 the  VB-groupoid morphisms $\Pi^\# $ and $\Pi^\#_T  $ considered in Proposition \ref{pro:quasiP} are homotopy equivalent,
with explicit VB-homotopy  being given by  $T^\# : A^\vee \to A$.
\end{example}

\begin{proposition}\label{VB_homotopy_start}
Homotopy equivalence of VB-groupoid morphisms is an equivalence relation and is compatible with composition of VB-groupoid morphisms.
\end{proposition}
\begin{proof}
Homotopy is an equivalence relation since VB-homotopies from $V_1$ to $V_2$ form a subspace of $\Hom_{\Gamma}(V_1, V_2)$. 
Compatibility with composition easily follows from the fact that the
 composition of a VB-homotopy with a VB-morphism is again a VB-homotopy.
\end{proof}

Recall that for $\Phi \in \Hom_{\Gamma}(V_1, V_2)$, its  dual vector bundle morphism $\Phi^ \vee \in \Hom_{\Gamma}( V_2^\vee, V_1^\vee)$. 


\begin{proposition}\label{VB_homotopy_dual}
Let $\Phi$ and $\Psi$ be homotopic VB-groupoid morphisms from $V_1$ to $V_2$ 
with VB-homotopy $J_h$ as in Definition \ref{defn:VB_homotopy}.
Then the dual VB-groupoid morphisms $\Phi^\vee$ and $\Psi^\vee$ are homotopic with VB-homotopy
$J_{h^\vee}$, where $h^\vee:C_2^\vee\rightarrow E_1^\vee$ is the dual of $h: E_1 \to C_2$.
\end{proposition}
\begin{proof}
The proposition is proved by taking the dual of Equation (\ref{eq:homotopy})
 combining with the fact that $ J_h^\vee = J_{h^\vee}$.
The latter can be easily verified by  using   Equation (\ref{dual_VB_maps_translate}).
\end{proof}

Now we are ready to introduce the notion of homotopy equivalence
 of VB-groupoids.

\begin{definition}
	\label{def:homotopyEquivVbgroup}
 Let $V_1$ and $V_2$ be VB-groupoids as in (\ref{eq:VB:samebase}).
An homotopy equivalence between 
 $V_1$  and $V_2$ is a pair of VB-groupoid morphisms
 $\Phi\in\Hom_\Gamma(V_1,V_2)$ and $\Psi\in\Hom_\Gamma(V_2,V_1)$
such that both $\Phi\circ\Psi$ and $\Psi\circ\Phi$ are homotopic to the identity VB-groupoid morphism.
\end{definition}

In the sequel, we use the following notation to denote a homotopy equivalence:
\begin{equation}
\SelectTips{eu}{12}
\label{eq:homotopyEquivNotations}
\xymatrix{V_1  \ar@(ul,dl)[]_{J_{h_1}} \ar@<-.5ex>[r]_\Phi & \ar@(ur,dr)[]^{J_{h_2}}  \ar@<-.5ex>[l]_\Psi V_2} 
\end{equation}
where $h_1: E_1\to C_1$ and $h_2: E_2\to C_2$ are bundle maps.  
 
\begin{remark}
\label{comparison_ortis_delhoyo}
A similar notion appeared in  \cite[Section 6]{OrtizDelHoyo}.
 It can be checked that
 VB-groupoid morphisms are homotopic as in Definition \ref{def:homotopyEquivVbgroup} if and only if  
they are isomorphic according to \cite{OrtizDelHoyo}.
\end{remark}

\subsection{Generalized morphisms and Morita equivalence of VB-groupoids}\label{Generalized_VB_morphisms}

In this section we consider generalized VB-morphisms,  extending the well known notion for Lie groupoids, and relate them with Morita equivalences of VB-groupoids.

Recall that a {\it Lie groupoid generalized morphism} $M_1 \stackrel{\varphi_1}{\leftarrow} X\stackrel{\varphi_2}{\rightarrow} M_2$ from $\Gamma_1 \toto M_1$ to $\Gamma_2\toto M_2$ consists
of a smooth manifold $X$, a left $\Gamma_1$-action and a right $\Gamma_2$-action on $X$ with anchor maps $\varphi_1$ and $\varphi_2$ respectively, such that the two actions commute and that $X$ is a right $\Gamma_2$-torsor, {\it i.e.} the right $\Gamma_2$-action on $\varphi_1:X\rightarrow M_1$ is principal.
 We will refer to anchor and multiplication maps as the structure maps of $X$.  A generalized morphism 
 where $X$ is also a left $\Gamma_1$-torsor, {\it i.e.} the left $\Gamma_1$-action on $\varphi_2:X\rightarrow M_2$ is principal is referred to as a \emph{Lie groupoid bitorsor} (see \cite{Hil-Ska}).

\begin{definition}\label{def_gen_VB_morph}
Let $V_1\toto E_1$ and $V_2\toto E_2$ be VB-groupoids over $\Gamma_1\toto M_1$ and $\Gamma_2\toto M_2$ respectively. 
A {\it  generalized  VB-morphism} (resp.  {\it VB-bitorsor} ) from $V_1$ to $V_2$ is a generalized  morphism (resp. a bitorsor) of Lie groupoids 
$\xymatrix{E_1 &Z\ar[l]_{\phi_1}\ar[r]^{\phi_2} & E_2}$
from  (resp. between) $V_1\toto E_1$ to $V_2\toto E_2$ such that
$Z$ is  a vector bundle over $X$, and $Z\to X$  is compatible with  the given
vector bundles $E_1\to M_1$ and $E_2\to M_2$ in the sense that 
 there are  vector bundle morphisms:
 \begin{equation}\label{diagram_generalized_VB_morphism}
\begin{tikzcd}[column sep=huge]
E_1  \arrow[d]&
  Z\arrow[l,swap,"\phi_1"]\arrow[r,"\phi_2"] \arrow[d]& E_2\arrow[d] \\
M_1  & X\arrow[r,swap,"\varphi_2"]\arrow[l,"\varphi_1"]& M_2 .
\end{tikzcd}
	 \end{equation}
\end{definition}

It is straightforward to check that generalized VB-morphisms induce on $M_1 \stackrel{\varphi_1}{\leftarrow} X\stackrel{\varphi_2}{\rightarrow} M_2$ the structure of generalized morphism from $\Gamma_1$ to $\Gamma_2$.

\begin{remark}
	Consider a generalized morphism as in Definition \ref{def_gen_VB_morph}.
	For all $v,v'$ in the same fiber of $ V_1 \to \Gamma_1$ and all $z,z'$ in the same fiber of $Z \to X$ such that $\phi_1(z) = s_{V_1}(v) $ and $\phi_1(z') = s_{V_1}(v')  $, the following identity holds
	\begin{equation}
	\label{eq:distr}
	(v+v')\cdot(z+z')=v \cdot z +  v'\cdot z'\;,
	\end{equation}
	and analogously for the right $V_2$-action.
\end{remark}
There is a natural equivalence relation on generalized VB-morphisms:

\begin{definition}
	Generalized VB-morphisms $E_1\leftarrow Z\rightarrow E_2$ and $E_1\leftarrow Z'\rightarrow E_2$ from $V_1$ to $V_2$ are said to be equivalent if there exists a $V_1-V_2$-biequivariant vector bundle isomorphism from 
	$Z$ to $Z'$.   
\end{definition}

 For disambiguation, VB-groupoid morphisms shall be referred to as \emph{strict} VB-groupoid morphisms, at least in this section. 
 As for Lie groupoids, a  VB-groupoid morphism $\Phi:V_1\rightarrow V_2$ induces a generalized VB-groupoid morphism  defined by
 $Z_\Phi=E_1\times_{ E_2} V_2 \rightarrow  M_1\times_{M_2}\Gamma_2$ with left $V_1$ and right $V_2$ actions given for every compatible
 $e_1\in E_1$ and $v_2,v_2'\in V_2$, respectively,  by
 \begin{equation}\label{strict_as_generalized}
 v_1 \cdot (e_1,v_2) = (t_{V_1}(v_1),\Phi(v_1)v_2) \, , \qquad (e_1,v_2) \cdot v_2'=(e_1,v_2 v_2')\;.\end{equation}

The following lemma contains the crucial technical result of this section:
 
 \begin{lemma}\label{lem:HomotopyMeansEqual}
 Let $V_1$ and $V_2$ be VB-groupoids as in \eqref{eq:florence}.
 VB-groupoid morphisms $\Phi$ and $\Psi: V_1 \to V_2$ are homotopic if and only if their induced  generalized VB-morphisms $Z_\Phi$ and $ Z_\Psi$ are equivalent. 
 \end{lemma}
 \begin{proof}  
 	Let $h: E_1 \to C_2 $ be an homotopy between the  VB-groupoid morphisms $\Phi,\Psi: V_1 \to V_2$.
 	Then, an explicit $V_1$-$V_2$ biequivariant vector bundle morphism from
 	 $Z_\Phi$ to $Z_\Psi$ is given by $T(e,v) = (e,h(e) \cdot 0^{V_2}+ v)$
 	 for all $ (e,v) \in E_1 \times_{E_2,  }  V_2$
 such that $\Phi(e)=t_{V_2}(v)$.
 	 Right $V_2$-equivariance is obvious.
Left $V_1$-equivariance can be checked as follows.
For any $v_1 \in V_1$ with $s_{V_1}(v_1)=e$, on one hand, we have
 	  \begin{equation}
 	  \label{eq:equiv_a}
 	   T (v_1 \cdot (e,v) )  = T (t_{V_1}(v_1) , \Phi(v_1) \cdot v ) =  (t_{V_1}(v_1) , h(t_{V_1}(v_1)) \cdot 0^{V_2} +  \Phi(v_1) \cdot v ), 
 	  \end{equation}
 	  while, on the other hand, we have
 	  \begin{equation}
 	  \label{eq:equiv_b}
 	   v_1 \cdot T  (e,v)  =  v_1 \cdot (e,h(e) \cdot 0^{V_2}+ v) = (t_{V_1}(v_1) ,
 	   \Psi(v_1) \cdot (h(e) \cdot 0^{V_2} +  v) ).
 	  \end{equation}
 	 Applying the  equation
 $ (h(e) \cdot 0^{V_2} +  v)^{-1} =v^{-1} + 0^{V_2} \cdot h(e)^{-1} $ to the right hand sides of  (\ref{eq:equiv_a}) and (\ref{eq:equiv_b}),
and   using (\ref{eq:distr}) and the relation
 	 $$ h(t_{V_1}(v_1))  \cdot 0^{V_2} \cdot 0^{V_2} \cdot h(e)^{-1}  =  h(t_{V_1}(v_1))  \cdot 0^{V_2} + 0^{V_2} \cdot h(e)^{-1}  ,$$
 	 we deduce that the left hand sides of (\ref{eq:equiv_a}) 
and (\ref{eq:equiv_b}) coincide if and only if the following relation holds:
 	  \begin{equation}
 	  \label{eq:equiv_C}
 	   \Psi(v_1)-\Phi(v_1) = h(t_{V_1}(v_1))\cdot 0^{V_2}+ 0^{V_2}\cdot h(s_{V_1}(v_1))^{-1} = J_h(v_1).
 	  \end{equation}
 	  This proves that $ T$ is an equivalence of generalized morphism of VB-groupoids.

Conversely, let $\Phi$ and $\Psi: V_1 \to V_2$ be  VB-groupoid morphisms,
and $T:Z_\Phi\rightarrow Z_\Psi$
 an equivalence of generalized 
morphism of VB-groupoids  over the identity
map of
 $M_1 \times_{M_2} \Gamma_2$. 
 	Since $T$ is left $V_1$-equivariant, the first component of any
 element $(e,v) \in E_1 \times_{E_2,  }  V_2$ coincides with the first
 component of its image under the map  $T$.
This implies that there exists a vector bundle morphism
 $T': E_1 \times_{E_2 } V_2 \to V_2 $ over the projection map 
 $M_1 \times_{M_2} \Gamma_2\to \Gamma_2$ 
 such that:
 		$$
 		T(e_1,v_2) = (e_1,T'(e_1,v_2)).
 		$$
Right $V_2$-equivariance implies that
$T'(e_1,v_2) $ and $v_2$ must have the same image under the source map
 $ s_{V_2}: V_2 \to E_2$.
Therefore there exists a vector bundle morphism 
 $H :E_1 \times_{E_2 } V_2 \rightarrow V_2|_{M_2}$ 
(over the natural projection map $M_1 \times_{M_2} \Gamma_2 \to M_2$),
 indeed valued in $C_2$, such that	
$$ T' (e_1,v_2) = H(e_1,v_2) \cdot 0^{V_2} + v_2 .  $$
 Again by the  right $V_2$-equivariance, we see that $H(e_1,v_2)$ should not
 depend on $v_2$.  Thus there is   a vector bundle morphism 
 $h :E_1  \rightarrow V_2|_{M_2}$
such that $H(e_1,v_2)= h(e_1)$. That is,
 $$T( e_1,v_2)=(e_1, h(e_1) \cdot 0^{V_2}+v_2).$$
 	Since the left hand sides of (\ref{eq:equiv_a}) and (\ref{eq:equiv_b})
 coincide if and only if Equation (\ref{eq:equiv_C}) holds, it follows that
 $h$ must be a homotopy between $\Phi$ and $\Psi$. 
\color{black}
 \end{proof}

The results described below are completely analogous to the Lie groupoid case and can be proved in the same way (see \cite{Hil-Ska}).
Let $E_1\leftarrow Z_1\rightarrow E_2$ and $E_2\leftarrow Z_2\rightarrow E_3$ be generalized VB-morphisms from $V_1\toto E_1$ to $V_2\toto E_2$ and from $V_2\toto E_2$ to $V_3\toto E_3$ respectively. Then 
$$
Z_1\circ Z_2 = \frac{Z_1\times_{E_2}Z_2}{V_{2}}
$$
where $V_2$ acts on $Z_1\times_{E_1} Z_2$ by $(z,z')\cdot v=(z \cdot v,v^{-1} \cdot z')$, together with the standard structure maps, defines a 
VB-groupoid generalized morphism from $V_1$ to $V_3$.
Composition is compatible with the equivalence, {\it i.e.} if $Z_1$ and
$Z_2$ are   equivalent to $Z_1'$ and  $Z_2'$, respectively, 
 then $Z_1 \circ Z_2$ is equivalent to $Z_1' \circ Z_2'$.

\begin{lemma}
	\label{lem:algebra}
\begin{enumerate}
	\item[(i)] Composition of generalized morphisms is associative up to equivalence, {\it i.e.} for any composable generalized morphisms $Z_1,Z_2,Z_3$, the compositions $ (Z_1 \circ Z_2) \circ Z_3$ and $ Z_1 \circ (Z_2 \circ Z_3)$ are equivalent.
\item[(ii)]  Let $V\toto E$ be a VB-groupoid; then $V$ together with the obvious structure maps is a generalized VB-morphism from $V$ to  itself. It is a neutral element with respect to the
 composition of  generalized VB-morphisms.
\end{enumerate}
\end{lemma}

We recall that $V$ being a neutral element means that the composition of $V$ with any generalized morphism $Z$  is equivalent to $Z$.

A generalized morphism $E\leftarrow Z\rightarrow F$ from $V \toto E$ to $W \toto F$ is said to be \emph{invertible} if there exists a generalized morphism  $F\leftarrow Z'\rightarrow E$ such that $Z \circ Z'$ is equivalent to  the neutral element $W $ and $Z' \circ Z$ is equivalent to the neutral element  $V$.
 Exactly as for Lie groupoids, we have the following result:
\begin{proposition}
	\label{prop:invertibleIsBitorsor}
A generalized  VB-groupoid morphism is invertible if and only if it is a VB-bitorsor.	
\end{proposition}

Two VB-groupoids related by a VB-groupoid bitorsor (or, equivalently, invertible generalized morphisms) are said to be \emph{Morita equivalent VB-groupoids}.
Let us list a few results about Morita equivalence of VB-groupoids.

\begin{proposition}
	\label{prop:equivRelations}
	\begin{enumerate}
		\item[(i)] 
		Morita equivalence  defines an equivalence relation among VB-groupoids. 
		\item[(ii)] 	VB-groupoids  are Morita equivalent if and only if their dual VB-groupoids are Morita equivalent.
	\end{enumerate}
\end{proposition}
\begin{proof}
The first assertion is a straightforward consequence of
 Proposition \ref{prop:invertibleIsBitorsor} and Lemma \ref{lem:algebra}. 
		For the second assertion,  
		let $ \xymatrix{E_1 &Z\ar[l]_{\phi_1} \ar[r]^{\phi_2}& E_2}$ be
 a $V_1$-$V_2$-bitorsor. For every $x \in X$ (the base manifold of $Z$) and
 $m \in M_1$ (the base manifold of $E_1$) with $ \varphi_1(x) = m$
 (where $\phi_1$ is over $\varphi_1:X\rightarrow M_1$),
 the $V_1$-action on $Z$ induces an injective linear map
 $ C_1|_m \hookrightarrow Z|_x$ defined as $c_m\rightarrow c_m\cdot 0_x^Z$.
 Dualizing this linear map, we obtain a vector bundle morphism $\psi_1 : Z^\vee \mapsto C_1 ^\vee$ which is a surjective submersion.
 We analogously obtain a vector bundle surjective submersion
 $\psi_2 : Z^\vee \mapsto C_2^\vee$.
Then  $ \xymatrix{C_1^\vee &Z^\vee\ar[l]_{\psi_1} \ar[r]^{\psi_2}& C_2^\vee}$
                is  a $V_1^\vee$-$V_2^\vee$ VB-bitorsor, where  $Z^\vee\to X$
is the dual vector bundle of $Z\to X$. To prove
this,  we denote by $\Lambda_1\subset V_1\times Z\times Z$, 
the graph of  the Lie groupoid $V_1$-action on $Z$.
It is simple to check that
 $\overline{\Lambda_1^\perp}=\{(\xi, w, -z)| (\xi, w, z)\in \Lambda_1^\perp\}$,
where $\Lambda_1^\perp$ denotes the annihilator of the graph $\Lambda_1$,
is again a graph that defines a left-action of $V^\vee_1$ on $Z^\vee$.
Similarly, we  obtain an right-action of  $V^\vee_2$ on $Z^\vee$.
One easily checks that
  $\xymatrix{C_1^\vee &Z^\vee\ar[l]_{\psi_1} \ar[r]^{\psi_2}& C_2^\vee}$
is indeed  a $V_1^\vee$-$V_2^\vee$ VB-bitorsor.
\color{black}
\end{proof} 
	
Below is a basic example of Morita equivalence.
	
\begin{proposition}
		\label{tangent_cotangent_morita}
		Let $\Gamma_1  $ and $\Gamma_2$ be Morita equivalent 
Lie groupoids, the tangent VB-groupoids $T\Gamma_1$ and $T\Gamma_2$ are
 Morita equivalent and so are the cotangent VB-groupoids $T^\vee\Gamma_1$ and $T^\vee\Gamma_2$.
		
Moreover, for  a $\Gamma_1-\Gamma_2$-bitorsor  $M_1\stackrel{}{\leftarrow} X \stackrel{}{\rightarrow} M_2$,
 $TM_1\stackrel{}{\leftarrow} TX \stackrel{}{\rightarrow} TM_2$   is a $ T\Gamma_1-T\Gamma_2$ VB-bitorsor and $A_1^\vee \stackrel{}{\leftarrow} T^\vee X \stackrel{}{\rightarrow} A_2^\vee$   is a $ T^\vee \Gamma_1-T^\vee \Gamma_2$VB-bitorsor.
	\end{proposition}
	\begin{proof}
		Let $M_1\stackrel{\varphi_1}{\leftarrow} X \stackrel{\varphi_1}{\rightarrow} M_2$  be a $\Gamma_1-\Gamma_2$ bitorsor. It is well-known that $TX$ is a $T\Gamma_1-T\Gamma_2$ bitorsor, with structure maps the tangent maps of the structure maps of $M_1\stackrel{\varphi_1}{\leftarrow} X \stackrel{\varphi_1}{\rightarrow} M_2$.
		These maps are vector bundle morphisms by construction, 
making $TX$ into  a $T\Gamma_1-T\Gamma_2$ VB-bitorsor.
The conclusion thus follows from Proposition \ref{prop:equivRelations} (ii).
	\end{proof}

\begin{definition}\label{bitorsor_equivalence_of_VB_morphism}			
	Let  $V_1\toto E_1$ and $V_2\toto E_2$ be VB-groupoids over $\Gamma_1$,
	and $W_1\toto F_1$ and $W_2\toto F_2$ be VB-groupoids over $\Gamma_2$ and let
  $\Phi_1: V_1 \to W_1$ and   $\Phi_2: V_2 \to W_2 $	be VB-groupoid morphisms. We will say that $\Phi_1$ and $\Phi_2$
   are \emph{equivalent VB-morphisms with respect to the bitorsors $Z$ and $Z'$}
if there exists a pair of VB-groupoid bitorsors $\xymatrix{E_1 &Z\ar[l]_{\phi_1}\ar[r]^{\phi_2} & E_2}$
	and $\xymatrix{F_1 &Z'\ar[l]_{\phi_1'}\ar[r]^{\phi_2'} & F_2}$
	such  that  $Z_{\Phi_2}\circ Z$ and $Z'\circ Z_{\Phi_1}$ are equivalent generalized morphisms.
\end{definition}

When this happens, we will diagrammatically denote it as:
		 \begin{equation}\label{diagram_VB_morphism_equivalence}
\begin{tikzcd}[column sep=huge,row sep=huge]
V_1 \arrow[r,"Z"] \arrow[d,swap,"Z_{\Phi_1}"]&
  V_2 \arrow[d,"Z_{\Phi_2}"] \\
W_1 \arrow[r,"Z^\prime"] \arrow[ur,Rightarrow]& W_2 
\end{tikzcd}
	 \end{equation}

Below we give an equivalent  description of   Morita equivalence of
 VB-groupoids. Let $V$ be a VB-groupoid and consider a vector bundle morphism:
		
		\begin{equation}
		\label{eq:SurjSubmVBundles}
		\xymatrix{  {\cal E} \ar[r]^{\phi} \ar[d]  & E \ar[d]   \\ X \ar[r]^{\varphi} & M,  }
		\end{equation}
		where both horizontal maps are surjective submersions.
		Consider the pull-back groupoid $ V[{\cal E}]: = {\cal E}\times_E V\times_E {\cal E}\toto{\cal E} $ of $V \toto E$ via $  {\cal E} \to E $,
		and the pull-back groupoid $\Gamma[X]:
		=X\times_M\Gamma\times_M X\toto X $ of $\Gamma \toto M$ via $  X \to M $.

\begin{proposition}\label{pull_back_VBgrpd}
	\begin{enumerate}
\item[(i)]  Then 
				\begin{equation}
				\label{eq:PBVBgrou}
				\VB{V[{\cal E}]}{{\cal E}}{\Gamma[X]}{X}
				\end{equation}
			is a VB-groupoid.
	\item[(ii)] The natural projection $\Phi_\phi$
				\begin{equation}
				\label{eq:PBVBgrouMorph}  
			\VB{V[{\cal E}]}{{\cal E}}{\Gamma[X]}{X}  \begin{array}{c}  \\  \\ \\ \longrightarrow \end{array} \VB{V}{E}{\Gamma}{M} 
				\end{equation} is a  VB-groupoid morphism.
\item[(iii)] The VB-generalized morphism associated to the  VB-groupoid
 morphism \eqref{eq:PBVBgrouMorph} is a VB-bitorsor.
	\end{enumerate}
	\end{proposition}

The VB-groupoid \eqref{eq:PBVBgrou} is called the
 \emph{pull-back VB-groupoid} of $V$ via $\phi: {\cal E} \to E$. By Proposition \ref{prop:invertibleIsBitorsor},
 the invertible VB-groupoid generalized morphism described
 in  Proposition \ref{pull_back_VBgrpd} (iii) will be denoted by ${\cal E} \leftarrow   Z_{\Phi_\phi} \rightarrow E $.
		
		\begin{remark}
			\normalfont
			A  VB-groupoid morphism $\Phi$ from $W\toto {\cal E}$ to $V \toto E$  that factors as
 the composition of a VB-groupoid isomorphism $W\stackrel{\sim}{\rightarrow}
 V[{\cal E}]$ with the natural projection 
 \eqref{eq:PBVBgrouMorph} corresponds to Morita morphisms of VB-groupoids, as introduced independently  in \cite{OrtizDelHoyo} in
			a different fashion.
			The  characterization of Morita equivalence of VB-groupoids  
			in terms of Morita morphisms goes exactly as for Lie groupoids  \cite{BehrendXu}: two VB-groupoids $V_1$ and $V_2$ are Morita equivalent if and only if there exists a VB-groupoid $W$ and  Morita morphisms of VB-groupoids $W\to V_1$ and $W\to V_2$.
		\end{remark}

Below	 are three additional important classes of examples
that  will be useful in  the future.

	\begin{example}
		\normalfont
		\label{ex:tangent}
	Let $ \Gamma \toto M$ be a Lie groupoid and $X \stackrel{\varphi}{\to} M$ a surjective submersion. Then the 
pull-back of the VB-groupoid $ T \Gamma \toto TM$ with respect	
	to $ T\varphi : TX \to TM$ is canonically isomorphic to the VB-groupoid $T \Gamma[X] \toto TX $.
	By  Proposition  \ref{pull_back_VBgrpd} (iii), it defines a $ T\Gamma[X]-T\Gamma $ VB-bitorsor,
  denoted by $ TX \stackrel{}{\leftarrow} Z_{\varphi}  \stackrel{}{\rightarrow} TM$,
 where we adopt  the simplified notation $Z_\varphi$ for $Z_{\Phi_{T\varphi}}$.
	\end{example}
	
\begin{example}
		\normalfont
		\label{ex:cotangent}
		Let $ \Gamma \toto M$ be a Lie groupoid with Lie algebroid $A$ and $X \stackrel{\varphi}{\to} M$ a groupoid right  action by $\Gamma$.
 The infinitesimal action of $A$ on $X$ yields a vector bundle morphism
 ${\mathfrak a}: \varphi^* A \to TX $.  The vector bundle morphism
 $({\mathfrak a},{\rm id}): \varphi^* A \stackrel{{\mathfrak a} \times id_A}{\longrightarrow} TX \times_{TM} A  \simeq A[X]$, is injective. Its dual is 
therefore a vector bundle morphism  $ {\mathfrak p}_\varphi : A[X]^\vee \to A^\vee $ which is a surjective submersion.
	The pull-back of the VB-groupoid $ T^\vee \Gamma \toto A^\vee$ with respect to $ {\mathfrak p}_\varphi: A[X]^\vee \to A^\vee $ is canonically isomorphic to the VB-groupoid $T^\vee \Gamma[X] \toto A[X]^\vee $. 
By Proposition \ref{pull_back_VBgrpd} (iii),  it defines
 a $ T^\vee \Gamma[X]-T^\vee \Gamma $ bitorsor,   denoted by 
	$ T^\vee X \stackrel{}{\leftarrow} Z_{\varphi}^\vee  \stackrel{}{\rightarrow} T^\vee M$.
	\end{example}

	\begin{example}
		\normalfont
		\label{ex:pullBack}
		Given a VB-groupoid $V\toto E$, and  a surjective submersion $ \varphi: X \to M$, the projection $ \varphi^* E \to E $ is a vector bundle morphism as in \eqref{eq:SurjSubmVBundles}.
 The resulting pull-back VB-groupoid shall be denoted as
 $\varphi^* V \toto \varphi^* E $. It is, by construction, a VB-groupoid over $ \Gamma[X]$.
\end{example}

\subsection{Homotopy and Morita equivalence}\label{VB_homotopy_and_Morita}

In this  subsection,
 we prove two propositions  about the relation between
 homotopy equivalence (see Definition \ref{def:homotopyEquivVbgroup}) and Morita equivalence of VB-groupoids, together with a study of the behavior of maps under such equivalences. Results of this subsection will be essential
 in understanding the behavior of homotopy $\Gamma$-modules under VB-groupoid Morita equivalence.
		
	\begin{remark}\normalfont
		\label{rmk:homotopIff}
		By Lemma \ref{lem:HomotopyMeansEqual}, a pair of VB-groupoid morphisms $\Phi_1 \colon V_1 \to V_2 $ and $\Phi_2 \colon V_2 \to V_1   $ form a homotopy equivalence if and only if $ Z_{\Phi_1} \circ Z_{\Phi_2} \simeq V_2 $ and $ Z_{\Phi_2} \circ Z_{\Phi_1} \simeq V_1 $.
		In particular $V_1$ and $V_2$ are Morita equivalent VB-groupoids and $Z_{\Phi_1} $ and $ Z_{\Phi_2}$ are bitorsors relating them.
	\end{remark}

		We now describe an important example of homotopy equivalence.
	 Let $V\toto E$ be a VB-groupoid and
	 consider two vector bundle morphisms:
	 \begin{equation}
	 \label{eq:SurjSubmVBundles2}
	 \xymatrix{  {\cal E}_1 \ar[r]^{\phi_1} \ar[d]  & E \ar[d]   \\ X \ar[r]^{\varphi} & M  } \hbox{  $\begin{array}{c}  \\ \\ \\ \hbox{ and } \end{array}$ }
	 \xymatrix{  {\cal E}_2 \ar[r]^{\phi_2} \ar[d]  & E \ar[d]   \\ X \ar[r]^{\varphi} & M,  }
	 \end{equation}
	 where all the horizontal maps are surjective submersions.
	 Using partitions of unity, one can construct vector bundle morphisms:
	   \begin{equation}
	   \label{eq:SurjSubmVBundles20} \xymatrix{  {\cal E}_1 \ar[rr]^{\psi^{(12)}} \ar[dr]_{\phi_1} & & {\cal E}_2 \ar[dl]^{\phi_2}   \\& E& } \hbox{  $\begin{array}{c}  \\ \\ \\ \hbox{ and } \end{array}$ }
	  \xymatrix{  {\cal E}_2 \ar[rr]^{\psi^{(21)}} \ar[dr]_{\phi_2} & & {\cal E}_1 \ar[dl]^{\phi_1}   \\ &E &} .\end{equation}
	  \begin{lemma}\label{ex:Moritaequivalence}
	  The VB-groupoid morphisms:
	 $$
	  \begin{array}{rrcl}  \Psi^{(12)} \colon &V[{\cal E}_1 ]  & \to  &  V[{\cal E}_2 ]  \\ & (e,v,e') & \mapsto & (\psi^{(12)} (e) , v, \psi^{(12)} (e')) \end{array} \hbox{  $\begin{array}{c}   \\ \hbox{ and } \end{array}$ } \begin{array}{rrcl}  \Psi^{(21)} \colon& V[{\cal E}_2 ]  & \to   &  V[{\cal E}_1 ]  \\ & (e,v,e') & \mapsto & (\psi^{(21)} (e) , v, \psi^{(21)} (e')) \end{array}
	 $$
	  form a homotopy equivalence of VB-groupoids:
	  \begin{equation}
	  \SelectTips{eu}{12}
	  \label{eq:homotopyEquivPullBack}
	  \xymatrix{V[{\cal E}_1] \ar@(ul,dl)[]_{} \ar@<-.5ex>[r]_{\Psi^{(12)}} & \ar@(ur,dr)[]^{}  \ar@<-.5ex>[l]_{ \Psi^{(21)}} V[{\cal E}_2]} .
	  \end{equation}
	  \end{lemma}
	\begin{proof}  
	 Since Diagrams \eqref{eq:SurjSubmVBundles20} commute, so do the following diagrams of  VB-groupoid morphisms:
	    $$ \xymatrix{  V[{\cal E}_1] \ar[rr]^{\Psi^{(12)}} \ar[dr]_{\Phi_{\phi_1}} & & V[{\cal E}_2] \ar[dl]^{\Phi_{\phi_2}}   \\& V& } \hbox{  $\begin{array}{c}  \\ \\ \\ \hbox{ and } \end{array}$ }
	    \xymatrix{  V[{\cal E}_2] \ar[rr]^{\Psi^{(21)}} \ar[dr]_{\Phi_{\phi_2}} & & V[{\cal E}_1] \ar[dl]^{\Phi_{\phi_1}}   \\ &V &} $$ 
   with $ \Phi_{\phi_1}$ and $\Phi_{\phi_2}$ as in  
Proposition \ref{pull_back_VBgrpd} (ii).
	    It follows from  Proposition \ref{pull_back_VBgrpd} (iii)  that the morphisms pointing downward correspond to bitorsors. In terms of generalized morphisms, the previous commutative diagrams read as follows:
	     $$ Z_{\Psi^{(12)}} = Z_{\Phi_{\phi_2}}^{-1}  \circ Z_{\Phi_{\phi_1}}  \hbox{ and } 
	     Z_{\Psi^{(21)}} = Z_{\Phi_{\phi_1}}^{-1}  \circ Z_{\Phi_{\phi_2}} .$$  
 As a consequence, $Z_{\Psi^{(12)} } \circ Z_{\Psi^{(21)}} = V_1 [{\cal E}_1]$ and 
	     $Z_{\Psi^{(21)} } \circ Z_{\Psi^{(12)}} = V_2 [{\cal E}_2]$.
	     By Lemma \ref{lem:HomotopyMeansEqual}, $\Psi^{(12)} \circ \Psi^{(21)}$ and $\Psi^{(21)} \circ \Psi^{(12)}$
	     are therefore homotopic to the identity. 
	\end{proof}

We can now state the first proposition, which uses the notations $ \varphi_1^* V_1,\varphi_2^* V_2$ of Example \ref{ex:pullBack}.

	\begin{proposition}
		\label{morita_VB_groupoids}
		Let  $V_1 \toto E_1$ and $V_2 \toto E_2$ be VB-groupoids over Lie groupoids $\Gamma_1 \toto M_1$ and $ \Gamma_2 \toto M_2$, respectively.
		Then  $V_1$ and $V_2$ are Morita equivalent  VB-groupoid if and only if there exist
		\begin{enumerate}
			\item[(i)] a $\Gamma_1 - \Gamma_2$ bitorsor  $\xymatrix{M_1 &X\ar[l]_{\varphi_1}\ar[r]^{\varphi_2} & M_2}$;
			\item[(ii)] an homotopy equivalence between the pull-back VB-groupoids $\PBVBG{V_1}{\varphi_1}{E_1}$ and $\PBVBG{V_2}{\varphi_2}{E_2}$:
			\begin{equation}
			\SelectTips{eu}{12}
			\label{eq:homotopyEquivPullBackneeded}
			\xymatrix{\varphi_1^* V_1 \ar@(ul,dl)[]_{} \ar@<-.5ex>[r]_{\Phi^{(12)}} & \ar@(ur,dr)[]^{}  \ar@<-.5ex>[l]_{ \Phi^{(21)}} \varphi_2^* V_2} 
			\end{equation}
		\end{enumerate}
	\end{proposition} 
\begin{remark}\label{base_pull_back_bitorsor}
{\rm  Since  the homotopy equivalence in
 Definition \ref{def:homotopyEquivVbgroup}  involves morphisms over the identity, it is convenient to think of
 the pull-back groupoids $\varphi_1^*V_1$ and $\varphi_2^*V_2$ above  as VB-groupoids over the same action groupoid $(\Gamma_1\times\overline{\Gamma}_2)\ltimes X\toto X$ that is canonically isomorphic to $\Gamma_1[X]$ and $\Gamma_2[X]$.}
\end{remark}
\begin{proof}
Assume that $V_1$ and $V_2$ are Morita equivalent VB-groupoids,
with  $\xymatrix{E_1 &Y\ar[l]_{\phi_1}\ar[r]^{\phi_2} & E_2}$  being
  a $V_1-V_2$ VB-bitorsor.
		Let $ X$ be the base manifold of $Y$ and $\varphi_1 \colon X \to M_1$,  $\varphi_2 \colon X \to M_2$
		be the base maps of $\phi_1 \colon Y \to E_1$ and $\phi_1 \colon Y \to E_2$, respectively.
It follows from Lemma \ref{ex:Moritaequivalence} that the VB-groupoid
 $ V_1[Y]$ is homotopy equivalent to $ \varphi_1^* V_1$,
 and likewise  $ V_2[Y]$ is homotopy equivalent to  $ \varphi_2^* V_2$.
By the definition of  $VB$-bitorsors, $V_1[Y]$ and $ V_2[Y]$ are
 isomorphic VB-groupoids, and are therefore  homotopy equivalent. 
		This implies that  $ \varphi_1^* V_1$ and $ \varphi_2^* V_2$
 are  homotopy equivalent. 
Let us denote by $Z_i:\varphi_i^*V_i\to V_i$, for $i=1,2$,
 the generalized VB-groupoid morphisms defined by $\varphi_i$, as in Example \ref{ex:pullBack}.
\color{black}
By construction, the following diagram is a commutative diagram of invertible generalized morphisms:
		\begin{equation} 
		\label{eq:proven} 
		\xymatrix{ V_1 \ar[r]^{Y}  \ar@{=>}[rd]& V_2 \\ \varphi_1^* V_1 \ar[u]^{Z_1}
			\ar[r]_{Z_{\Phi^{(12)}}} 
			& \varphi_2^* V_2  \ar[u]_{Z_2 } }.	
	\end{equation}		
		
Conversely, if there exist  data as in
 Proposition \ref{morita_VB_groupoids} (i)-(ii), then
 Remark  \ref{rmk:homotopIff} implies that
 $ \varphi_1^* V_1$ and  $ \varphi_2^* V_2$ are Morita equivalent.
 Since $ V_1$ and  $ \varphi_1^* V_1$, 
as well as $ V_2$ and  $ \varphi_2^* V_2$,
	are Morita equivalent according to 
Proposition \ref{pull_back_VBgrpd} (iii),
 it follows from  Proposition \ref{prop:equivRelations} (i) that $V_1$ and $V_2$ are Morita equivalent VB-groupoids.
	\end{proof}

Let  $V_1\toto E_1$, $W_1\toto F_1$  be VB-groupoids over $\Gamma_1$, and $V_2\toto E_2$, $W_2\toto F_2$ be VB-groupoids over $\Gamma_2$.
		
\begin{proposition}
	\label{prop:MoritaAndHomotMorphism}
			Morphisms of  VB-groupoids  $\Phi_1:V_1\to W_1$ and $\Phi_2:V_2\to W_2$ are equivalent  with respect to a bitorsor if and only if
 there exist 
				\begin{enumerate}
					\item[(i)] a $\Gamma_1 - \Gamma_2$ bitorsor  $\xymatrix{M_1 &X\ar[l]_{\varphi_1}\ar[r]^{\varphi_2} & M_2}$; and
					\item[(ii)]   homotopy equivalences
of the pull-back VB-groupoids between
 $\PBVBG{V_1}{\varphi_1}{E_1}$ and $\PBVBG{V_2}{\varphi_2}{E_2}$,  
				and 
between  $\PBVBG{W_1}{\varphi_1}{E_1}$ and $\PBVBG{W_2}{\varphi_2}{F_2}$:
					\begin{equation}
					\SelectTips{eu}{12}
					\label{eq:homotopyEquivPullBackneeded}
					\xymatrix{\varphi_1^* V_1 \ar@(ul,dl)[]_{} \ar@<-.5ex>[r]_{\Phi^{(12)}} & \ar@(ur,dr)[]^{}  \ar@<-.5ex>[l]_{ \Phi^{(21)}} \varphi_2^* V_2} \hbox{ ,}
					\xymatrix{\varphi_1^* W_1 \ar@(ul,dl)[]_{} \ar@<-.5ex>[r]_{\Psi^{(12)}} & \ar@(ur,dr)[]^{}  \ar@<-.5ex>[l]_{ \Psi^{(21)}} \varphi_2^* W_2} ,
					\end{equation} 
					\end{enumerate}
such that $\Psi^{(12)}\circ\varphi_1^*\Phi_1$ and $\varphi_2^*\Phi_2\circ\Phi^{(12)}$ are homotopic equivalent VB-groupoid morphisms
 from $\varphi_1^*V_1$ to $\varphi_2^*W_2$. 

Here the VB-groupoid morphisms	$\varphi_1^* \Phi_1 \colon \varphi_1^* V_1 \to \varphi_1^* W_1$ 
and $\varphi_2^* \Phi_2 \colon \varphi_2^* V_2 \to \varphi_2^* W_2$
	are the pull-backs of the VB-morphisms $\Phi_1\colon V_1 \to W_1$
and $\Phi_2 \colon V_2 \to W_2$, respectively. 	
\end{proposition}

		\begin{proof}		
If there exist data as in  Proposition 
			\ref{prop:MoritaAndHomotMorphism} (i)-(ii), 
then  we have a commutative diagram of VB-groupoid generalized morphisms as follows:

\[			
 \begin{tikzcd}[row sep=scriptsize, column sep=scriptsize]
V_1\arrow[dd,swap,"Z_{\Phi_1}"]&&\varphi_1^*V_1\arrow[ll,swap, "Z_1"]\arrow[rr,shift left, "Z_{\Phi^{(12)}}"]\arrow[dd,swap,"Z_{\varphi_1^*\Phi_1}"]
&& \varphi_2^*V_2\arrow[dd,"Z_{\varphi_2^*\Phi_2}"]
\arrow[rr,"Z_2"]&&V_2\arrow[dd,"{Z_{\Phi_2}}"]\\
&&&&&&\\
W_1&&\varphi_1^*W_1\arrow[ll, "Z_1'"]\arrow[rr,shift left, swap, "Z_{\Psi^{(12)}}"]
\arrow[uurr,Rightarrow]
&& \varphi_2^*W_2
\arrow[rr, swap,"Z_2'"]&&W_2\\
 \end{tikzcd}			
\]	
where $Z_i$ and $Z_i'$ denote the generalized morphisms associated to the pullback as in Example \ref{ex:pullBack}.

All horizontal arrows in the diagram are bitorsors in view of
Proposition \ref{pull_back_VBgrpd} (iii)
 and  Lemma \ref{lem:HomotopyMeansEqual}.	Therefore $\Phi_1$ and $\Phi_2$ are equivalent VB-groupoid morphisms with respect to a
bitorsor obtained by suitable composition of horizontal arrows, according to Definition \ref{bitorsor_equivalence_of_VB_morphism}.
									

Conversely, assume that  $\xymatrix{E_1 &Y\ar[l]_{\phi_1}\ar[r]^{\phi_2} & E_2}$
			and $\xymatrix{F_1 &Z\ar[l]_{\psi_1}\ar[r]^{\psi_2} & F_2}$
			are VB-groupoid bitorsors with respect to which
 $\Phi_1$ and $\Phi_2$ are equivalent VB-groupoid morphisms. 
			It is easy to show that they can be chosen to induce
 the same $\Gamma_1 - \Gamma_2$ bitorsor  $\xymatrix{M_1 &X\ar[l]_{\varphi_1}\ar[r]^{\varphi_2} & M_2}$,
  where $X$ is the base manifold of both vector bundles $Y$ and $Z$. According to Proposition \ref{morita_VB_groupoids}, there exists
a homotopy equivalence between $ \varphi_1^* W_1$ and $\varphi_2^* W_2 $
as in			\eqref{eq:homotopyEquivPullBackneeded}.

Consider the following commutative diagram of generalized morphisms where all horizontal arrows are invertible and the middle diagram is a commutative diagram of generalized VB-morphisms by hypothesis:
	
\begin{equation}\label{equivalence_VB_pullback_diagram2}			
  \begin{tikzcd}[row sep=scriptsize, column sep=scriptsize]
    \varphi_1^*V_1\arrow[dd,swap,"Z_{\varphi_1^*\Phi_1}"]\arrow[rr,"Z_1"]&&V_1\arrow[rr, "Y"]\arrow[dd,swap,"Z_{\Phi_1}"]
&& V_2\arrow[dd,"Z_{\Phi_2}"]&&\varphi_2^*V_2\arrow[ll,swap,"Z_2"]\arrow[dd,"{Z_{\varphi_2^*\Phi_2}}"]\\
&&&&&&\\
 \varphi_1^*W_1\arrow[rr,swap,"Z_{1}"]&&W_1\arrow[rr, swap,"Z"]\arrow[uurr,Rightarrow]
&& W_2&&\varphi_2^*W_2\arrow[ll,"Z_{2}"]
   \end{tikzcd}			
\end{equation}

The commutative diagram \eqref{eq:proven}  applied to the surjective submersions $\varphi_1$ and $\varphi_2$ implies the following equivalences of generalized VB-groupoid morphisms:
\[
 Z_{\Phi^{(12)}} \simeq  Z_{2}^{-1} \circ Y \circ Z_{1}\, , \qquad
			   Z_{\Psi^{(12)}} \simeq  Z_{2}^{-1} \circ Z \circ Z_{1}\, .
			   \]
Substituting in the horizontal arrows
 of \eqref{equivalence_VB_pullback_diagram2} implies

			 \begin{equation}
			 Z_{\Psi^{(21)}}\circ Z_{\varphi_1^*\Phi_1}\simeq Z_{\varphi_2^*\Phi_2}\circ Z_{\Phi^{(12)}}.
			 \end{equation}
This concludes the proof.
\end{proof}

From a pullback diagram as \eqref{eq:proven}, one can also prove that
	$$
		Z_{\Phi^{(21)}} \simeq  Z_{1}^{-1} \circ Y \circ Z_{2},
 \hbox{  } \hbox{ and  }Z_{\Psi^{(21)}} \simeq  Z_{2}^{-1} \circ Z \circ Z_{1}. 
		$$
It is then immediate to conclude that 
 $\varphi_1^*\Phi_1\circ \Phi^{(21)}$ and $\Psi^{(21)}\circ \varphi_2^*\Phi_2$
are also homotopic equivalent VB-groupoid morphisms from $\varphi_2^*V_2$ to $\varphi_1^*W_1$.

	\subsection{Morita equivalent quasi-Poisson groupoids}
 We can now state the main result of this section.
	 
	 \begin{theorem}
	 	\label{morphism_quasi_poisson2}
	 	Let $(\Gamma_1,\Pi_1,\Lambda_1)$ and $(\Gamma_2,\Pi_2,\Lambda_2)$ be
	 	Morita equivalent quasi-Poisson groupoids and let $M_1\leftarrow X\rightarrow M_2$ be a bitorsor as in Proposition \ref{prop:equiv.def.MEquasiPoisson}. The VB-groupoid morphisms
	 	 $\Pi^\#_1:T^\vee\Gamma_1\rightarrow T\Gamma_1$ and 
	 	$\Pi^\#_2:T^\vee\Gamma_2\rightarrow T \Gamma_2$ are equivalent with respect to the VB-bitorsors $T^\vee X$ and $T X$.
	 \end{theorem}
	More precisely, let $(\Gamma_1,\Pi_1,\Lambda_1)$ and $(\Gamma_2,\Pi_2,\Lambda_2)$
be Morita equivalent quasi-Poisson groupoids with respect to a  $\Gamma_1-\Gamma_2$  bitorsor $M_1\leftarrow X\rightarrow M_2$ as in Proposition \ref{prop:equiv.def.MEquasiPoisson}. Theorem \ref{morphism_quasi_poisson2} states that the following is a commutative diagram of generalized VB-groupoid morphisms, where
 $TX$ and $T^\vee X$ are the VB-bitorsors described in Proposition~\ref{tangent_cotangent_morita}:
	 
	  \begin{equation}\label{Equivalence_Poisson_VB_morphism}
\begin{tikzcd}[column sep=huge,row sep=huge]
T^\vee\Gamma_1 \arrow[r,"T^\vee X"] \arrow[d,swap,"Z_{\Pi_1^\sharp}"]&
 T^\vee \Gamma_2 \arrow[d,"Z_{\Pi^\sharp_2}"] \\
T\Gamma_1 \arrow[r,"TX"] \arrow[ur,Rightarrow]& T\Gamma_2 
\end{tikzcd}
	 \end{equation}
\begin{proof}[Proof of Theorem \ref{morphism_quasi_poisson2}]
Let $M_1\stackrel{\varphi_1}{\leftarrow} X \stackrel{\varphi_2}{\rightarrow} M_2$
 be a $\Gamma_1-\Gamma_2$ bitorsor as in
 Proposition \ref{prop:equiv.def.MEquasiPoisson}.
Then there is a natural isomorphism of pull back groupoids:
\begin{equation}
\label{eq:COVID19}
\Gamma_1[X]\stackrel{\sim}{\longrightarrow}  \Gamma_2[X].
\end{equation}
By Definition \ref{def:morita},
  there exist  twist equivalent quasi-Poisson structures $ (\Pi_1^X, \Lambda_1^X)$ and $ (\Pi_2^X, \Lambda_2^X)$ on $ \Gamma_1[X] \simeq \Gamma_2[X]$ such that the bivector fields $ \Pi_1^X$ and $ \Pi_2^X$ are
 projectable and project to
 $ \Pi_1$ and $ \Pi_2$, respectively,
 under the Morita morphisms $ \Gamma_1[X] \to \Gamma_1$ and $ \Gamma_2[X] \to \Gamma_2$, respectively.
This implies that the following diagrams are commutative as
  VB-groupoid morphisms:
		$$ 
		\xymatrix{ T^\vee \Gamma_1\ar[d]_{{ (\Pi_1)^\#}}   &\ar[l]  T^\vee (\Gamma_1[X]) \ar[d]^{{(\Pi^X_1)^\#}} \\  T \Gamma_1  & \ar[l]T (\Gamma_1[X])}   
		\xymatrix{  \\ \hbox{ and }  }
 		\xymatrix{ T^\vee (\Gamma_2[X])\ar[d]_{{( \Pi_2^X)^\#}}  \ar[r] &  T^\vee \Gamma_2 \ar[d]^{{(\Pi_2)^\#}} \\  T (\Gamma_2[X]) \ar[r] &  T \Gamma_2}    $$
 	where the horizontal maps are projections
 as in Proposition \ref{pull_back_VBgrpd} (ii).
Here $T^\vee (\Gamma_1[X])$ and $T^\vee (\Gamma_2[X])$ are 
identified with the  pull-back VB-groupoids as in Example  \ref{ex:cotangent},
while $T (\Gamma_1[X])$ and  $T (\Gamma_2[X])$
are   
identified with the  pull-back VB-groupoids as in 
Example  \ref{ex:tangent}.
It thus follows that the following diagrams of generalized 
 		VB-groupoid morphisms are commutative:
		$$ 
		\xymatrix{ T^\vee \Gamma_1\ar[d]_{Z_{ (\Pi_1)^\#}}   &\ar@{->}[l]_{Z_{\varphi_1}^\vee }  T^\vee (\Gamma_1[X]) \ar[d]^{Z_{(\Pi^X_1)^\#}} \\  T \Gamma_1  & \ar@{->}[l]^{Z_{\varphi_1}}T (\Gamma_1[X])}   
		\xymatrix{  \\ \hbox{ and }  }
		\xymatrix{ T^\vee (\Gamma_2[X])\ar[d]_{Z_{( \Pi_2^X)^\#}}  \ar@{->}[r]^{Z_{\varphi_2}^\vee} &  T^\vee \Gamma_2 \ar[d]^{Z_{(\Pi_2)^\#}} \\  T (\Gamma_2[X]) \ar@{->}[r]_{Z_{\varphi_2}} &  T \Gamma_2}    $$
		where  $ TX \stackrel{}{\leftarrow} Z_{\varphi_1}  \stackrel{}{\rightarrow} TM_1$ and 
 $ TX \stackrel{}{\leftarrow} Z_{\varphi_2}  \stackrel{}{\rightarrow} TM_2$ 
are as in Example \ref{ex:tangent}, and   $ A_1[X]^\vee \stackrel{}{\leftarrow} Z_{\varphi_1}^\vee  \stackrel{}{\rightarrow} A_1^\vee$
 and $ A_1[X]^\vee \stackrel{}{\leftarrow} Z_{\varphi_2}^\vee  \stackrel{}{\rightarrow} A_2^\vee$
are  as in Example \ref{ex:cotangent}.
All horizontal maps are invertible VB-groupoid generalized morphisms
according to Proposition   \ref{pull_back_VBgrpd} (iii).
		
		Since the quasi-Poisson structures  $ (\Pi_1^X, \Lambda_1^X)$ and $ (\Pi_2^X, \Lambda_2^X)$ are twist equivalent, $(\Pi_1^X)^\#$ and $(\Pi_2^X)^\#$ are homotopy equivalent 
according to  Example \ref{prop:TGtwist} so that
 $Z_{(\Pi_1^X)^\#}$ and $Z_{(\Pi_2^X)^\#}$ are equivalent
  by Lemma \ref{lem:HomotopyMeansEqual}. 
Therefore, we have  the following commutative diagram of generalized morphisms,
 where the horizontal arrows are bitorsors and 
 the horizontal isomorphisms in the middle square are those
induced by the isomorphism \eqref{eq:COVID19}:

\begin{equation}\label{tangent_cotangent_generalized_diagram}		 	
			\xymatrix{
		 		 T^\vee \Gamma_1\ar[d]_{Z_{ (\Pi_1)^\#}}   &\ar@{->}[l]_{Z_{\varphi_1}^\vee } 
				 T^\vee (\Gamma_1[X]) \ar[d]_{Z_{(\Pi^X_1)^\#}}\ar@{<=>}[r] &
		 	T^\vee (\Gamma_2[X])\ar[d]^{Z_{( \Pi_2^X)^\#}}  \ar@{->}[r]^{Z_{\varphi_2}^\vee} &  T^\vee \Gamma_2 \ar[d]^{Z_{(\Pi_2)^\#}}
		 		\\
		 		  T \Gamma_1  & \ar@{->}[l]^{Z_{\varphi_1}}T (\Gamma_1[X]) \ar@{<=>}[r] \ar@{=>}[ur]& T (\Gamma_2[X]) \ar@{->}[r]_{Z_{\varphi_2}} &  T \Gamma_2}   
\end{equation}				  

		 Theorem \ref{morphism_quasi_poisson2} then follows from the fact that the composition
	 	$ (Z_{\varphi_2}^\vee) \circ (Z_{\varphi_1}^\vee)^{-1}  $ is a $T^\vee \Gamma_1 - T^\vee \Gamma_2 $ bitorsor equivalent to $A_1^\vee \leftarrow T^\vee X\rightarrow A_2^\vee$, while
		 the composition
		 $ (Z_{\varphi_2}) \circ (Z_{\varphi_1})^{-1}  $ is a $T \Gamma_1 - T \Gamma_2 $ bitorsor equivalent to $TM_1 \leftarrow T X\rightarrow TM_2$.
	\end{proof}


\section{ 2-term complexes over a differentiable stack}

The aim of this section is to introduce the notion of 2-term complexes over 
a differentiable stack  and to show that it is essentially equivalent to  
 Morita equivalence classes of VB-groupoids.

For this purpose, we first recall the definition of homotopy $\Gamma$-modules over a given Lie groupoid $\Gamma$.
 We  present a dictionary between VB-groupoids over $ \Gamma$ and
 $2$-term homotopy $\Gamma$-modules, following \cite{Gracia-Saz-Mehta}. 
We then interpret several results on $VB$-groupoids established in the previous section
 in terms of homotopy $\Gamma$-modules.
In this way, we are led naturally to the category
of 2-term complexes over a given differentiable stack $\XX$,
and  obtain an efficient way of studying this category
in terms of VB-groupoids.

We shall use this material  in Section \ref{sec:COVID} to associate,
 to any (+1)-shifted Poisson structure on a differentiable stack,
 a morphism from its cotangent complex shifted by $+1$, to its tangent complex.

\subsection{Homotopy $\Gamma$-modules}

We recall in this subsection some standard materials
 from \cite{AriasAbadCrainic, Gracia-Saz-Mehta}.
For a Lie groupoid $\gpoidmm$, let  $(C^\bullet (\Gamma), \delta)$ denote the
Lie groupoid cohomology  cochain complex:
\[ 
C^0(\Gamma)\stackrel{\delta}{\longrightarrow} C^1(\Gamma) \stackrel{\delta}{\rightarrow}  C^2(\Gamma) \cdots   
\]
 where, for any $p\geq 0$,
$C^p(\Gamma): = C^\infty(\Gamma^{(p)})$,
 and $\Gamma^{(p)}$ denotes the manifold consisting of $p$-composable arrows
in $\gpoidmm$.
 We recall that $\delta f (\gamma) = f(s(\gamma)) - f(t(\gamma)) $  for all $f\in C^0(\Gamma) = C^\infty(M)$ and $\gamma \in \Gamma$,
and that for all $f\in C^p(\Gamma) = C^\infty(\Gamma^{(p)})$ and $(\gamma_0,\ldots,\gamma_p)\in \Gamma^{(p+1)}$:
\begin{eqnarray*}
(\delta f)(\gamma_0,\ldots,\gamma_p) &=& f(\gamma_1,\ldots\gamma_p) + \sum_{j=1}^p (-1)^j f(\gamma_0,\ldots,\gamma_{j-1}\gamma_j\ldots,\gamma_p) \cr 
& & + (-1)^{p+1} f(\gamma_0,\ldots,\gamma_{p-1})\;.
\end{eqnarray*}
There is also a natural multiplication, called the \emph{cup product},
 on $C^\bullet (\Gamma)$ given by
\[
(f \cup g)(\gamma_1,\ldots,\gamma_{p+q}) = f(\gamma_1,\ldots,\gamma_p)g(\gamma_{p+1},\ldots,\gamma_{p+q}),  
\]
for all $f\in C^p(\Gamma)$ and $g\in C^q(\Gamma)$. In this way,
$\big(C^\bullet (\Gamma),  \delta \big)$ becomes a differential 
algebra  (dga in short).

Let $\ec : =\bigoplus_{r \in {\mathbb Z}} E_r$ be a ${\mathbb Z}$-graded vector bundle over $M$,
 and let $C^q(\Gamma,E_r)=\Gamma((t^{(q)})^*E_r)$, where $t^{(q)}:\Gamma^{(q)}\rightarrow M$ 
is defined by $t^{(q)}(\gamma_1,\ldots,\gamma_q)=t(\gamma_1)$ for  $q>0$,
 and $t^{(0)}={\rm id}_M$. The ${\mathbb Z}$-graded vector space $C^\bullet (\Gamma,\ec)=\oplus_{p \in {\mathbb Z}} C^p(\Gamma,\ec)$, where 
 $C^p(\Gamma,\ec)=\oplus_{q+r=p} C^q(\Gamma,E_r)$
admits a right $C^\bullet(\Gamma)$-module structure defined,
 for any $\omega\in C^p(\Gamma,E_r)$ and $f\in
C^q(\Gamma)= C^\infty(\Gamma^{(q)})$, by
\[
(\omega\cdot f)(\gamma_1,\ldots,\gamma_{p+q}) = \omega(\gamma_1,\ldots,\gamma_p)f(\gamma_{p+1},\ldots,\gamma_{p+q})  \;.
\]
There is a natural isomorphism 
\begin{equation}
\label{eq:Tensors}
C^p(\Gamma,E_r) \simeq \Gamma(E_r) \otimes_{C^\infty(M)} C^p (\Gamma), 
\end{equation}
where $C^p(\Gamma) $ is seen as a $C^\infty(M)$-module with the help of the algebra morphism $(t^{(p)})^*: C^\infty(M) \hookrightarrow C^p (\Gamma)$.   
In particular, a $C^\bullet(\Gamma)$-linear map $\Phi : C^\bullet(\Gamma,\ec) \mapsto C^{\bullet+k}(\Gamma,\ec)$ of degree $k$
is entirely determined by its restriction to sections of $\ec$.

As in \cite{AriasAbadCrainic, Gracia-Saz-Mehta}, we define a 
homotopy $\Gamma$-module (also referred to as  a {\em representation up to homotopy}) as a pair $(\ec,D)$ with $\ec$ a ${\mathbb Z}$-graded
vector bundle over $M$ and $D$ a degree $+1$ operator $D: C^\bullet (\Gamma, \ec)
\to C^{\bullet+1} (\Gamma, \ec)$ satisfying the equation  $D^2=0$ and 
the \emph{Leibniz identity}
\begin{equation}
\label{eq:leibnizId}
D(\omega\cdot f)=(D\omega)\cdot f+(-1)^{|\omega|}\omega \cdot (\delta f)
\end{equation}
for all $\omega\in C^{\bullet} (\Gamma, \ec)$ and $f\in C^\bullet (\Gamma)$.
That is, $(\ec,D)$ is a dg right  module of the dga $(C^\bullet (\Gamma), \delta)$. 

When the graded vector bundle $\ec$ is concentrated in two consecutive degrees only, we shall speak of a {\em 2-term homotopy $\Gamma$-module}.

\smallskip
\begin{remark}\label{homotopy_gamma_mod_data}
	\cite{AriasAbadCrainic, Gracia-Saz-Mehta}
For a 2-term homotopy $\Gamma$-module $(\ec,D)$, $\ec$ is given by a pair $(C,E)$ of vector bundles over $M$ and the operator $D$ 
is determined by a triple $(\rho,R,\Omega)$ where
 \begin{equation}
  \label{eq:3datae}
\left\{ \begin{array}{l}
\hbox{ $\rho:C\rightarrow E$ is a vector bundle morphism over the identity of $M$}	 \\
   R= (R^C , R^E) \hbox{ with  $R^E \in \Gamma(s^*E^\vee\otimes t^*E)$  and  $ 	R^C\in\Gamma(s^* C^\vee \otimes t^* C) $  } \\
  \Omega\in \Gamma( (s^{(2)})^* E^\vee \otimes (t^{(2)})^* C)  
\end{array}
\right.
\end{equation}
 It is often convenient to consider $\rho: C \to E$ as a $2$-term complex,
denoted  $\rho: C[1] \to E$ to
indicate that $C$ is of degree $(-1)$ and $E$ is 
of degree $0$. In the sequel, we always adapt this degree
conversion unless specified. We also  consider both 
  $R^C$ and $R^E $ as  families of
linear maps  $R^C_\gamma : C_{s(\gamma)}  \to C_{t(\gamma)}$ and $R^E_\gamma : E_{s(\gamma)}  \to E_{t(\gamma)}$,  respectively,
 associated to any   $\gamma \in \Gamma$, while
  $\Omega$ as a family of linear map $\Omega_{\gamma_1,\gamma_2} :E_{s(\gamma_2)}
  \to C_{t(\gamma_1)}$  associated  to any  $(\gamma_1,\gamma_2) \in \Gamma^{(2)}$.

The condition $D^2=0$ imposes several constraints, whose meaning we write on the right column:
$$\left\{ \begin{array}{lll}
 R^E_\gamma \circ \rho = \rho \circ R^C_\gamma & \forall \gamma\in\Gamma & \hbox{``The pair $R_\gamma := (R^E_\gamma,R^C_\gamma)$ is a chain map}  \\ & &  \hbox{from $C_{s(\gamma)} \stackrel{\rho}{\rightarrow} E_{s(\gamma)} $ to 
  $C_{t(\gamma)}\stackrel{\rho}{\rightarrow} E_{t(\gamma)} $}" \\
 R^E_{\gamma_1\gamma_2} - R^E_{\gamma_1}\circ R^E_{\gamma_2}  = \rho\circ\Omega_{\gamma_1,\gamma_2} & \forall (\gamma_1,\gamma_2)\in \Gamma^{(2)} & \hbox{ ``$\Omega_{\gamma_1,\gamma_2}$ is a homotopy  between}\\
  R^C_{\gamma_1\gamma_2} - R^C_{\gamma_1}\circ R^C_{\gamma_2}   = \Omega_{\gamma_1,\gamma_2}\circ\rho & & \hbox{the chain maps $R_{\gamma_1} \circ R_{\gamma_2} $ and $R_{\gamma_1 \gamma_2}$"} \\
   \Omega_{\gamma_1\gamma_2,\gamma_3}- \Omega_{\gamma_1,\gamma_2}\circ R^E_{\gamma_3} &  \forall (\gamma_1,\gamma_2,\gamma_3)& \hbox{ ``Both natural homotopies between the chain maps }\\
\hspace{1cm}   =  \Omega_{\gamma_1,\gamma_2\gamma_3}- R^C_{\gamma_1}\circ\Omega_{\gamma_2,\gamma_3} & 
    \hspace{0.5cm}\in \Gamma^{(3)} &  \hbox{  $ R_{\gamma_1} \circ R_{\gamma_2} \circ R_{\gamma_3} $ and $ R_{\gamma_1\gamma_2\gamma_3}$ are equal"} 
\end{array}
\right.$$

The \emph{dual of a 2-term homotopy $\Gamma$-module $ (\ec,D)$} is the
2-term homotopy $\Gamma$-module $(\ec^\vee,D^\vee)$ obtained by dualizing
all the data in Remark \ref{homotopy_gamma_mod_data}.
More precisely, if $\ec$ is concentrated in degrees $k$ and $k+1$ with
 $\ec_k = C$ and $\ec_{k+1} = E$, then $\ec^\vee$ is concentrated in degrees
 $-k-1$ and $-k$ with  $\ec^\vee_{-k-1} = E^\vee$ and $\ec_{-k}=C^\vee$.

 The data that correspond to $D^\vee$ are given by  $\rho^\vee: E^\vee \to C^\vee$,
the dual of $\rho:C \to E$,  and  $ (R^{E^\vee},R^{C^\vee},\Omega^\vee) $ where:
$$ R_{\gamma}^{E^\vee} = (R_{\gamma^{-1}}^E)^\vee,  R_\gamma^{C^\vee} =
 (R_{\gamma^{-1}}^C)^\vee  \hbox{ and }  \Omega^\vee_{\gamma_1,\gamma_2} = (\Omega_{\gamma_2^{-1}, \gamma_1^{-1}})^\vee 
$$ 
for all $\gamma \in \Gamma$ and $(\gamma_1,\gamma_2) \in \Gamma^{(2)}$.

Let $(\ec,D)$ be a $2$-term homotopy $\Gamma$-module. 
The \emph{$k$-shifted $2$-term homotopy $\Gamma$-module $(\ec[k],D)$}, is the $2$-term homotopy $\Gamma$-module which has the same differential $D$, but for which the degree of $\ec$ is shifted by $-k$ (that is,  elements of degree $i$ in $\ec[k]$ are 
those  of degrees $k+i$ in $\ec$). 
\end{remark}

Let  $(\ec,D)$ and $(\ec',D')$ be homotopy
$\Gamma$-modules.	    
A  morphism of homotopy $\Gamma$-modules
from $\ec$ to $\ec'$ is a  $C^{\bullet} (\Gamma)$-linear chain map
\begin{equation}
\label{eq:homotopy-mor}
\Phi:  C^\bullet (\Gamma, \ec) \to  C^\bullet (\Gamma, \ec').
\end{equation}

Let $\Gamma \toto M$ be a Lie groupoid.
For any surjective submersion $\varphi:X\rightarrow M$ and any 2-term
homotopy $\Gamma$-module $(\ec,D)$,
there is a natural  2-term
homotopy $\Gamma[X]$-module $(\varphi^* \ec , \varphi^* D)$
obtained by pulling
back the graded bundle $\ec$ along $\varphi$ and all the data in Remark \ref{homotopy_gamma_mod_data} by the Morita morphism $\phi_{\varphi}:\Gamma[X]\rightarrow\Gamma$.
This structure shall be called \emph{pull back 2-term homotopy $\Gamma[X]$-module}.
For every morphism of  2-term 
homotopy $\Gamma$-module $ \Psi:(\ec_1,D_1) \to (\ec_2,D_2) $,
there is also a natural  pull back morphism $\varphi^* \Psi : (\varphi^* \ec_1,\varphi^* D_1) \to (\varphi^* \ec_2,\varphi^* D_2) $.

Morphisms of homotopy $\Gamma$-modules $ \Phi$ and $\Psi:  C^\bullet (\Gamma, \ec_1) \to  C^\bullet (\Gamma, \ec_2)$ are said to be \emph{homotopic} 
if  there exists a $C^\bullet(\Gamma)$-linear degree $-1$ map $H:C^{\bullet}(\Gamma,\ec_1)\rightarrow C^{\bullet-1}(\Gamma,\ec_2)$ such that
$$
\Phi-\Psi = D_2\circ H + H\circ D_1\;.
$$
Composition of homotopy $\Gamma$-module morphisms respects homotopies. This allows
us  to define the following:  homotopy $\Gamma$-modules
$\ec_1$ and $\ec_2$  are {\it homotopy equivalent} if there exist morphisms
of  homotopy $\Gamma$-modules
 $\Phi:C(\Gamma,\ec_1)\rightarrow C(\Gamma,\ec_2)$ and
	$\Psi:C(\Gamma,\ec_2)\rightarrow C(\Gamma,\ec_1)$ such that the morphisms $\Phi\circ\Psi$ and $\Psi\circ\Phi$ are homotopic to the identity. 

 In view of the isomorphism \eqref{eq:Tensors}, by $C^\bullet(\Gamma)$-linearity, a morphism of homotopy $\Gamma$-modules $\Phi :  C^\bullet (\Gamma, \ec) \to  C^\bullet (\Gamma, \ec')$ is  determined by its restriction to sections of $\ec $ over the manifold $ M$.
In particular, for 2-term homotopy $\Gamma$-modules, a morphism $\Phi$ is determined
by a pair $(\phi, \mu)$,  where $\phi$, called the {\it linear term},
consists of   a pair $ (\phi^C ,\phi^E)$,  
with  $\phi^C : C \to C'$  and $\phi^E : E \to E'$ being
vector bundle morphisms over the identity of $M$, 
 and $\mu$ is a section of $ s^* E^\vee \otimes t^* C'$.
The latter can be considered as a family of
  linear maps $\mu_\gamma : E_{s(\gamma)} \to C_{t(\gamma)}'$
  associated to any $\gamma \in \Gamma$.
 A homotopy between two  morphisms of homotopy $\Gamma$-modules can, therefore, be constructed by a vector bundle morphism $ h : E \to C'$.
 
 \begin{remark}
 Since $\Phi$ is a chain map,  the pair $(\phi, \mu)$ must satisfy several constraints, 
 that we now spell out in terms of the data $(\rho,R=(R^C,R^E),\Omega)$
 and $(\rho',R'=(R^{C'},(R^{E'}),\Omega')$ associated to $2$-term homotopy $\Gamma$-modules $ C[1] \stackrel{\rho}{\rightarrow} E$ and $
 C'[1] \stackrel{\rho'}{\rightarrow} E' $, respectively below:
 \begin{equation*}
 \left\{ 
 \begin{array}{ll}
 \phi^E \circ \rho = \rho' \circ \phi^C &  \hbox{ ``$\phi$ is a chain map
 from $C[1] \stackrel{\rho}{\rightarrow} E $ to $ C'[1]\stackrel{\rho'}{\rightarrow} E'$"}  \\  
 R_\gamma^{C'}  \circ \phi^C - \phi^C \circ R_\gamma^C = \mu_\gamma \circ \rho_{s(\gamma)}  & \hbox{\lq\lq$\mu_\gamma$ is a homotopy between } \\
 R_\gamma^{E'}  \circ \phi^E - \phi^E \circ R_\gamma^E = \rho_{t(\gamma)} \circ \mu_\gamma     & \hspace{0.5cm}\hbox{ the chain maps  $\phi \circ R_\gamma$ and $R_\gamma' \circ \phi $"} \\
 \mu_{\gamma_1\gamma_2} - \phi^C \circ \Omega_{\gamma_1,\gamma_2} - \Omega_{\gamma_1,\gamma_2}' \circ \phi^E  & \hbox{``Both natural homotopies between} \\
\hspace{1cm} = \mu_{\gamma_1} \circ R_{\gamma_2}^E +  R_{\gamma_1}^{C'} \circ  \mu_{\gamma_2}  &
\hspace{0.5cm}\hbox{$\phi \circ R_{\gamma_1} \circ  R_{\gamma_2}$ and $ R_{\gamma_1} \circ  R_{\gamma_2} \circ \phi $ are equal"} 
 \end{array}
 \right.
 \end{equation*}
  \end{remark}

A morphism $(\phi, \mu)$ of homotopy $\Gamma$-modules
 $\Phi :  C^\bullet (\Gamma, \ec) \to  C^\bullet (\Gamma, \ec')$ for which $\ec =\ec'$,  $\phi={\rm id}$ and $\mu_{\epsilon(m)}=0$ for all $m \in M$ is said to be a \emph{gauge transformation} 
 \cite{Gracia-Saz-Mehta}.
Note that a gauge transformation $\Phi: C^\bullet (\Gamma, \ec) \to
  C^\bullet (\Gamma, \ec)$ is an invertible $C^\bullet(\Gamma)$-linear map.
 The inverse of $({\rm id},\mu)$ is the gauge transformation $({\rm id},-\mu)$.
  Whenever two homotopy $\Gamma$-modules are transformed one into the other by
 a gauge transformation,  we will call them \emph{gauge equivalent}.

\medskip

\subsection{From VB-groupoids to 2-term homotopy $\Gamma$-modules}
\label{dictionary}

Let us introduce the VB-groupoid cohomology, following \cite{Gracia-Saz-Mehta}.
 
Let $V\toto E$ be a VB-groupoid over $\Gamma\toto M$ with core $C$ as in \eqref{eq:florence}. 
We define a graded vector space $  \oplus_{p \geq 0} C_{VB}^p(V)$ as follows.
For $p=0$, we define $C^{0}_{VB}(V)$ to be the space $\Gamma(C)$ of sections of the core $C \to M$. 
For $p \geq 1$, we define $C_{VB}^p (V)$ to be the space of those sections
 $\sigma \in \Gamma((\pi^{(p)})^* V)$ satisfying

\begin{equation}\label{left_projectability_2}
s_V(\sigma(\gamma_1,\ldots,\gamma_p))= s_V(\sigma(1_{s(\gamma_1)},\gamma_2,\ldots,\gamma_p)),
\end{equation}
where
$\pi^{(p)}: \Gamma^{(p)} \to \Gamma$ is  the projection
 $$ \pi^{(p)} : (\gamma_1, \dots, \gamma_{p}) \mapsto \gamma_1 ,$$
for all $(\gamma_1,\ldots,\gamma_p)\in \Gamma^{(p)}$. 
The graded vector space
$C^\bullet_{VB}(V)$ has the structure of right $C^\bullet(\Gamma)$-module defined,
 for  any $p \geq 1$, by
\[
(\sigma\star f)(\gamma_1,\ldots,\gamma_{p+q}) = \sigma(\gamma_1,\ldots,\gamma_p) f(\gamma_{p+1},\ldots,\gamma_{p+q}),
\]
for all $\sigma\in C^{p}_{VB}(V)$, $f\in C^\infty(\Gamma^{(q)})$ and $(\gamma_1,\ldots,\gamma_{p+q})\in \Gamma^{(p+q)}$.
For $p=0$, the module structure is defined by
\[
( \sigma \star f) \, (\gamma_1,\ldots,\gamma_{q}) \, =  \, \sigma (t (\gamma_1))  \cdot 0_{\gamma_1}^\vee \,  f(\gamma_{1},\ldots,\gamma_{q}),
\]
for all $\sigma\in C^{0}_{VB}(V)=\Gamma(C)$, $f\in C^\infty(\Gamma^{(q)})$ and $(\gamma_1,\ldots, \gamma_{q})\in \Gamma^{(q)}$.

In order to turn  $C_{VB}^\bullet (V)$ into a complex, we  consider it as a subcomplex
 of  the Lie groupoid cohomology cochain
 complex $(C^\bullet(V^\vee),\delta_{V^\vee})$
 of the dual VB-groupoid $V^\vee\toto C^\vee$ (as defined in (\ref{eq:perugia})). 

\begin{proposition}
\label{VB-cochains}
Let $V\toto E$ be a VB-groupoid over $\Gamma\toto M$ with core $C$ as in \eqref{eq:florence}. 
Let $i:C^p_{VB}(V)\hookrightarrow C^p(V^\vee)$ be the linear map
 defined, when $p \geq 1$, by
\[
i(\sigma)(\eta_1,\ldots,\eta_p) =
 \langle \eta_1,\sigma(\gamma_1,\ldots,\gamma_p)\rangle
\]
for all compatible $\eta_1{\in}V^\vee_{\gamma_1} \, , \dots, \, \eta_p{\in}V^\vee_{\gamma_p}$,
 and,  when $p=0$,
 by $i(\sigma)(\alpha)=\langle \alpha,\sigma(m)\rangle$ 
 for all $\alpha\in C^\vee_m$. Then $(C^\bullet_{VB}(V), \delta_{V^\vee})$
is a subcomplex of $(C^\bullet(V^\vee),\delta_{V^\vee})$.

Moreover, the restriction of the coboundary differential $\delta_{V^\vee}$ applied to $\sigma\in C^p_{VB}(V)$ reads, when $p \geq 1$,
\begin{eqnarray}\label{VB_coboundary_p}
(\delta_{V^\vee}\sigma )(\gamma_0,\ldots,\gamma_p) &=& -\sigma(\gamma_0\gamma_1,\ldots,\gamma_p)\cdot \sigma(\gamma_1,\ldots,\gamma_p)^{-1} + \sum_{i=2}^p (-1)^i \sigma(\gamma_0,\ldots,\gamma_{i-1}\gamma_i,\ldots,\gamma_p) \cr
& & + (-1)^{p+1} \sigma(\gamma_0,\ldots,\gamma_{p-1}),
\end{eqnarray}
and when $p=0$ 
\begin{equation}\label{VB_coboundary_0}
\delta_{V^\vee}(\sigma)(\gamma) = -0_\gamma\cdot\sigma(s(\gamma))^{-1} -\sigma(t(\gamma))\cdot 0_\gamma\;, 
\end{equation}
for any $\sigma\in C^0_{VB}(V)=\Gamma(C)$ and $\gamma\in\Gamma$. 
\end{proposition}
\begin{proof}
The first statement is the content of Proposition 5.5 of \cite{Gracia-Saz-Mehta}. Formulas (\ref{VB_coboundary_p}-\ref{VB_coboundary_0}) follows from a direct computation. 
\end{proof}

\smallskip
Since we have  the projection map $ V^\vee \to \Gamma$,
  $C^\bullet(V^\vee)$ is clearly  a  $C^\bullet(\Gamma)$-module.
 It is routine to check that $C^\bullet_{VB}(V)$
is a $C^\bullet(\Gamma)$-submodule. 
By construction, $\big(C^\bullet(V^\vee),  \delta_{V^\vee}\big)$ 
 is a dg right  module of the dga $(C^\bullet (\Gamma), \delta)$.
%
\color{black}

\begin{lemma}\label{VB_chain_homotopy}
	Let $V_1$ and $V_2$ be VB-groupoids as in \eqref{eq:VB:samebase}, with
 cores $C_1$ and $C_2$, respectively.
	\begin{enumerate}
\item[(i)] Assume that  $\Phi:V_1\rightarrow V_2$ 
is a VB-groupoid morphism over 
${\rm id}:\Gamma\rightarrow\Gamma$. Then $\hat\Phi:C^\bullet_{VB}(V_1)
\rightarrow C^\bullet_{VB} (V_2)$
defined by $\hat{\Phi}(\sigma)=\Phi\circ \sigma$ for all
 $\sigma\in C_{VB}^\bullet(V_1)$, is a cochain map
 and a right $C^\bullet(\Gamma)$-module morphism.
\item[(ii)] Assume that $\Phi$ and $\Psi:V_1\rightarrow V_2$ 
are  homotopic VB-groupoid morphisms  with homotopy $h:E_1\rightarrow C_2$.
 Then the chain maps
$\hat{\Phi}$ and $\hat{\Psi}$ are homotopic with homotopy being the $C^\bullet(\Gamma)$-linear morphism 
$\hat{h}:C^{p+1}_{VB}(V_1)\rightarrow C^{p}_{VB}(V_2)$ defined as
\begin{equation}\label{VB_homotopy}
\hat{h}(\sigma)(\gamma_1,\ldots,\gamma_p) = - h(s_{V_1}(\sigma(1_{t(\gamma_1)},\gamma_1,\ldots,\gamma_p)))\cdot 0_{\gamma_1},
\end{equation}
$\forall (\gamma_1,\ldots,\gamma_p)\in \Gamma^{(p)}$.
\item[(iii)] Assume that the VB-groupoids $V_1$ and $V_2$ are homotopy equivalent, then  so are $  (C^\bullet_{VB} (V_1), \delta_{V^\vee_1})$ and  $  (C^\bullet_{VB} (V_2), \delta_{V^\vee_2})$.
\end{enumerate}
\end{lemma}
\begin{proof}  Assertion (i) is obvious.
To prove  (ii), by $C^\bullet(\Gamma)$-linearity, it suffices to check
 this for $\sigma$  in $C^0_{VB}(V)$ and $C^1_{VB}(V)$.
 For $\sigma\in C^0_{VB}(V)$, we have,   for all $ m \in M$,
\begin{eqnarray*}
(\delta_{V^\vee_2}\hat{h}+\hat{h}\delta_{V^\vee_1})(\sigma)(m) &=& \hat{h}(\delta_{V^\vee_1}\sigma)(m)= - h(s_{V_1}(\delta_{V^\vee_1}\sigma(1_m))) = h(s_{V_1}(\sigma(m)^{-1}))\cr
&=& h(t_{V_1}(\sigma(m))) = J_h(\sigma)(m)\;,
\end{eqnarray*}
where, in the third equality of the first line,
 we used (\ref{VB_coboundary_0}) and,  in the second line,
 we used (\ref{eq:homotopy3}) and the fact that
$0_{1_m} = 1_{0_m}$. 

Let $\sigma\in C^1_{VB}(V)$; we have,  for all $ \gamma \in \Gamma$,
\begin{eqnarray*}
(\delta_{V^\vee_2}\hat{h}+\hat{h}\delta_{V^\vee_1})(\sigma)(\gamma) &=&  \delta_{V^\vee_2}(\hat{h}(\sigma))(\gamma) + \hat{h}(\delta_{V^\vee_2}\sigma)(\gamma)\cr 
& =& - 0_\gamma\cdot\hat{h}(\sigma)(s(\gamma))^{-1} - \hat{h}(\sigma)(t(\gamma))\cdot 0_\gamma - h(s_{V_1}(\delta_{V^\vee_2}\sigma(1_{t(\gamma)},\gamma)))\cdot 0_\gamma\cr
&=& 0_\gamma\cdot h(s_{V_1}(\sigma(1_{s(\gamma)})) )^{-1} + h(s_{V_1}(\sigma(1_{t(\gamma)})))\cdot 0_\gamma \cr
& & - h(s_{V_1}(\sigma(1_{t(\gamma)}) -\sigma(\gamma)\cdot\sigma(\gamma)^{-1}))\cdot 0_\gamma\cr
&=& 0_\gamma\cdot h(s_{V_1}(\sigma(1_{s(\gamma)})) )^{-1} + h(t_{V_1}(\sigma(\gamma)))\cdot 0_\gamma\cr
&=& 0_\gamma\cdot h(s_{V_1}(\sigma(\gamma)) )^{-1} + h(t_{V_1}(\sigma(\gamma)))\cdot 0_\gamma = J_h(\sigma)(\gamma),
\end{eqnarray*}
where, in the last line, we used the defining property (\ref{left_projectability_2}) of VB-cochains. 

Assertion (iii) now follows immediately from (i) and (ii).
\end{proof}

Recall that right decompositions of VB-groupoids are defined  following
 Equation \eqref{VB_exact_sequence} in Section 
\ref{sec:homot_equiv}. 

\begin{lemma}
	\label{lem:isomChoiceRD}
	Let $V \toto E$ be a VB-groupoid as in \eqref{eq:florence} with core $C$.
Every right decomposition  
defines an isomorphism of $C^\bullet(\Gamma)$-modules
between $ C^\bullet_{VB}(V)$ and $C^{\bullet-1}(\Gamma, C[1] \oplus E) $.
\end{lemma}
\begin{proof}
	Let us fix a right decomposition of $V\toto E$. The induced isomorphism
$V \simeq  t^*C \oplus s^*E$ as  vector bundles  over $\Gamma$
	allows us to decompose $ \sigma \in C^p_{VB}(V)$ as a sum $ \sigma = \sigma_C + \sigma_E  $ where for all ${\gamma_1, \dots, \gamma_p} \in \Gamma^{(p)}$:
$$  \left\{ \begin{array}{rcl}
\sigma_C ({\gamma_1, \dots, \gamma_p}) & \in & C_{t(\gamma_1)} \\
\sigma_E ({\gamma_1, \dots, \gamma_p})  & \in & E_{s(\gamma_1) =t(\gamma_2)}
\end{array}\right.$$

By construction, $\sigma_C  $ is a section of $(t^{(p)})^* C $, i.e. it belongs to
 $C^p(\Gamma;C) \subset C^{p-1}(\Gamma, C[1] \oplus E)$. 
Condition (\ref{left_projectability_2}) implies that $\sigma_E(\gamma_1,\ldots,\gamma_p)$ 
does not depend on $\gamma_1$,
it can therefore be identified with a section $\hat{\sigma}_E$ of $(t^{(p-1)})^*E$,
 i.e. it belongs to $C^{p-1}(\Gamma;E) \subset  C^{p-1}(\Gamma, C[1] \oplus E)$.

One can check that  the map
$\sigma\rightarrow (\sigma_C,\hat{\sigma}_E)$ is an isomorphism,
and therefore  identifies
$C^p_{VB}(V)$ with $C^{p-1}(\Gamma,C[1]\oplus E)$, which 
 is also  $C^\bullet(\Gamma)$-linear by construction.
\end{proof}

For every VB-groupoid $V \toto E$ with core $C$, the $C^\bullet(\Gamma)$-module isomorphism described in Lemma \ref{lem:isomChoiceRD} allows us
 to transfer the differential $\delta_{V^\vee}$ defined in (\ref{VB_coboundary_p}-\ref{VB_coboundary_0}) to
a differential $D_{V}$ on $C^\bullet(\Gamma, C[1] \oplus E)$. Thus
 $(C[1] \oplus E, D_{V})$  becomes a $2$-term homotopy $\Gamma$-module.
The associated map $\rho : C[1] \to E$ as in \eqref{eq:3datae} is easily seen to be the core-anchor of the VB-groupoid $V$. 

The backward construction is given in \cite{Gracia-Saz-Mehta} by verifying that the data listed in Remark \ref{homotopy_gamma_mod_data} induce a VB-groupoid structure on $t^*C \oplus s^* E $, referred to as a \emph{split VB-groupoid}. This gives the following:

\begin{proposition}\label{VBgroupoid_2termrep}
	\cite{Gracia-Saz-Mehta} 
Let $\Gamma \toto M$ be a Lie groupoid.
\begin{enumerate}
\item[(i)] There is a  one-to-one correspondence between VB-groupoids
 over $\Gamma$ equipped with  right-decompositions and
 2-term homotopy $\Gamma$-modules.
\item[(ii)] Different
 right-decompositions of a VB-groupoid over $\Gamma$ induce {\it gauge equivalent}
 2-term homotopy $\Gamma$-modules. 
\end{enumerate}
\end{proposition}

Since gauge morphisms are invertible morphisms,
 Lemma \ref{VB_chain_homotopy} implies the following:

\begin{lemma}\label{VB_chain_homotopy_2}
	Let $ V_1 \toto E_1$ and $V_2 \toto E_2$ be VB-groupoids
 as in \eqref{eq:VB:samebase}, with cores $C_1$ and $C_2$, respectively.
	Choose any  right decompositions of $V_1$ and $V_2$. The following statements
hold.
\begin{enumerate}
	\item[(i)] A VB-groupoid morphism $\Phi: V_1 \to V_2$ over the identity of $\Gamma$
	induces a morphism of 2-term homotopy $\Gamma$-modules
	$$ \underline{\Phi} : ( C_1 [1] \oplus  E_1, D_{V_1}) \to ( C_2 [1] \oplus  E_2, D_{V_2})  .$$ 
	\item[(ii)] Assume that
 $\Phi$ and $\Psi:V_1\rightarrow V_2$ are  homotopic VB-groupoid 
	morphisms  with homotopy $h:E_1\rightarrow C_2$. Then the induced morphisms of 2-term homotopy $\Gamma$-modules $ \underline{\Phi}$ and $\underline{\Psi}$
	are homotopic with homotopy $h$.
		\item[(iii)] If the VB-groupoids $V_1$ and $V_2$ are homotopy 
equivalent,   so are their induced 2-term homotopy $\Gamma$-modules $( C_1 [1] \oplus  E_1, D_{V_1}) $ and $( C_2 [1] \oplus  E_2, D_{V_2})  $.
	\end{enumerate}
\end{lemma}

Let us consider now the case of the tangent and cotangent groupoid
of a Lie groupoid $\Gamma \toto M$ with unit map  $\epsilon: M \hookrightarrow \Gamma$.
A \emph{compatible Ehresmann connection on $s : \Gamma \to M$}  is a Ehresmann connection on the source map $s : \Gamma \to M$ which coincides with
$\epsilon_* (T_m M) $ at the point $\epsilon (m)$ for all $m \in M$.
The following lemma is obvious.

\begin{lemma}
\label{lem:paris}
The following are equivalent:
\begin{enumerate}
\item[(i)] right-decompositions for the tangent VB-groupoid $T\Gamma$ ;
\item[(ii)] right-decompositions for the cotangent VB-groupoid $T^\vee\Gamma$;  
\item[(iii)] compatible Ehresmann connections on $s: \Gamma \to M$.
\end{enumerate}
\end{lemma}

Let $\gpoidmm$ be a Lie groupoid, and let us choose a
compatible Ehresmann connection on $s:\Gamma\to M$.
According  to  Lemma \ref{lem:paris}, we thus obtain right-decompositions for
 the tangent VB-groupoid $T\Gamma$ 
and  cotangent VB-groupoid $T^\vee\Gamma$.
By  Proposition \ref{VBgroupoid_2termrep} (i), these
 VB-groupoids correspond to  $2$-term homotopy $\Gamma$-modules
 denoted $(A[1]\oplus  TM, D_T)$ and 
 $(T^\vee [1]M \oplus  A^\vee, D_{T^\vee})$, 
referred to as the \emph{adjoint $2$-term homotopy $\Gamma$-module} 
and \emph{coadjoint $2$-term homotopy $\Gamma$-module},
 respectively. A different choice of 
compatible Ehresmann connection gives rise to  gauge equivalent 
$2$-term homotopy $\Gamma$-modules.

\begin{proposition}
	\label{prop:coadjointIsAdjointShifted}
	Let $\Gamma$ be a Lie groupoid.
The coadjoint 	$2$-term homotopy $\Gamma$-module is the dual of the adjoint 	$2$-term homotopy $\Gamma$-modules
shifted by $+1$.
\end{proposition}

This proposition follows from a more general fact.
	Let $V \toto E$ be a Lie groupoid with core $C$.
	Then the $2$-term homotopy $\Gamma$-module $( C[1] \oplus E , D_{V^\vee} ) $ associated to the dual VB-groupoid $V^\vee$ is the dual of the $2$-term homotopy $\Gamma$-module  $( C[1] \oplus E , D_{V} ) $ associated to $V$, shifted by $+1$. In particular $D_{V^\vee}=D^\vee$.

\medskip 

\subsection{2-term complexes over a differentiable stack}

In this subsection, we interpret the results obtained in 
Section \ref{VB_homotopy_and_Morita}
 on Morita equivalence of VB-groupoids in terms of homotopy $\Gamma$-modules.

\begin{definition}
 \label{morita_homotopy_gamma_module}
 A 2-term homotopy $\Gamma_1$-module $(\ec_1,D_1)$ and a  2-term homotopy
 $\Gamma_2$-module $(\ec_2,D_2)$  are said to be Morita equivalent if 
there exist
\begin{enumerate}
 \item[(i)] a $\Gamma_1-\Gamma_2$ bitorsor  $\xymatrix{M_1 & X\ar[l]^{\varphi_1}\ar[r]_{\varphi_2} & M_2}$; and
  \item[(ii)] an homotopy equivalence between the pull-backs  $ (\varphi_1^* \ec_1, \varphi_1^* D_1) $ and $ (\varphi_2^* \ec_2, \varphi_2^* D_2) $ along $\varphi_1$ and $\varphi_2$ respectively.
\end{enumerate}
\end{definition}

In Definition \ref{morita_homotopy_gamma_module} (ii),
 we canonically identified   the base groupoids
  $\Gamma_1[X]$ with  $\Gamma_2[X]$ as in
 Remark \ref{base_pull_back_bitorsor}.

\begin{definition}
\label{def:2termoverstack}
Let $\XX$ be a differentiable stack. A 2-term complex over $\XX$ is a
  Morita equivalence class of  2-term homotopy $\Gamma$-modules $\ec$,
 where $\Gamma\toto M$ is any representative of $\XX$.  
\end{definition}

We denote the Morita equivalence class of the homotopy $\Gamma$-module $(\ec,D)$
 as $[\ec]$,
 and say that $(\ec,D)$ represents $[\ec]$ on $\Gamma\toto M$.

\begin{proposition}
	\label{prop:translation}
Let $V_i$, $i=1,2$, be Morita equivalent VB-groupoids
 over $\Gamma_i$ with $C_i$ and $E_i$ being the cores and the units,
 respectively.
 For any choice of right decompositions, 
the induced 2-term homotopy $\Gamma_1$-module
$(C_1[1]\oplus E_1,D_{V_1})$ and 2-term homotopy $\Gamma_2$-module
$(C_2[1]\oplus E_2,D_{V_2})$ are Morita equivalent. 
\end{proposition}
\begin{proof}
	For any VB-groupoid $V$ as in \eqref{eq:florence} and any
 surjective submersion $ \varphi : X\to M$,
	the 2-term homotopy $\Gamma$-module associated to the pull-back
 VB-groupoid $ \varphi^* V$ is the pull-back along $\varphi:X \to M$ of the 2-term homotopy $\Gamma$-module associated to $V$.
The result is then a straightforward consequence of Lemma \ref{VB_chain_homotopy_2} (iii) and Proposition \ref{morita_VB_groupoids}.
\end{proof}

\begin{corollary}
\label{cor:PSU}
Let $\Gamma_1\toto M_1$ and $\Gamma_2\toto M_2$ be Morita equivalent
 Lie groupoids. Choose right decompositions on $T\Gamma_1$ and $T\Gamma_2$. Then
\begin{enumerate}
 \item[(i)] the adjoint 2-term homotopy $\Gamma_1$-module 
$(A_1[1]\oplus TM_1,D_{T\Gamma_1})$ and the adjoint homotopy $\Gamma_2$-module $(A_2[1]\oplus TM_2,D_{T \Gamma_2})$  are Morita 
 equivalent;
 \item[(ii)] the induced  coadjoint 2-term homotopy $\Gamma_1$-module 
 $((T^\vee M_1)[1]\oplus A^\vee_1,D_{T^\vee \Gamma_1})$  and the
 coadjoint 2-term homotopy $\Gamma_2$-module $((T^\vee M_2)[1]\oplus A^\vee_2,D_{T^\vee \Gamma_2})$
 are Morita equivalent.
\end{enumerate}
\end{corollary}

The following definition directly  interprets the content of 
Proposition \ref{prop:MoritaAndHomotMorphism} in terms 
of homotopy $\Gamma$-modules.

\begin{definition}\label{generalized_VB_homotopy_bis}
Let $(\ec_1,D_1)$ and $(\ec_1',D_1')$ be homotopy $\Gamma_1$-modules and
$\Phi:\ec_1\rightarrow \ec_1'$  a morphism of homotopy $\Gamma_1$-modules.
Similarly,  let $(\ec_2,D_2)$ and $(\ec_2',D_2')$ be homotopy
 $\Gamma_2$-modules and $\Psi:\ec_2\rightarrow \ec_2'$ 
 a morphism of homotopy $\Gamma_2$-modules. 
We say that $\Phi$ and $\Psi$ are equivalent with respect to a bitorsor,
 if there exist
\begin{enumerate}
\item[(i)] a $\Gamma_1-\Gamma_2$ bitorsor $\xymatrix{M_1 & X\ar[l]^{\varphi_1}\ar[r]_{\varphi_2} & M_2}$; and
\item[(ii)]  an homotopy equivalence between the pull back morphisms
\[
\varphi^*_1 \Phi: \varphi^*_1 \ec_1  \rightarrow \varphi^*_1  \ec_1' \ ,\;\;  \varphi_2^*  {\Psi}:   \varphi_2^* \ec_2 \rightarrow \varphi_2^* \ec_2'
\]
over canonically isomorphic groupoids, as in Remark \ref{base_pull_back_bitorsor}.
\end{enumerate}
\end{definition}

Let us remark that the homotopy $\Gamma_1$-modules $(\ec_1,D_1)$ and $(\ec_1',D_1')$ are
then, respectively, Morita equivalent to $(\ec_2,D_2)$ and $(\ec_2',D_2')$.
 The following result is easily obtained combining 
 Lemma \ref{VB_chain_homotopy_2} (ii) with Proposition \ref{prop:MoritaAndHomotMorphism}.
\begin{proposition}
\label{pro:Lyon} 
 VB-groupoid morphisms which are equivalent as generalized VB-morphisms
as in Definition \ref{bitorsor_equivalence_of_VB_morphism} give rise to  equivalent
morphisms between their Morita equivalent 2-term
 homotopy  groupoid modules in the sense of
 Definition \ref{generalized_VB_homotopy_bis}.
\end{proposition}


\begin{definition}
Let $\XX$ be a differentiable stack. A morphism between
 2-term complexes over $\XX$
 is an  equivalence class  of morphisms with respect to a bitorsor (in the sense of
 Definition \ref{generalized_VB_homotopy_bis})  between Morita equivalent 2-term homotopy  $\Gamma$-modules.
\end{definition}

\begin{corollary}
\label{cor:Lyon}
An equivalence class of VB-groupoid morphisms
as generalized VB-morphisms with respect to a bitorsor (see Definition
\ref{bitorsor_equivalence_of_VB_morphism}) induces a morphism
of the corresponding 2-term  complexes over the stack.
\end{corollary}

\section{The rank of a $(+1)$-shifted  Poisson stack}

\label{sec:COVID}
\subsection{The tangent complex and cotangent complex}

Now we are ready to introduce {\em the tangent complex} and
 {\em the  cotangent complex} of a differentiable stack $\XX$.

\begin{definition}
Let $\XX$ be a differentiable stack.
\begin{enumerate}
\item[(i)] By  {\em the tangent complex of $\XX$}, denoted
by ${T_\XX}$, we mean the
$2$-term complex over $\XX$  defined by
 the Morita equivalence class of
the adjoint 2-term homotopy $\Gamma$-module $(A[1]\oplus TM,D_{T} )$;
\item[(ii)] by  {\em the cotangent complex of $\XX$}, denoted
by ${L_\XX}$, we mean 
 the $2$-term complex over $\XX$  defined by
 the  Morita equivalence class of the dual  $( TM^\vee \oplus A^\vee[-1],D_{T^\vee} )$
of the adjoint 2-term homotopy $\Gamma$-module. 
\end{enumerate}
Here $\Gamma\toto M$ is  any Lie groupoid representing $\XX$.
\end{definition}

Corollary \ref{cor:PSU} and Proposition \ref{prop:coadjointIsAdjointShifted} imply that the definition above is indeed justified.
The representative $(A[1]\oplus TM,D_{T} )$ of $T_\XX$ is denoted
 $T_\XX|_M$, 
while the representative  $( TM^\vee \oplus A^\vee[-1],D_{T^\vee} )$ of $L_\XX$ by $L_\XX|_M$. 

The following result is an immediate  consequence of 
Theorem \ref{morphism_quasi_poisson2}.

\medskip
\begin{theorem}
\label{pro:paris}
A $(+1)$-shifted Poisson structure on a differentiable stack $\XX$
defines a morphism of $2$-term complexes over
$\XX$ from the shifted cotangent complex to 
the tangent complex:
\begin{equation}
\label{eq:shifted}
 \Pi^{\#}:  {L_\XX}[1]\to {T_\XX}.
\end{equation}
\end{theorem}

\begin{remark}
\label{rk:3}
The morphism \eqref{eq:shifted} is analogous to
one  in \cite[Definition 3.2.1]{CPTVV}.
\end{remark}

Choosing a compatible  Ehresmann connection on $s: \Gamma \to M$ as in Lemma 
\ref{lem:paris}, one can describe  explicitly the morphism of
homotopy  $\Gamma$-modules
$\Pi^{\#}:  {L_\XX}[1]_M\to {T_\XX}|_M$.

Let $(\Gamma,\Pi,\Lambda)$ be a quasi-Poisson groupoid over $M$.
The VB-groupoid morphism $\Pi^{\#}: T^\vee \Gamma \to T\Gamma $, recalled in Proposition \ref{pro:quasiP},
induces, on the unit manifold, a vector bundle morphism
(see \cite{MackenzieX:1994}):
\begin{equation}
\label{eq:Congqiang}
\rho_*: A^\vee\to TM.
\end{equation}

\begin{proposition}
\label{pro:beijing}
Under the same hypothesis as in  Theorem \ref{pro:paris}, for any
presentation $\Gamma \toto M$ of $\XX$, the linear term of
the morphism of homotopy  $\Gamma$-modules $\Pi^{\#}: {L_\XX} [1]|_M\to 
{T_\XX}|_M$ is the morphism of complexes:

\begin{equation}
\label{eq:perugiabis}
\begin{tikzcd}0\arrow{d}&0\arrow{d}\\
{(T^\vee M) [1]} \arrow[xshift=-2pt]{d}{\rho^\vee}  \arrow{r}{-\rho_*^\vee} &
{A[1]} \arrow[xshift=-2pt]{d}{\rho}  \\
{A^\vee}\arrow{d} \arrow{r}{\rho_*} & {TM}\arrow{d}\\
0&0
\end{tikzcd}
\end{equation}
\end{proposition}

\subsection{Rank of a $(+1)$-shifted  Poisson stack}

The main purpose of this section is to introduce
the notion of rank of a $(+1)$-shifted  Poisson  structure on a
  differentiable  stack.

Recall that the rank of an ordinary Poisson manifold is defined at 
each point of the underlying manifold.
For a $(+1)$-shifted  Poisson  structure  on  a  differentiable  stack,
we  define its rank at each 
point of the \emph{coarse moduli space $|\XX|$} of the differentiable stack $\XX$.
The latter can be identified  with the
 orbit space $M/\BaseGroupoid$ as a topological space, 
where $\BaseGroupoid \toto M$ is any  Lie groupoid  representing
 the differentiable  stack $\XX$. It is known that $M/\Gamma$ is invariant
under Morita equivalence of $\BaseGroupoid$.

Our strategy is, first of all,
 to define  the rank of a quasi-Poisson  groupoid
$(\Gamma \toto M, \Pi, \Lambda)$ at any given point $m\in M$.
Then, we show that this rank is constant along Lie groupoid orbits.
Furthermore, it is also  invariant under 
twists of the quasi-Poisson structures,
 and indeed is invariant under Morita equivalence. In this way, we
are led to a well defined map $|\XX|\to \zz$, called the \emph{rank of the 
 $(+1)$-shifted  Poisson stack}.

%
%

\begin{definition}
 \label{def:rankPoint}
The \emph{rank} of a quasi-Poisson groupoid $\quasip$ at any $m\in M$
is defined to be
$${\rm dim}( \rho (A_m) +  \rho_* (A^\vee_m)) - {\rm rank}(A),$$
where $\rho: A\to TM$ 
is the anchor of Lie algebroid $A $,
 and $\rho_* : A^\vee \to TM$ is the bundle map
as in Equation \eqref{eq:Congqiang}.
\end{definition}

\begin{remark}
\label{rmk:practical}
Recall that the dimension  \cite{BehrendXu} of  a differentiable  stack
$\XX$ is defined as  $\dim \XX=\dim (M)-{\rm rank}(A)$,
where $\Gamma\toto M$ is a Lie groupoid representing $\XX$, and
$A$ its Lie algebroid.
Hence  the rank of the quasi-Poisson groupoid
 $\quasip$ at $m\in M$ 
can also be expressed as
$$\dim \XX-\dim(  \ker \rho^\vee|_m \cap \ker  \rho_*^\vee|_m)$$
\end{remark}

For a Poisson groupoid \cite{MackenzieX:1994}, 
the rank is  maximal at a point $m \in M$, i.e.,  equal to $\dim \XX$, if and only if
 the orbit of the Lie algebroid $A$ and the orbit of the dual Lie algebroid $A^\vee$ 
intersect transversally at $m \in M$.

\begin{proposition}
\label{lem:zurich}
Let $\quasip$ be  a quasi-Poisson groupoid.  The  rank of the quasi-Poisson structure $(\Pi,\Lambda)$ 
is constant on any orbit of the groupoid.
\end{proposition}
\begin{proof}
According to Remark \ref{rmk:practical},
 it suffices to show that  $\text{dim}  (\ker \rho^\vee|_m \cap \ker  \rho_*^\vee|_m)$ is constant along the Lie groupoid orbits.

We start with a few linear algebra facts.
Let us call \emph{butterfly} a commutative diagram ${\mathcal C}$ of the
 form below,
where $NW,NE,SW,SE,C$ are vector spaces 
and  both diagonal lines are short exact sequences:
$$  
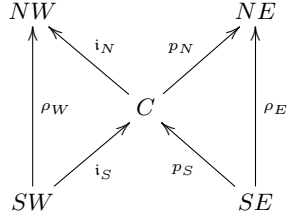
\xymatrix{ 0 \ar@{<-}[rd] &  & & & 0 \ar@{<-}[ld] \\ & NW  \ar@{<-}[rd]^{p_W} & & NE  & \\& & C\ar@{<-}[rd]_{{\mathfrak i}_W} \ar[ru]^{p_E} & & \\ & SW  \ar_{\rho_W}[uu] \ar[ru]_{{\mathfrak i}_E} & & SE \ar_{\rho_E}[uu] \\
	 0 \ar@{->}[ru] &  & & & 0 \ar@{->}[lu]
	} $$
From now on, the four exterior arrows (pointing to $0$ or from $0$) square shall not be drawn when representing a butterfly. We remark that similar butterfly diagrams were previously considered also by Aldrovandi and Noohi in \cite{Noo0}.

By diagram chasing,
  each butterfly induces a vector space isomorphism:
$$ {\mathcal K}_{\mathcal C} : {\rm Ker}(\rho_W) \stackrel{\sim}{\longrightarrow} {\rm Ker}(\rho_E)\, .  $$
More explicitly: $a_W \in {\rm Ker}(\rho_W) $ and $a_E \in {\rm Ker}(\rho_E) $
correspond one to the other through the isomorphism ${\mathcal K}_{\mathcal C}  $ if and only if
$ {\mathfrak i}_E( a_W) = {\mathfrak i}_W( a_E)  $.

By a {\em butterfly morphism}  from a butterfly ${\mathcal C}$ to another
 butterfly ${\mathcal C}'$, we mean a family of five linear maps, as represented by dotted lines in diagram (\ref{eq:butterflyMorph}) below, making it commutative:
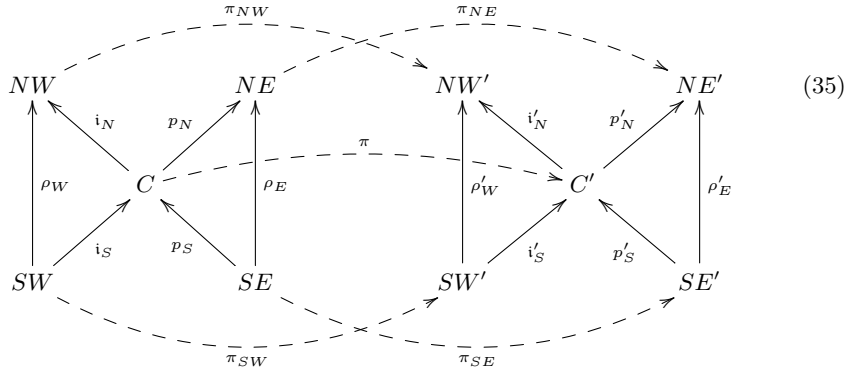
\begin{equation}
 \label{eq:butterflyMorph}
 \xymatrix{ NW  \ar^{\pi_{NW}}@/^2pc/@{-->}[rrrr] \ar@{<-}[rd]^{p_E} & & NE \ar^{\pi_{NE}}@/^2pc/@{-->}[rrrr] & & NW'  \ar@{<-}[rd]^{p_E'} & & NE'\\ 
& C\ar@{<-}[rd]_{{\mathfrak i}_E} \ar[ru]^{p_W} \ar^{\pi}@{-->}@/^1pc/[rrrr] & & &  & C'\ar@{<-}[rd]_{{\mathfrak i}_E'} \ar[ru]^{p_W'} & \\
SW  \ar_{\rho_W}[uu] \ar_{\pi_{SW}}@/_2pc/@{-->}[rrrr] \ar[ru]_{{\mathfrak i}_W} & & SE   \ar_{\rho_E}[uu] \ar_{\pi_{SE}}@/_2pc/@{-->}[rrrr] &  & SW'
 \ar_{\rho_W'}[uu] \ar[ru]_{{\mathfrak i}_W'} & & SE'  \ar_{\rho_E'}[uu]} 
\end{equation}
By diagram chasing, it is routine to check that $\pi_{SW} $ (resp. $\pi_{SE}$) maps  $ {\rm Ker}(\rho_W)$ to $ {\rm Ker}(\rho_W')$
(resp.  $ {\rm Ker}(\rho_E)$ to $ {\rm Ker}(\rho_E')$).
Moreover, the commutativity of the diagram (\ref{eq:butterflyMorph}) implies the commutativity of the following diagram (where vertical maps are
 vector space isomorphisms):

$$
\xymatrix{  {\rm Ker}(\rho_W) \ar^{\pi_{SW}}[r] \ar@{<->}^{{\mathcal K}_{\mathcal C} }[d] & \ar@{<->}^{{\mathcal K}_{\mathcal C}'}[d] {\rm Ker}(\rho_W')\\  {\rm Ker}(\rho_E) \ar^{\pi_{SE}}[r] &  {\rm Ker}(\rho_E')} 
$$

Therefore, it follows that
 the butterfly morphism induces a vector space isomorphism: 
\begin{equation}
 \label{eq:linearAlgebraFact}
{\rm Ker}(\rho_W) \cap {\rm Ker}(\pi_{SW}) \simeq {\rm Ker}(\rho_E) \cap {\rm Ker}(\pi_{SE}).
\end{equation}

For any $\gamma \in \BaseGroupoid$ with source $m$ and target $n$, the commutative diagrams below are easily verified to be butterflies:

\begin{tabular}{ccc}    
$  \xymatrix{ A_m^\vee   \ar@{<-}[rd]^{L^\vee_\gamma } & & A_n^\vee   \\ & T_\gamma^\vee \BaseGroupoid \ar@{<-}[rd]_{t_{T\Gamma}^\vee} \ar[ru]^{R^\vee_\gamma} & \\
T_m^\vee M  \ar[ru]_{s_{T\Gamma}^\vee} \ar^{\rho^\vee|_m}[uu]& & T_n^\vee M \ar_{\rho^\vee|_n}[uu]} $
&
\centering{ \begin{tabular}{c} \\ \\   and \\ \end{tabular}} 
&
$  \xymatrix{ T_m M  \ar@{<-}[rd]^{s_{T\Gamma}} & & T_n M  \\ & T_\gamma \BaseGroupoid \ar@{<-}[rd]_{R_\gamma} \ar[ru]^{t_{T\Gamma}} & \\ A_m  \ar^{\rho|_m}[uu] \ar[ru]_{L_\gamma}& & A_n \ar_{\rho|_n}[uu]  } $
\end{tabular}

The multiplicative bivector field $\Pi$ induces a butterfly morphism with central map $ \Pi^\#: T_\gamma^\vee \BaseGroupoid \to T_\gamma \BaseGroupoid$, from the first butterfly to the second one. Here the four remaining arrows are (in notations of  (\ref{eq:butterflyMorph})):
 \begin{eqnarray*} 
   \pi_{NW} = \rho_* |_m\, &  \pi_{SW}= \rho_*^\vee |_m\,  \\
     \pi_{NE} = \rho_* |_n\,  & \pi_{SE}= \rho_*^\vee |_n\,  . 
 \end{eqnarray*}

It thus follows from the isomorphism (\ref{eq:linearAlgebraFact})
that the rank of the quasi-Poisson structure $(\Pi,\Lambda)$
is indeed  constant on any orbit of the groupoid. This
concludes the proof.
\end{proof}

Let $\quasip$ be  a quasi-Poisson groupoid, $T \in \Sigma^1 A$ a twist,
and  $(\Pi_T,\Lambda_T)$ the corresponding
 twisted Poisson structure as in Definition \ref{def:quasiPoisson}.
 Denote by $\rho_* : A^\vee \to TM$ and $\rho_*^T : A^\vee \to TM$ the vector bundle morphisms associated to the quasi-Poisson structures $(\Pi,\Lambda)$ and $(\Pi_T,\Lambda_T)$, respectively.
 The  following relations  can be easily verified:

\begin{equation}
\label{eq:rho*}
\rho_*^T=\rho_* +\rho \smalcirc T^\# \hbox{ and } (\rho_*^T) ^\vee=(\rho_*)^\vee - T^\# \smalcirc \rho^\vee.
\end{equation}
	
Now \eqref{eq:rho*} implies the following:

\begin{lemma}
\label{lem:zurich2}
Let $\quasip$ be  a quasi-Poisson groupoid.  Then for any $T\in \Sigma^1 A$, 
\[
\text{dim}  \big(\ker \rho^\vee|_m   \cap \ker (\rho^T_*)^\vee|_m\big)
=\text{dim}
\big(  \ker \rho^\vee|_m \cap \ker   \rho_*^\vee|_m \big), \  \ \ \forall m \in M.
\]
\end{lemma}
	
As an immediate consequence  of Definition \ref{def:rankPoint},
 Remark \ref{rmk:practical} and Lemma \ref{lem:zurich2}, we have

\begin{corollary}
\label{cor:zurich2}
The ranks at a given  orbit of any two  quasi-Poisson
 structures on a Lie groupoid $\Gamma \toto M$,
which are  equivalent up to a twist,  are equal.
\end{corollary}

Finally, we have the following

\begin{lemma}
\label{lem:rankMorita}
Let  $(\Gamma' , \Pi', \Lambda')$ and $(\Gamma , \Pi, \Lambda)$
 be quasi-Poisson groupoids.
Assume that  
\begin{equation}
\begin{tikzcd}
{\Gamma'} \arrow[xshift=-3pt]{d} \arrow[xshift=3pt]{d} \arrow{r}{\phi} &
{\Gamma} \arrow[xshift=-3pt]{d} \arrow[xshift=3pt]{d} \\
{M'} \arrow{r}{\varphi} & {M}
\end{tikzcd}
\end{equation}
is a Morita morphism of quasi-Poisson groupoids
from  $(\Gamma' , \Pi', \Lambda')$ to $(\Gamma , \Pi, \Lambda)$.
Then, for any  $ m \in M'$, the rank of the quasi-Poisson structure
$( \Pi', \Lambda')$  at $m$ is equal to  the rank of
 the quasi-Poisson structure $(\Pi, \Lambda)$ at
$\varphi (m)\in M$.
\end{lemma}
\begin{proof}
	Let $(A',\rho')$ and $(A,\rho)$ be the Lie algebroids of $\Gamma$ and $\Gamma'$ respectively,
	and let $ \phi_A:A ' \to A$ the Lie algebroid morphisms induced by $\phi$.
Since $\phi$ is a Lie groupoid morphism, for all $m' \in M'$, both pairs $( \phi_A,\varphi_*)$ and $(\varphi_*^\vee , \phi_A^\vee)$ are chain maps:
	$$ 
	\xymatrix{ 
	0\ar[d]&0\ar[d]\\
	A_{m'}'[1] \ar[r]^{\phi_A} \ar[d]_{\rho'} & A_{\varphi(m')}[1] \ar[d]^{\rho}\\ T_{m'}M'\ar[r]^{\varphi_*}\ar[d] & T_{\varphi(m')}M \ar[d]\\
	0&0
	} 
	\begin{array}{c} \\ {}\\{}\\  \hbox{ and } \end{array} 
	\xymatrix{ 
	0\ar[d]&0\ar[d]\\
	T_{\varphi(m')} M^\vee[1] \ar[r]^{\varphi_*^\vee} \ar[d]_{\rho^\vee} & (T_{\varphi(m')} M'^\vee) [1] \ar[d]^{(\rho')^\vee} \\ 
	A^\vee_{\varphi(m')} \ar[r]^{(\phi_A)^\vee} \ar[d] & (A')_{m'}^\vee\ar[d]\\ 
	0&0} 
	$$
Since $ \phi$ is a Morita morphism, it is routine to check that these
 chain maps are indeed  quasi-isomorphisms.
	
In view of Corollary \ref{cor:zurich2},
 without  loss of generality, we can assume
that $ \Pi'$ is the horizontal lift of $ \Pi$
 with respect to an Ehresmann connection on $ \varphi: M'  \to M$.
 This implies the commutativity of the following diagram, for all $m' \in M'$:
$$ \xymatrix{ (T_{m'}\Gamma')^\vee \ar[rr]^{(\Pi')_{m'}^\#} & & T_{m'} \Gamma' \ar[d]^{\phi_* }\\\ar[u]^{(\phi_*)^\vee } 
	\ar[rr]^{\Pi_{\varphi(m')}^\#}  (T_{\varphi(m')} \Gamma)^\vee & & T_{\varphi(m')} \Gamma .} $$

In turn, this implies the commutativity of the following diagram,
  $\forall m' \in M'$:
\begin{equation}
  \xymatrix{ 
	(T_{m'}M')^\vee[1] \ar[rrr]^{(\rho_*')^\vee }   \ar[rd]^{(\rho')^\vee} &  & &  (A')_{m'}[1] \ar[ddd]^{\phi_A} \ar[ld]_{\rho'} \\ 
	&\ar[r]^{-\rho_*' }   (A')_{m'}^\vee & T_{m'}M'  \ar[d]^{\varphi_* }  & \\
	& \ar[r]^{-\rho_* } \ar[u]^{\phi_A^\vee} A^\vee_{\varphi(m')}& T_{\varphi(m')}M'& \\
\ar[uuu]^{\varphi_*^\vee } \ar[ru]^{\rho^\vee}  (T_{\varphi(m')}M)^\vee[1]	\ar[rrr]^{(\rho_*)^\vee } & & & \ar[lu]_{\rho} A_{\varphi(m')}[1]
}
\end{equation}
where $\rho_*: A^\vee\to  T M$ and $\rho_*':  (A')^\vee \to T M'$ are
the bundle maps associated to $\Pi$ and $\Pi'$, respectively,
as in   \eqref{eq:Congqiang}.
Since both vertical chain maps
  $(\varphi_*^\vee , \phi_A^\vee)$ and $( \phi_A,\varphi_*)$ are 
quasi-isomorphisms, both horizontal chain maps
 $ ((\rho_*)^\vee, - \rho_* )$ and $ ((\rho_*')^\vee, - \rho_*' )$ induce
 the same map at the level of cohomology.
The latter  implies  that the rank of $ (\Pi',\Lambda')$ at $m'$ is equal to the rank of $( \Pi,\Lambda)$ at $ \varphi (m') \in M$.
\end{proof}

Now, we are ready to introduce the rank of a $(+1)$-shifted Poisson
structure on a differentiable stack $\XX$.

\begin{definition}
For a $(+1)$-shifted Poisson structure  $\pixx$
on a differentiable stack $\XX$,
let $\quasip$ be  any quasi-Poisson groupoid representing it.  Define
the rank of $\pixx$ as a map $|\XX|\to \zz$:
$$\text{rank} \, \pixx = \dim \, \XX- \text{dim} \big(
 \ker \rho^\vee|_m \cap \ker  \rho_*^\vee|_m \big), $$
where $m$ is any point in the groupoid orbit representing the element
in the coarse moduli space $|\XX|$ of the~stack~$\XX$.
\end{definition}

According to Lemma \ref{lem:rankMorita}, $\text{rank} \, \, \pixx$
 is indeed well
defined.  Let us now describe non-degenerate Poisson structures on a differentiable stack.

\begin{definition}
\label{def:nondeg}
A $(+1)$-shifted Poisson structure  $\pixx$ on a differentiable stack $\XX$
is  non-degenerate if and only if   the linear
term \eqref{eq:perugiabis} of the morphism of homotopy $\Gamma$-modules 
$\Pi^{\#}:  {L_\XX}[1] |_M\to {T_\XX}|_M$ is a quasi-isomorphism
of  2-term complexes of vector bundles. That is, for any $m\in M$, the
morphism defined by the horizontal arrows as in
 \eqref{eq:perugiabis}:
\begin{equation}
\label{eq:perugia2}
\begin{tikzcd}
0\arrow{d}&0\arrow{d}\\
{(T_m^\vee M) [1]} \arrow[xshift=-2pt]{d}{\rho^\vee}  \arrow{r}{-\rho_*^\vee} &
{A_m[1]} \arrow[xshift=-2pt]{d}{\rho}  \\
{A^\vee_m} \arrow{r}{\rho_*}\arrow{d} & {T_mM}\arrow{d}\\
0&0
\end{tikzcd}
\end{equation}
is a  quasi-isomorphism of the 2-term complexes.
\end{definition}

Not all differentiable stacks admit $(+1)$-shifted
 non-degenerate Poisson structures.

\begin{lemma}
\label{lem:dim0}
If $\XX$ is a $(+1)$-shifted non-degenerate Poisson stack,
 then $\dim \, \XX=0$.
\end{lemma}  
\begin{proof}
Since the 2-term complexes of vector bundles associated to
 ${L_\XX}[1] |_M $ and $ {T_\XX}|_M$ are $(T^\vee M)[1] \stackrel{\rho^\vee}{\to} A^\vee $
 and $A[1] \stackrel{\rho}{\to} TM$, respectively, 
 their Euler characteristics are $-\text{dim} \,\XX$ and $\text{dim} \, \XX$,
 respectively.
 Since quasi-isomorphic 2-term complexes have the same Euler characteristic,
 we have $\text{dim} \, \XX = - \text{dim} \, \XX$. Therefore,
  it follows that  $\text{dim} \, \XX=0$.
\end{proof}

The following proposition gives an alternative description
of non-degenerate Poisson stacks. 
\begin{proposition}\label{non_degeneracy_criterion}
A $(+1)$-shifted Poisson structure  $\pixx$ on a differentiable stack $\XX$
is non-degenerate if and only if $\text{rank} \, \pixx=\text{dim} \, \XX = 0$
 uniformly on the coarse moduli space of the stack.
\end{proposition}
\begin{proof}
Assume that $\pixx$  is non-degenerate. By Lemma \ref{lem:dim0},
we know that  $\text{dim} \,\XX = 0$.  From assumption,
it follows  that $\text{dim} \big(
 \ker \rho^\vee|_m \cap \ker  \rho_*^\vee|_m \big)=0$.
Therefore, $\text{rank}\, \pixx=0$ according to Remark \ref{rmk:practical}. 

Conversely, assume that $\text{rank}  \pixx=\text{dim}\, \XX = 0$.
It thus follows that  all vector spaces in Diagram \eqref{eq:perugia2} 
have the same dimension. A simple linear algebra argument
implies that the
morphism defined by the horizontal arrows
in \eqref{eq:perugia2}
must be  a  quasi-isomorphism of the 2-term complexes.
\end{proof}

\section{Examples}

In this section, we present several examples of quasi-Poisson groupoids,
which have appeared in literatures.

\subsection{Quasi-Poisson groups}
\begin{example}
Let $(G, \Pi, \Lambda)$ be a quasi-Poisson group of dimension $n$ in the sense of Kosmann-Schwarzbach \cite{Yvette}.
As a Lie groupoid over a point,
 it defines a  $(+1)$-shifted Poisson structure on $[\cdot/G]$ of rank $-n$, since $\rho=\rho_*=0$.  
 
Indeed $(+1)$-shifted Poisson structures on  $[\cdot/G]$ correspond exactly to
equivalence classes of quasi-Poisson group structures on $G$, where the equivalence relation is given
by ``Drinfeld twists" \cite{Yvette}.  In particular,  they cannot be non-degenerate.

When $G$ is a  Lie group whose Lie algebra $\frakg$ is
equipped with a symmetric $\frakg$-invariant element
 $t\in S^2(\frakg)^G$, then $(G, \Pi, \Lambda)$, where
$\Pi=0$ and $ \Lambda=-\frac{1}{4}[t_{12}, t_{23}]\in (\wedge^3 \frakg)^G$,
defines a quasi-Poisson group. This induces a $(+1)$-shifted Poisson structure on $[\cdot/G]$. 
In particular,
 it is known that any {\em quasi-triangular}  Poisson Lie group is twist-equivalent to a quasi-Poisson group
 on $G$ with $0$ bivector field, and therefore  its corresponding $(+1)$-shifted Poisson stack on $[\cdot/G]$
is  isomorphic to the $(+1)$-shifted Poisson stack on $[\cdot/G]$ described
above.  This viewpoint can certainly be traced back to 
Drinfeld \cite{Drinfeld2}.

 For the specific case when $G$ is a connected  and simply connected semi-simple Lie group,
it is possible to show \cite{Drinfeld2, Yvette} that any quasi-Poisson group
structure on $G$  is twist-equivalent to  the one  as above,
where the twist $t\in S^2(\frakg)^G \cong S^2(\frakg^\vee)^G$ is a multiple of the Killing form.
This establishes a one to one correspondence between 
$(+1)$-shifted Poisson structures  on $[\cdot/G]$ and  elements in $(\wedge^3 \frakg)^G$.
\end{example}

\subsection{Manin pairs}

Another type of quasi-Poisson groupoid arises as integration of Manin pairs. 
Let $(\mathfrak d, \mathfrak g)$ be a {\em Manin pair} \cite{Drinfeld},
 that is, ${\mathfrak d}$ is an even dimensional quadratic Lie algebra 
(i.e. a Lie algebra equipped
 with an ad-invariant, non-degenerate symmetric bilinear form) of  
 signature $(n,n)$ and $\mathfrak g$ is a maximal isotropic  Lie subalgebra of $\mathfrak d$.
Choose   an isotropic complement $\mathfrak h$ of $\mathfrak g$
 in $\mathfrak d$. The data
 $(\mathfrak d, \mathfrak g, \mathfrak h)$ is called
a {\em Manin quasi-triple}. A Manin quasi-triple induces a quasi-Lie bialgebra
 $(\mathfrak g, {\mathfrak g}^\vee)$ \cite{Drinfeld2}. Two different
 choices of isotropic complement differ by a  skew-symmetric linear map
 $T: {\mathfrak g}^\vee  \simeq \mathfrak h \to {\mathfrak g}$.
 
 Let $D$ be the connected and simply connected Lie group
 with Lie algebra ${\mathfrak d}$, and $G \subset D$
a closed Lie subgroup with Lie algebra
 $\mathfrak g$.
Then $(D, G)$ is called the corresponding {\em group pair}.
Denote by $S$, the homogeneous space $S=D/G$.
 The action of the Lie group $D$ on itself by left multiplication
induces an action of $D$ on $S=D/G$,  and this, in turn,
 restricts to a  $G$-action on $S$,   called the {\em dressing action}.

It was shown in \cite{IPLGX} how the Manin quasi-triple allows
us  to define a quasi-Lie bialgebroid structure  $(A,\delta,\Omega)$ on the 
transformation Lie algebroid $A=\mathfrak g\ltimes S\to S$,
 where the anchor map $\rho:A\to TS$ and the opposite anchor map
 $\rho_*:A^\vee \to TS$ are defined by the restrictions of
 the infinitesimal dressing action to $\mathfrak g$ and 
$\mathfrak h\simeq\mathfrak g^\vee $, respectively. 

The corresponding transformation groupoid $G\ltimes S\toto S$ is
therefore naturally endowed with a quasi-Poisson structure
 $(\Pi_S, \Lambda)$; any two such quasi-Poisson structures, related to
different choices of  the complement $\mathfrak h$,  are
 equivalent by a twist determined by $T$.

Indeed we  have the following:

\begin{theorem}
\label{thm:DG}
Let $(\mathfrak d, \mathfrak g)$ be a Manin pair, and $(D, G)$ its 
corresponding group pair. Then  the quotient stack $[S/G]$,
where $S=D/G$ and $G$ acts on $S$ by the dressing action,  is
 naturally  a non-degenerate $(+1)$-shifted Poisson stack.
\end{theorem}	 
\begin{proof}
According to Proposition \ref{non_degeneracy_criterion} we need to check that the rank of the quasi-Poisson groupoid is uniformly zero, i.e.
\[
\mathrm{dim}[(\mathrm{Im}\rho)_s+(\mathrm{Im}\rho_*)_s]=\mathrm{dim} T_sS
\]
for every $s\in S$. As previously remarked, by definition,
\[
(\rho,\rho_*):A_s\oplus A_s^\vee\simeq\mathfrak d \to T_sS
\]
is the infinitesimal dressing action map. Since $S$ is $D$-homogeneous this map is surjective at every point.

\end{proof}

\subsection{AMM groupoid}
\label{sec:AMM}
A particular subcase of the one considered in the previous example deserves
 some special  attention.
Let $\mathfrak g$ be a  quadratic compact Lie algebra endowed with an
ad-invariant non-degenerate bilinear form $K$. On the direct sum $\mathfrak
d=\mathfrak g\oplus \mathfrak g$,  one  constructs a scalar
product of signature $(n,n)$
 (with $n$  being the dimension of ${\mathfrak g}$) by
\[
((u_1,u_2)|(v_1,v_2))=K(u_1,v_1)-K(u_2,v_2),
\]
$\forall (u_1,u_2),(v_1,v_2)\in \mathfrak d$.
 Then $(\mathfrak d , \Delta (\mathfrak g), \Delta _- (\mathfrak g))$
 is a Manin quasi-triple,
 where $\Delta (v)=(v,v)$ and $\Delta _-(v)=(v,-v)$,
$\forall v\in\mathfrak g$  \cite{AKKS}.
If $G$ is the compact, connected and simply connected Lie group with
 Lie algebra $\mathfrak g$,
 then $D=G\times G$ is the  connected and simply connected Lie group
with Lie algebra $\mathfrak d$,
and $G$ is identified with the diagonal inside
$D=G \times G$. The map
$[(g',g)] \mapsto g' g^{-1}$ allows us to
identify the homogeneous space $D/G$ with $G$ itself;
under this identification,
 the dressing action of $G$ becomes the conjugation action.
Hence the transformation groupoid of Subsection 7.3
becomes the transformation groupoid $G\ltimes G \toto  G$. 

On this transformation groupoid, the multiplicative bivector
 field $\Pi$ on $G\times G$:
\begin{equation}
 \label{eq:Pi_gs}
\Pi|_{(g,s)}  = \half \sum _{i=1}^n \ceV{e^2_i}\wedge \Vec{e^2_i}-
\ceV{e^2_i}\wedge \ceV{e^1_i}-\Vec{(Ad_{g ^{-1}}e_i)^2}\wedge
\Vec{e^1_i}\, ,
\end{equation}
together with the constant section $\Lambda \in\Gamma (G; \wedge^3 (\mathfrak g\ltimes G))$
 corresponding to the 3-vector in $(\wedge^3 {\mathfrak g})^G$ induced by
the Cartan 3-form $\frac{1}{4}K(\cdot ,[\cdot ,\cdot ]_\mathfrak g)\in \wedge^3 {\mathfrak g}^*$,
defines a  quasi-Poisson groupoid structure \cite[Corollary 4.24]{IPLGX}).
Here $\{ e_i\}$ is an orthonormal basis of ${\mathfrak g}$ and the superscript
refers to the respective $G$-component.

By a direct verification, one can show that $\Pi$ is indeed
non-degenerate in the sense  of Definition \ref{def:nondeg}.
In summary, we have the following

\begin{theorem}
\label{AMM}
Let  $G$ be a compact, connected and simply connected Lie group
whose Lie algebra $\mathfrak g$ is  quadratic.
Then the quotient stack $[G/G]$, where $G$ acts on $G$ by
conjugation, is  naturally a non-degenerate $(+1)$-shifted Poisson stack.
\end{theorem}

\begin{remark}
\label{rk:4}
Similar to  \cite[Theorem 3.2.5]{CPTVV} and \cite[Theorem 3.33]{Pr1}
 in the algebraic geometry setting,
 we   expect that one can
invert non-degenerate $(+1)$-shifted Poisson structures
on a differentiable stack to obtain  $(+1)$-shifted  symplectic stacks.
It is known   that quasi-symplectic structures on a
Lie groupoid transfer to any  Morita equivalent Lie groupoids,
and indeed a well defined notion of Morita equivalence of
quasi-symplectic groupoids was introduced in \cite{X}.
A $(+1)$-shifted symplectic differentiable stack
is  a  Morita equivalent class of quasi-symplectic groupoids.

In \cite{BCLX2}, we will explore the question how to invert a
  non-degenerate $(+1)$-shifted
Poisson stack to obtain  a $(+1)$-shifted symplectic stack 
 by ``homotopy inverting" non-degenerate quasi-Poisson groupoids
to obtain quasi-symplectic groupoids \cite{X}.
In particular, we prove that  by homotopy inverting the above  non-degenerate
quasi-Poisson groupoid,
 we obtain the AMM quasi-symplectic groupoid $G\ltimes G\toto G$ \cite{X}.
 Therefore,  we  obtain AMM $(+1)$-shifted symplectic stack
 $[G/G]$ by inverting the  non-degenerate $(+1)$-shifted Poisson stack 
$[G/G]$  in Theorem \ref{AMM}.
\color{black}
\end{remark}

\appendix

\section{$\zz$-graded Lie $2$-algebras}

We discuss here the extension to the 
$\zz$-graded case of the standard notions of Lie $2$-algebras
 and their morphisms. This extension is straightforward.
However since we could not find it in the literature,
  we will give a self contained presentation.

\label{sec:strictLie2}

\subsection{Definitions}\label{AppendixA1}

\begin{definition}
\label{def:strictLie2}
 A \emph{{\strict}} (or graded Lie algebra crossed-module) $ \leftgla \stackrel{\diff}{\mapsto} \rightgla$ is a pair $(\leftgla , \rightgla)$
 of $\zz$-graded Lie algebras, equipped with
 \begin{enumerate}
  \item[(i)] a degree $0$ graded Lie algebra morphism $\diff:\leftgla  \to \rightgla  $,
  \item[(ii)] a graded Lie algebra action  of $\rightgla$ on the graded vector space $\leftgla$:
   $$ \begin{array}{rcl}  \rightgla \times \leftgla &\to     & \leftgla \\
                          (\pi,a)                   &\mapsto & \pi \cdot a,  
      \end{array} $$
 \end{enumerate}
 such that:
 \begin{enumerate}
  \item[(a)]  for all 
$a , \pi\in \leftgla$, the relation $  \diff (\pi \cdot a) =  [\pi,\diff a ]$ holds,
 and 
  \item[(b)]  for all $a_1$ and $a_2 \in \leftgla$, the relation $ [a_1,a_2] = (\diff a_1) \cdot a_2  $ holds.
 \end{enumerate}
\end{definition}

In the non-graded case when $\leftgla$ and $\rightgla$ are ordinary
 Lie algebras,  i.e. graded Lie algebras concentrated in degree $0$,
 we recover the usual notion of a crossed module, which  is  also 
called a {\it strict Lie 2--algebra } \cite{BaezCranz}.
 Since it  is the only case we are interested in,
 we omit the term ``strict" in 
Definition \ref{def:strictLie2}.

\medskip

\begin{remark}
\label{rmk:item_c}
The conditions (a) and (b) in Definition \ref{def:strictLie2}
 imply that $\rightgla$ acts
 on $\leftgla$ by derivations of graded Lie algebras.
 Moreover, the fact that $d$ respects the Lie algebra bracket 
is a consequence of (a) and (b),
 and can be omitted from Definition \ref{def:strictLie2}.
\end{remark}

To any $\zz$-graded Lie $2$-algebra $\leftgla \stackrel{\diff}{\mapsto} \rightgla$, there is an associated differential graded Lie algebra, denoted
 $\Nu( \leftgla \stackrel{\diff}{\mapsto} \rightgla)$,
which is defined  as follows:
\begin{enumerate}
 \item[(i)] as a graded vector space, $\Nu= \leftgla[1] \oplus \rightgla$,
 i.e. for any $k\in {\mathbb Z}$,
the degree $k$-component $\assogla_k$  is the direct sum 
 $ \leftgla_{k+1} \oplus \rightgla_{k}$;
 \item[(ii)] the differential
 is $\diff(a\oplus\pi) =0\oplus\diff a $ for all $a \in \leftgla_{k+1} \subset \assogla_k $
 and  $\pi\in \rightgla_k \subset \assogla_{k}$
 \item[(iii)] the graded Lie bracket is given, 
  for all $ a_1 \oplus \pi_1 \in \assogla_k$ and  $a_2 \oplus \pi_2 \in \assogla_l $ by

  \begin{eqnarray*}[ a_1 \oplus \pi_1 ,  a_2 \oplus \pi_2 ]  &:=&  ( (-1)^k  \pi_1 \cdot a_2 - (-1)^{l}  a_1\cdot \pi_2)\oplus [\pi_1,\pi_2]\, .\\
   & =&
  (  (-1)^k \pi_1 \cdot a_2 - (-1)^{(k+1)l} \pi_2\cdot a_1 )\oplus [\pi_1,\pi_2].
  \end{eqnarray*}
\end{enumerate}
In the sequel, we will denote by $ a \cdot \pi$ the element
 $ (-1)^{k(l+1)} \pi \cdot a$ for all $\pi \in \rightgla_k$ and $a \in \leftgla_l$.

For an ordinary crossed module  $\leftgla \stackrel{\diff}{\mapsto} \rightgla$ (i.e. the non-graded case), the only non vanishing components are 
$\Nu_0=\rightgla$ and $\Nu_{-1}=\leftgla$. 
In this case, a $L_\infty$-morphism from the dgla $\Nu(\leftgla \stackrel{\diff}{\mapsto} \rightgla)$
to the dgla $\Nu(\leftgla' \stackrel{\diff'}{\mapsto} \rightgla')$ is
 determined by a pair of linear maps
  $\leftgla \to \leftgla' $ and $  \rightgla \to  \rightgla'$,
 respectively,	 together with a bilinear skew-symmetric map
   $ \wedge^2 \rightgla \to \leftgla'$. For degree reasons,
 no other Taylor coefficients may exist.
This is no longer true for $\zz$-graded crossed modules.
Below we introduce the  notion of morphisms of $\zz$-graded crossed modules
(or {\strict}s) by imposing these conditions.

\begin{definition}
\label{def:strictLie2_morph}
A  morphism of {\strict}s from
 $ \leftgla \stackrel{\diff}{\mapsto} \rightgla$ to
 $ \leftgla' \stackrel{\diff'}{\mapsto} \rightgla'$ is an
 $L_\infty$-morphism $\Phi$ between their associated
dglas $ \assogla(\leftgla \stackrel{\diff}{\mapsto} \rightgla) $ and $ \assogla(\leftgla' \stackrel{\diff'}{\mapsto} \rightgla') $   whose Taylor coefficients $(\Phi_n)_{n \geq 1}$ satisfy the following properties:
\begin{enumerate}
 \item[(i)] the \emph{linear Taylor coefficient} $\Phi_1$ maps 
	 $\leftgla$ to $\leftgla'$ and maps  $\rightgla$ to $\rightgla'$;
 \item[(ii)] the only non-trivial component
of the  \emph{quadratic Taylor coefficient} is
 $\Phi_2: \wedge^2 \rightgla\to \leftgla'$;
 \item[(iii)] all higher Taylor coefficients $(\Phi_n)_{n \geq 3}$ vanish.
\end{enumerate}
\end{definition}

When the quadratic Taylor coefficient $\Phi_2$ is zero, 
we call it  a \emph{strict} morphism of {\strict}s. 
Strict morphisms are then simply pairs of maps $\leftgla \to \leftgla'$ 
and  $\rightgla \to \rightgla'$
 that preserve the structures defining {\strict}s.

Let us spell out Definition \ref{def:strictLie2_morph}.
 A morphism  $\Phi$ consists of  a pair of degree $0$ linear maps
 $\Phi_1: \leftgla \to \leftgla'$ and
  $\Phi_1:\rightgla \to \rightgla'$,
 called {\em the linear terms},
together with a graded skew-symmetric bilinear map
 $\Phi_2 : \wedge^2 \rightgla   \to  \leftgla'$
of degree $+1$, called {\em the quadratic term}, such that:
\begin{enumerate}
 \item[(a)] $\Phi_1$ is a chain map:
       $\xymatrix{ \leftgla \ar[r]^{\diff} \ar[d]^{\Phi_1}  &  \rightgla \ar[d]^{\Phi_1} \\ \leftgla' \ar[r]^{\diff'} & \rightgla' },$
 \item[(b)] for all $\pi_1,\pi_2 \in \rightgla $, the relation
 $ (\diff' \circ  \Phi_2)\, (\pi_1,\pi_2) =  \Phi_1([\pi_1,\pi_2])-[\Phi_1(\pi_1),\Phi_1(\pi_2)]  $ holds, 
 \item[(c)]  for all $\pi \in \rightgla, a \in \leftgla $, the relation
 $  \Phi_2 (\pi \cdot \diff a) =
  \Phi_1(\pi\cdot a)-\Phi_1(\pi) \cdot \Phi_1(a)$ holds,
 \item[(d)] the relation $(-1)^{|\pi_1| |\pi_3|} \left( \Phi_2 (\pi_1,[\pi_2,\pi_3]) - \Phi_1(\pi_1) \cdot \Phi_2 (\pi_2,\pi_3) \right)+ \circlearrowright{\hbox{\tiny{$\pi_1\pi_2\pi_3$}}} = 0  $ 
 holds for all $\pi_1,\pi_2, \pi_3 \in \rightgla $.
\end{enumerate} 
In the non-graded case, these are  exactly the  relations
satisfied by the Taylor coefficients of an $L_\infty$-morphism between dglas
 concentrated in degrees $0$ and $-1$ (see \cite{NOOHI}).

\begin{remark}
Morphisms of $L_\infty$-algebras can be composed; it is routine to check that morphisms of graded Lie 2-algebras 
are stable under the  composition of $L_\infty$-morphisms.
 Indeed, if $\Phi$ and $\Psi$ are morphisms of {\strict}s, 
so is  $\Phi\circ\Psi$,  whose only non-vanishing terms are 
the linear and quadratic ones that read as follows:
 \begin{equation}
  \label{eq:compositionStrictLie2}
  (\Phi \circ \Psi)_1 = \Phi_1 \circ \Psi_1  \hbox{ and } (\Phi \circ \Psi)_2 = \Phi_1 \circ \Psi_2 + \Phi_2 \circ (\wedge^2 \Psi_1) \, .
 \end{equation}
\end{remark}

\medskip
The following definition generalizes to the graded case the notion of
 homotopy between morphisms of Lie 2-algebras (see, for instance, 
 \cite[Definition 2.9]{NOOHI}).
Again, for the non-graded case, such homotopies (called natural transformations) are the only possible ones. 
In the graded case, we impose their form to mimic the non-graded case.

\begin{definition}
\label{def:strictLie2_morph_homotopy}
Let $\Phi$ and $\Psi$ be morphisms of {\strict}s from 
 $ \leftgla \stackrel{\diff}{\mapsto} \rightgla$ to
 $ \leftgla' \stackrel{\diff'}{\mapsto} \rightgla'$. 
An \emph{homotopy} between $\Phi$ and $\Psi$ is a linear map
 $h : \rightgla \to \leftgla'$ of 
degree\footnote{Note that $h$ becomes of degree $-1$ if $ \rightgla$ and $\leftgla' $ 
are seen as subspaces of their associated dglas.} $0$ such that:
\begin{enumerate}
 \item[(i)] $h$ is a homotopy between the chain maps
 $\Phi_1$ and $\Psi_1$,
        i.e., 
\[ 
\Psi_1\Big|_{\rightgla} =\Phi_1\Big|_{\rightgla}+\diff' \smalcirc\, h\, ,\qquad
\Psi_1\Big|_{\leftgla} =\Phi_1\Big|_{\leftgla}+ h \smalcirc\, \diff 
\]
 \item[(ii)] 
 \begin{equation}
  \label{eq:defhomotopy}
\Psi_2=\Phi_2 +\Theta_h^\Phi;
 \end{equation}
 where $\Theta_h^\Phi: \leftgla\times \leftgla'\to \leftgla'$ is the map
 defined, for all $\pi_1 \in \rightgla , \pi_2 \in \rightgla_l$ by 
  \begin{equation}
  \label{eq:defTheta}
\Theta_h^\Phi (\pi_1,\pi_2 ) :=  h([\pi_1,\pi_2]) - [h(\pi_1),h(\pi_2)]
 -\Phi_1(\pi_1) \cdot h(\pi_2) + (-1)^l h(\pi_1) \cdot \Phi_1(\pi_2). 
  \end{equation}
\end{enumerate}
\end{definition}

It is straightforward  to check  that this definition is
indeed  compatible with the usual notion of homotopy of $L_\infty$-morphisms.

\begin{proposition}
\label{lem:homotopyMakesSense1}
\begin{enumerate}
\item[(i)]Homotopy is an equivalence relation on  morphisms 
between  $\mathbb Z$-graded Lie $2$-algebras; 
\item[(ii)] Composition of morphisms of {\strict}s is compatible with 
homotopies. 
\end{enumerate}
\end{proposition}
\begin{proof}
 (i).
Assume that  $\Phi$ is homotopic to $\Psi$  with respect to a
 homotopy $h$.  From the relation $\Theta_h^\Phi=\Theta_{-h}^\Psi$,
it follows that
$\Psi$ is homotopic to $\Phi$ with respect to
the homotopy $-h$.

Now assume that $\Phi$ is homotopic to $\Psi$ with respect to a
homotopy $h$, and $\Psi$ is homotopic to $\Xi$ with respect to a
homotopy $g$. We need to prove that $\Phi$ is homotopic to $\Xi$ 
with respect to the homotopy $h+g$. For this purpose, it suffices
to prove the following relation:
\begin{equation}
\label{eq:lem:homotopyMakesSenseA} 
\Theta_{h+g}^\Phi =  \Theta_{g}^\Psi + \Theta_{h}^\Phi
\end{equation}
with $h,g:\rightgla \to \leftgla'$.
For  any $\pi_1 \in \rightgla_k$ and $\pi_2 \in \rightgla_l$,
\begin{eqnarray*}
   \Theta_{g}^\Psi \, (\pi_1,\pi_2) &=& g([\pi_1,\pi_2]) -[g(\pi_1),g(\pi_2)] +
   (-1)^l g(\pi_1) \cdot \Psi_1(\pi_2)  - \Psi_1(\pi_1) \cdot g(\pi_2) \\
  &=& g([\pi_1,\pi_2]) - [g(\pi_1),g(\pi_2)] + (-1)^l g(\pi_1) \cdot \Phi_1(\pi_2) - \Phi_1(\pi_1) \cdot g(\pi_2) \\
 & & + (-1)^l  g(\pi_1) \cdot (\diff' \circ h) (\pi_2) - (\diff' \circ h)(\pi_1) \cdot g(\pi_2) \\
  &=& g([\pi_1,\pi_2]) - [g(\pi_1),g(\pi_2)]  + (-1)^l
  g(\pi_1) \cdot \Phi_1(\pi_2) - \Phi_1(\pi_1) \cdot g(\pi_2)  \\
  & & -  [ g(\pi_1) , h (\pi_2) ] - [h (\pi_1 )   , g (\pi_2) ], 
  \end{eqnarray*}
where we used  the relation:
$\Psi_1 (\pi) = \Phi_1 (\pi) + (\diff' \circ h) (\pi)$,
 $\forall 	\pi \in {\rightgla}$,
in the second equality,
 and  Definition \ref{def:strictLie2} (ii) in  the last  equality.
Equation (\ref{eq:lem:homotopyMakesSenseA}) thus 
follows immediately. This completes the proof of (i).

 Let us prove (ii). Let  $\Phi$ and $\Phi'$ be 
 homotopic morphisms of {\strict}s and let $h$ be a homotopy between them.
  For all morphisms $\Psi$ and $\Xi$ such that the compositions
 $ \Psi \circ \Phi \circ \Xi $ and $  \Psi \circ \Phi' \circ \Xi $ make sense, the degree $0$ linear map
 $\Xi_1 \circ h \circ \Psi_1 $ is a homotopy between them.  This implies that
 for any  homotopic morphisms  $\Phi$ and  $\Psi$ from 
 $\leftgla \stackrel{\diff}{\mapsto} \rightgla$
 to $\leftgla' \stackrel{\diff'}{\mapsto} \rightgla'$
 and any  homotopic morphisms $\Phi'$ and $\Psi'$ from
  $\leftgla' \stackrel{\diff'}{\mapsto} \rightgla' $ to $\leftgla'' \stackrel{\diff''}{\mapsto} \rightgla''$,
 the composition $\Psi \circ \Phi $ is homotopic to $\Psi' \circ \Phi$.
 The latter is  homotopic to $\Psi' \circ \Phi' $. This proves the claim.
\end{proof}

Proposition \ref{lem:homotopyMakesSense1} allows
us  to make sense of the following:

\begin{definition}
A \emph{homotopy equivalence} between {\strict}s 
$(\leftgla \stackrel{\diff}{\mapsto} \rightgla)$
and $(\leftgla' \stackrel{\diff'}{\mapsto} \rightgla')$ is a pair of 
morphisms of {\strict}s: 
\begin{equation}
 \xymatrix{ \rightgla & \rightgla'   \\
    \ar@{}[r]^(.25){}="a"^(.75){}="b" \ar@<8pt>@{->}^{\Phi} "a";"b"  &    \ar@{}[l]^(.25){}="a"^(.75){}="b" \ar@<8pt>@{->}^{\Psi} "a";"b"   \\
 \leftgla  \ar[uu]^{\diff}  & \leftgla'  \ar[uu]_{\diff'} 
}  
\end{equation}
 such that $\Phi \circ \Psi$ and $\Psi \circ \Phi$ are homotopic to the identity map.
The morphism $\Phi$ (resp. $\Psi$) is 
 said to be an \emph{homotopy inverse} of $\Psi $ (resp. $\Phi$).
\end{definition}

\subsection{Homotopy inverses of morphisms of {\strict}s}

Recall that if $\Phi$ is a morphism of {\strict}s from 
$\leftgla \stackrel{\diff}{\mapsto} \rightgla$ to
 $\leftgla'\stackrel{\diff'}{\mapsto} \rightgla'$, its 
 linear part $\Phi_1$ is a chain map. 
 An homotopy inverse of $\Phi_1$
is a graded chain map $\Psi_1$ from $\leftgla' \stackrel{\diff'}{\mapsto}
 \rightgla'$ to
 $\leftgla \stackrel{\diff}{\mapsto} \rightgla$ together with an homotopy
$h$ between $\Psi_1\circ\Phi_1$ and the identity map,
 and an homotopy $h'$ between $\Phi_1\circ\Psi_1$ and the identity map:

\begin{equation}\label{inverse_chain_map}
 \xymatrix{  \ar@[blue]@/_{16pt}/[dd]_{h} \rightgla & \rightgla' \ar@[blue]@/^{16pt}/[dd]^{h'}  \\
    \ar@{}[r]^(.25){}="a"^(.75){}="b" \ar@[red]@<4pt>@{->}^{\Phi_1} "a";"b"  &    \ar@{}[l]^(.25){}="a"^(.75){}="b" \ar@[red]@<4pt>@{->}^{\Psi_1} "a";"b"   \\
 \leftgla  \ar[uu]^{\diff}  & \leftgla'  \ar[uu]_{\diff'} .}  
\end{equation}

The following theorem states that a morphism of $\zz$-graded Lie 2-algebras
 has an homotopy inverse as long as its linear part  is homotopy
invertible as a  chain map.
 Moreover,
the homotopy  inverse  is uniquely determined by the 
homotopy inverse  of the linear part.
 
\begin{theorem}
\label{th:invertingStrictLie2Morphisms}
Let $\Phi$ be a morphism of {\strict}s from $ \leftgla \stackrel{\diff}{\mapsto} \rightgla $ 
to $ \leftgla' \stackrel{\diff'}{\mapsto}  \rightgla' $ and
let $\Phi_1$ be its linear Taylor coefficient.
Assume that an homotopy inverse  $\Psi_1$ of $\Phi_1$ is given
with homotopies $h$ and $h'$ as in (\ref{inverse_chain_map}). Then 
there exists a unique morphism 
 of {\strict}s $\Psi$ from $\leftgla'\stackrel{\diff'}{\mapsto} \rightgla'$ 
to $ \leftgla \stackrel{\diff}{\mapsto}  \rightgla $ 
such that
\begin{enumerate}
 \item[(i)] $\Psi_1$ is the linear Taylor coefficient of $\Psi$; and
 \item[(ii)] $h$ (resp. $h'$) is a homotopy of morphisms of 
$\zz$-graded Lie $2$-algebras between the composition
 $\Phi \circ \Psi$ (resp. $\Psi \circ \Phi$) and the identity map.
\end{enumerate}
\end{theorem}
\begin{proof}
Let $\Psi_2 : \wedge^2 \rightgla' \to \leftgla$ be the  quadratic term
of $\Psi$. Then  $\forall \pi_1',\pi_2',\pi' \in \rightgla'$, $a'\in\leftgla'$, $\pi_1,\pi_2\in\rightgla$, 
\begin{equation}
 \label{eq:constraints}
 \left\{ \begin{array}{rcl}
           \diff \, \Psi_2 (\pi_1',\pi_2') & = & \Psi_1 ([\pi_1',\pi_2'])  - [ \Psi_1(\pi_1') , \Psi_1 (\pi_2') ]    \\
           \Psi_2(\pi',\diff'a') &=& \Psi_1(\pi'\cdot a')-\Psi_1(\pi')\cdot \Psi_1(a')\\
            (\Phi_1 \circ \Psi_2)  (\pi_1',\pi_2')   & = & \Theta_{h'}^{id} ( \pi_1',\pi_2' ) - \Phi_2 \left( \Psi_1(\pi_1') , \Psi_1(\pi_2') \right)  \\
             (\Psi_1 \circ \Phi_2)  (\pi_1,\pi_2)   & = & \Theta_{h}^{id} ( \pi_1,\pi_2 ) - \Psi_2 \left( \Phi_1(\pi_1) , \Phi_1(\pi_2) \right) 
           \end{array} \right.
\end{equation}
where $\Theta_h^{id}$ and $\Theta_{h'}^{id}$ are defined as in Equation (\ref{eq:defTheta}). The first two relations above say that $\Psi_2$ is the quadratic Taylor coefficient of a  
{\strict} morphism whose linear Taylor coefficient is $\Psi_1$. The third (resp. fourth) relations say that $h'$ and $h$ are the homotopies between  
$\Phi\circ\Psi$ (resp. $\Psi\circ\Phi$) and the identity map.

In order to prove the uniqueness, note that
 if both  $\Psi_2$ and $\tilde\Psi_2$ satisfy (\ref{eq:constraints}),
 then ${\rm Im}(\Psi_2-\tilde\Psi_2)\subset\Ker\diff\cap\Ker\Phi_1=0$, 
since $\Psi_1$ and $\Phi_1$ are homotopy inverse to each other.

Existence is proved by describing
 an explicit formula of $ \Psi_2$ that satisfies  (\ref{eq:constraints}).
Let 
\begin{equation}
 \label{eq:explicitPsi2}
\Psi_2 := \Psi_1 \circ \Theta_{h'}^{id} + h \circ \kappa_{\Psi_1} -
 \Psi_1 \circ \Phi_2 \circ (\wedge^2 \Psi_1)  ,  
 \end{equation}
where 
$$ \kappa_{\Psi_1}  (\pi_1',\pi_2') = -\Psi_1 ([\pi_1',\pi_2'])  + [ \Psi_1(\pi_1') , \Psi_1 (\pi_2') ]  .$$

It follows from a tedious but direct computation  that $\Psi_2$ indeed satisfies 
 \eqref{eq:constraints}.
 \end{proof}

\subsection{Maurer-Cartan moduli set of a {\strict}}\label{set:MC}

Let $\leftgla \stackrel{\diff}{\mapsto}\rightgla $ be a {\strict} and $\Nu :=\Nu(\leftgla \stackrel{\diff}{\mapsto}\rightgla)$ its associated dgla.
The \emph{Maurer-Cartan elements of $\leftgla \stackrel{\diff}{\mapsto}\rightgla $} are the Maurer-Cartan elements of its associated dgla. 
The set of Maurer-Cartan elements is denoted by $ MC(\leftgla \stackrel{\diff}{\mapsto}\rightgla) $.
 
\begin{lemma}
\label{lem:maurerCartanStrict}
Maurer-Cartan elements of a {\strict} $\leftgla \stackrel{\diff}{\mapsto}\rightgla $
are elements  $ \Lambda\oplus\Pi \in  \Nu_1=\leftgla_2\oplus\rightgla_1$
 satisfying
 $$ \diff \Lambda + \frac{1}{2} [\Pi,\Pi] =0,\ \   \hbox{ and }  \Pi \cdot \Lambda =0  .$$
\end{lemma}

For any
  Maurer-Cartan element $ \Lambda\oplus\Pi \in  \leftgla_2\oplus\rightgla_1$
 and any 
${\twist} \in \leftgla_{1} $, define $\Lambda_\twist\oplus\Pi_\twist \in
  \leftgla_2\oplus\rightgla_1$ by:
 \begin{equation}
 \label{eq:twistGeneral}
  \Pi_{\twist} := \Pi + {\rm d} {\twist} \hbox{ and } \Lambda_{\twist} :=\Lambda - \Pi \cdot {\twist} - \frac{1}{2} [{\twist},{\twist}]. 
 \end{equation} 
 
Then
$ \Lambda_{\twist} \oplus \Pi_{\twist}$
is called   the {\it twist} of $\Lambda\oplus\Pi$
by $T$, denoted $(\Lambda\oplus\Pi)_\twist$.
Twist transformations are related to  gauge transformations of dglas. 
Recall that two Maurer-Cartan elements $m$ and $m'$ in a dgla are said to be
 \emph{gauge equivalent},
if   there exists an element $b$ of degree $0$, called a \emph{gauge element},
such that $m' = {\rm exp} (b) \cdot m$ where
\begin{equation}
 \label{eq:gauge}
  {\rm exp}(b) \cdot m  := m - \sum_{i \geq 0} \frac{ {\rm ad}_b^{i}  }{(i+1)!} \left( {\rm d} b + [m,b] \right).
\end{equation}
See, for instance, \cite[Equation (3.7)]{BGNT}. 
In general, the right hand side of Equation \eqref{eq:gauge} may not
be convergent.
However when the gauge element  $b$ is a nilpotent element
of the graded Lie algebra,
 the right hand side  of Equation  \eqref{eq:gauge}
is well defined.
 In our situation, it is clear that 
$\leftgla_1\subset\Nu_0=\leftgla_1\oplus\rightgla_0$ is
an abelian Lie subalgebra; in particular the gauge transformations
 \eqref{eq:gauge} makes sense for all $T\in\leftgla_1$.
 
 \begin{proposition}
 \label{prop:gauge=twist}
 Let $\leftgla \stackrel{\diff}{\mapsto} \rightgla $ be a {\strict}
 and $\Nu := \Nu(\leftgla[1] \stackrel{\diff}{\mapsto} \rightgla)$ its associated dgla. For any Maurer-Cartan element $ \Lambda\oplus\Pi$ and 
 $T\in\leftgla_1$,
 $$
 {\rm exp}(-T) \cdot(\Lambda\oplus\Pi) = (\Lambda\oplus\Pi)_T\;.
 $$
 \end{proposition}
\begin{proof} 
A direct computation gives:
\begin{eqnarray*}
 \diff \left( {\twist} \oplus 0 \right)   + \left[{\twist} \oplus 0  , \Lambda \oplus \Pi \right] &=& -\Pi \cdot {\twist}\oplus  \diff {\twist},\\
  {\rm ad}_{\twist} \left( \diff \left( {\twist}\oplus 0 \right)   +
 \left[ {\twist}\oplus 0 , \Lambda\oplus\Pi \right] \right) &=&
 [{\twist},{\twist}]\oplus  0, \\
\color{black}
  {\rm ad}_{\twist}^i \left(    \diff \left( 0 \oplus {\twist} \right)   + \left[ {\twist}\oplus 0  , \Lambda\oplus\Pi \right] \right) &=& 0 \hbox{ for $i \geq 2$}\, .
\end{eqnarray*}
The result then follows by using these relations to compare the right hand side of (\ref{eq:gauge})  with (\ref{eq:twistGeneral}).
\end{proof}

\begin{corollary}
 For any  Maurer-Cartan element $\Lambda \oplus \Pi $ of
 a {\strict} $ \leftgla \stackrel{\diff}{\mapsto} \rightgla$ and
any  ${\twist} \in \leftgla_1$,
the element  $  \left( \Lambda \oplus \Pi \right)_{\twist}$ is 
also  a Maurer-Cartan element.
 Moreover, twist transformations define an equivalence relation on 
$MC(\leftgla \stackrel{\diff}{\mapsto}\rightgla )$.
\end{corollary}


\smallskip
\begin{definition}\label{MC_moduli_set}
 The {\it Maurer-Cartan moduli set} $\underline{MC(\leftgla \stackrel{\diff}{\mapsto}\rightgla)}$ is the quotient of $MC(\leftgla \stackrel{\diff}{\mapsto}\rightgla)$ by twist equivalence.
\end{definition}

It is a general fact that if $\{F_n\}_{n\geq 0}$ is a morphism of 
$L_\infty$ algebras from $\g_1$ to $\g_2$
 and $m\in \g_1$ is a Maurer-Cartan element, then
$$
\sum_{n=0}^\infty\frac{1}{n!} F_n(m,,\ldots,m),
$$
if it is convergent,
is a Maurer-Cartan element of $\g_2$. By applying this formula to a
morphism $\Phi$ of {\strict}s from  $ \leftgla \stackrel{\diff}{\mapsto} \rightgla $
to  $  \leftgla' \stackrel{\diff'}{\mapsto} \rightgla' $, 
we obtain a  map 
$MC(\Phi): MC(\leftgla \stackrel{\diff}{\mapsto} \rightgla) \to MC(\leftgla' \stackrel{\diff'}{\mapsto} \rightgla')$ that reads
\begin{equation}
 \label{eq:inducedMap}
 MC(\Phi)(\Lambda\oplus\Pi) =   \left( \, \Phi_1 (\Lambda) + \frac{1}{2} \Phi_2(\Pi,\Pi) \, \right)\oplus \Phi_1(\Pi) \;.
\end{equation}

The following result is also a straightforward consequence of a general result
 valid for any  $L_\infty$-morphisms.
For completeness,   we outline a  proof below.

\begin{lemma}\label{MC_functor_morphisms}
Let $\Phi$ be a morphism of graded Lie 2-algebras from  $ \leftgla \stackrel{\diff}{\mapsto} \rightgla $
to  $ \leftgla' \stackrel{\diff'}{\mapsto} \rightgla'$.
Then  $MC({\Phi})$ maps twist equivalent Maurer-Cartan elements 
to  twist equivalent Maurer-Cartan elements.
\end{lemma}
\begin{proof}
We  prove that if ${\twist}\in\leftgla_1$ and $\Lambda\oplus\Pi\in MC(\leftgla \stackrel{\diff}{\mapsto} \rightgla)$, then
\begin{equation} 
\label{MC-twist}
MC(\Phi)\left( (\Lambda\oplus\Pi)_{\twist}  \right) =  \left( MC(\Phi)\left(\Lambda\oplus \Pi  \right) \right)_{\Phi_1 ({\twist})} .
\end{equation}
In view of the definition of twist equivalence (\ref{eq:twistGeneral}),
Equation \eqref{MC-twist}  decomposes into the following two relations:
\begin{equation}
 \label{eq:twoEquations} \left\{ \begin{array}{rcl} 
     \Phi_1 (\Pi + \diff {\twist} ) & = & \Phi_1 (\Pi) +
( \diff' \circ \Phi_1) \, ({\twist}) \\ 
     \begin{array}{c}\Phi_1 \left( \Lambda - \Pi \cdot {\twist} - \frac{1}{2}[{\twist},{\twist}]  \right) \\  + \frac{1}{2} \Phi_2 \left( \Pi + \diff {\twist}, \Pi + \diff {\twist} \right) \end{array} & = &  
     \begin{array}{c}
     \Phi_1 ( \Lambda )  + \frac{1}{2}  \Phi_2 ( \Pi, \Pi) \\ -  \Phi_1( \Pi) \cdot \Phi_1 ( {\twist} ) - \frac{1}{2} [ \Phi_1 ( {\twist} ),\Phi_1 ( {\twist} )] \end{array}
    \end{array}  \right.
\end{equation}
The first equation follows from the fact that $\Phi_1$ is a chain map.
We prove the second equation by a direct computation. First,
 by the definition of {\strict} morphism, for all $P \in \rightgla$,
we have
 \begin{equation}
  \label{eq:Phi2}
\Phi_2(P,\diff{\twist})  =  \Phi_1 (P \cdot{\twist}) - \Phi_1 ( P) \cdot \Phi_1 ({\twist})\, .  
 \end{equation}
In particular,
 for $P=\diff {\twist}$, Equation (\ref{eq:Phi2})  implies that
 \begin{eqnarray*}
  \Phi_2( \diff {\twist},\diff{\twist})  &=& 
 \Phi_1 (\diff {\twist} \cdot{\twist}) - \Phi_1 ( \diff {\twist} ) \cdot \Phi_1 ({\twist}) \\
   &=&  \Phi_1 (\diff {\twist} \cdot{\twist}) -
 \left( (\diff' \circ \Phi_1) (  {\twist})  \right) \cdot \Phi_1 ({\twist}) \\
  &=&  \Phi_1 ([{\twist},{\twist}]) - [\Phi_1 ({\twist}), \Phi_1({\twist})].
 \end{eqnarray*}
The second equation in  (\ref{eq:twoEquations}) follows from the previous relation and Equation (\ref{eq:Phi2})  (being applied to $P=\Pi$).
\end{proof}

As a consequence, $MC(\Phi): MC(\leftgla \stackrel{\diff}{\mapsto} \rightgla) \to MC(\leftgla' \stackrel{\diff'}{\mapsto} \rightgla')$
induces a map between Maurer-Cartan moduli sets: 
$$ \underline{MC(\Phi)}: \underline{MC(\leftgla \stackrel{\diff}{\mapsto} \rightgla)} \to \underline{MC(\leftgla' \stackrel{\diff'}{\mapsto} \rightgla')}. $$

In the following lemma, we prove  that such a map depends
 only on the homotopy type of $\Phi$.

\begin{lemma}\label{MC_homotopy}
Let $\Phi$ and $\Psi$ be morphisms of {\strict}s from  $ \leftgla \stackrel{\diff}{\mapsto} \rightgla $
to  $  \leftgla' \stackrel{\diff'}{\mapsto} \rightgla' $,
 and let $h:\rightgla\rightarrow\leftgla'$ be an homotopy between $\Phi$ and $\Psi$; then 
$$MC(\Psi) (\Lambda\oplus\Pi)=MC(\Phi)(\Lambda\oplus\Pi)_{h(\Pi)}$$ 
for all $\Lambda\oplus\Pi\in MC(\leftgla \stackrel{\diff}{\mapsto} \rightgla)$. As a consequence,
 $\underline{MC(\Phi)}=\underline{MC(\Psi)}$.
\end{lemma}
\begin{proof}
According to Equation (\ref{eq:defTheta}),  we have 
\begin{eqnarray}
\Theta^\Phi_h(\Pi,\Pi) &=& h ( [\Pi,\Pi]) - [h(\Pi),h(\Pi)] - 2 \Phi_1(\Pi) \cdot h(\Pi) \cr
   & =& -2 h ( \diff \Lambda) - [h(\Pi),h(\Pi)] - 2 \Phi_1(\Pi) \cdot h(\Pi) ,\label{eq:ThetaPiPi}
   \end{eqnarray}
where we used the Maurer-Cartan condition $[\Pi,\Pi]= - 2\diff \Lambda $.

  By Equations  (\ref{eq:inducedMap}) and \eqref{eq:defhomotopy}, 
and using the fact that $h$ is a homotopy between the chain maps
 $\Phi_1$ and $ \Psi_1$,
we  obtain
 \begin{eqnarray*}  
 MC(\Psi)(\Lambda\oplus\Pi) &=&  \left( \Psi_1 (\Lambda) + \frac{1}{2} \Psi_2 (\Pi,\Pi) \right)\oplus \Psi_1 (\Pi) \\
   & =& \left( \Phi_1(\Lambda) +  (h \circ \diff ) ( \Lambda)  +  \frac{1}{2} \Phi_2 (\Pi,\Pi) +  \frac{1}{2}  \Theta^\Phi_h (\Pi,\Pi) \right) \\
 & & \oplus   \left(  \Phi_1(\Pi) + (\diff' \circ h) ( \Pi) \right)  \\
 &=& \left( \Phi_1(\Lambda) + \frac{1}{2} \Phi_2 (\Pi,\Pi) - \frac{1}{2}[h(\Pi),h(\Pi)] -  \Phi_1(\Pi) \cdot h(\Pi)  \right)\\
 && \oplus
 \left(  \Phi_1(\Pi) + (\diff' \circ h) ( \Pi) \right)  \\
 &=& MC(\Phi)(\Lambda\oplus\Pi)_{h(\Pi)}.
 \end{eqnarray*}
where in the third equality, we used Equation (\ref{eq:ThetaPiPi}). 
 The result  thus follows. 
\end{proof}

Lemma \ref{MC_homotopy} implies immediately the following

\begin{corollary}
\label{cor:MC}
An homotopy equivalence between two strict Lie $2$-algebras induces
 a one-to-one correspondence between their Maurer-Cartan moduli sets.
\end{corollary}


We are now ready to consider the following assignments:
\begin{enumerate}
 \item to any {\strict} $\leftgla \stackrel{\diff}{\mapsto} \rightgla $, we associate the moduli set $\underline{MC(\leftgla \stackrel{\diff}{\mapsto} \rightgla)}$;
 \item to any morphism $\Phi$ from $\leftgla \stackrel{\diff}{\mapsto} \rightgla $ to $\leftgla' \stackrel{\diff'}{\mapsto} \rightgla' $, we associate a map 
 $ \underline{MC(\Phi)} : \underline{MC(\leftgla \stackrel{\diff}{\mapsto} \rightgla)} \to \underline{MC(\leftgla \stackrel{\diff}{\mapsto} \rightgla)}$.
 \end{enumerate}

According to  Lemma \ref{MC_functor_morphisms},  $ \underline{MC(\Phi)}$
is indeed well defined.
 It is simple to check that the relations $\underline{MC(\Phi \circ \Psi)}
 =\underline{MC(\Phi )} \circ \underline{MC( \Psi)}$ and
  $\underline{MC(id)} = id$ hold.
Moreover, if morphisms $\Phi$ and $\Psi$ are homotopic,
 then $\underline{MC(\Phi)}=\underline{MC(\Psi)}$ by Lemma \ref{MC_homotopy}.
Hence $\underline{MC}$ is a functor from the category $\Lieh$
(where objects are {\strict}s and arrows are homotopy classes of {\strict} morphisms) to the category of sets.
By Corollary \ref{cor:MC}, this functor is in fact valued in the subcategory 
of sets where objects are sets and all arrows are bijections. It will be called
the {\it Maurer-Cartan functor} $\underline{MC}$.

\section{$\zz$-graded  Lie groupoids and  cohomology}
\label{sec:AnnexSuper}

This section is devoted to establishing those results,
 which we need in order to prove Proposition~\ref{cor:superGroupoidsConclusion}.

\subsection{$\zz$-graded  Lie  groupoids and  truncated  
2-term groupoid cohomology complexes}

$\zz$-graded  Lie  groupoids are Lie groupoids in the category of
$\zz$-graded manifolds  \cite{Mehta, Mehta1}. 
 Many standard    constructions have straightforward extensions
 to the context of $\zz$-graded  Lie groupoids
including  groupoid cohomology, morphisms, and Morita morphisms.
In particular,  for a $\zz$-graded Lie groupoid 
${\mathcal G} \toto {\mathcal M}$, 
its  cohomology complex is the $\zz$-graded complex:
 \begin{equation}
 \label{eq:diffGrouComplex}
 \xymatrix{  {\mathcal C}^\infty({\mathcal M})\ar[r]^{\delta}  & {\mathcal C}^\infty({\mathcal G}) \ar[r]^{\delta}& {\mathcal C}^\infty({\mathcal G}^{(2)}) \ar[r]^{\delta}& \dots }
 \end{equation}
where
 for $k \geq 2$, ${\mathcal G}^{(k)}$ denotes the space of
 composable $k$-arrows.
Cocycles in  $C^\infty ({\mathcal G})$ 
 are called \emph{multiplicative} functions.
The space of multiplicative functions is
 denoted by 	 ${\mathcal Z}({\mathcal G})$.
We are interested in the shifted truncation of the complex
\eqref{eq:diffGrouComplex} at degree $1$.

\begin{definition}
\label{def:2termComplexInterestedIn}
The \emph{2-term truncated (groupoid cohomology) complex} of a  $\zz$-graded Lie groupoid
${\mathcal G} \toto {\mathcal M}$ is the graded 2-term  complex:
\begin{equation}
\label{eq:diffGrouComplex1}
 \xymatrix{  {\mathcal C}^\infty({\mathcal M})[1]\ar[r]^{\delta}  & {\mathcal Z}({\mathcal G})[1]}\, .
 \end{equation}
\end{definition}

A  morphism of $\zz$-graded groupoids
 $\Phi: {\mathcal G}' \mapsto  {\mathcal G} $ induces a 
cochain map $\Phi^*$ between their
$\zz$-{\superG} cohomology cochain complexes:
 \begin{equation}
 \label{eq:diffGrouComplexMaps}
 \xymatrix{  {\mathcal C}^\infty({\mathcal M}')\ar[r]^{\delta'}  & {\mathcal C}^\infty({\mathcal G}') \ar[r]^{\delta'}& {\mathcal C}^\infty({\mathcal G}'^{(2)}) \ar[r]^{\delta'}& \dots \\
 \ar[u]^{\phi^*} {\mathcal C}^\infty({\mathcal M})\ar[r]^{\delta}  & \ar[u]^{\Phi^*}{\mathcal C}^\infty({\mathcal G}) \ar[r]^{\delta}& {\mathcal C}^\infty({\mathcal G}^{(2)}) \ar[r]^{\delta}
 \ar[u]^{  } & \dots
 }
 \end{equation}
Therefore $\Phi^*$
induces a morphism of  their corresponding
truncated 2-term  $\zz$-graded complexes:

\begin{equation}
 \label{eq:diffGrouComplexChain}
\xymatrix{   {\mathcal Z}({\mathcal G})[1] &  {\mathcal Z}({\mathcal G}')[1]
 \\
\ar@{}[r]^(.25){}="a"^(.75){}="b" \ar@<4pt>@[red]@{->}^{\Phi^*} "a";"b"  & \\
\ar[uu]^{\delta}  {\mathcal C}^\infty({\mathcal M}) [1] &  {\mathcal C}^\infty({\mathcal M}') [1]\ar[uu]^{\delta'} }
 \end{equation}

When $\Phi$ is a Morita morphism, 
the map $\Phi^*$ in  \eqref{eq:diffGrouComplexChain} becomes a
quasi-isomorphism.  The main purpose of this section is to describe an explicit construction of its  homotopy inverse.

 From now on, assume that
$ {\mathcal G}' \toto {\mathcal M}'$ is the pullback groupoid
 ${\mathcal G}[{\mathcal X}] \toto {\mathcal X}$, 
and $\Phi: {\mathcal G}[{\mathcal X}] \to {\mathcal G}$
 is the natural projection,
where $\phi:{\mathcal X} \to {\mathcal M}$ is
 a surjective submersion of $\zz$-graded manifolds.
Assume that $\phi: {\mathcal X} \to {\mathcal M}$  admits a section
$\sigma: {\mathcal M}\to {\mathcal X}$. Introduce
maps $\hat{\sigma}: {\mathcal G}\to  {\mathcal G}[{\mathcal X}]$
and $\tau:  {\mathcal X}\to  {\mathcal G}[{\mathcal X}]$, respectively
by 
\begin{equation}
\label{eq:sigma}
\hat{\sigma}=(\sigma \smalcirc  t, \id, \sigma \smalcirc  s)
\end{equation}
and
\begin{equation}
\label{eq:tau}
\tau=(\id , \epsilon\smalcirc\phi, \sigma \smalcirc \phi),
\end{equation}
 where
we identify ${\mathcal G}[{\mathcal X}]$
with 
${\mathcal X} \times_{{\mathcal M}, t} {\mathcal G}\times_{{\mathcal M}, s} {\mathcal X}$ and $\epsilon:{\mathcal M}\rightarrow{\mathcal G}$ is the embedding
of units of ${\mathcal G}$. 
It is simple to check that  the pair of maps
$(\hat{\sigma}, \sigma)$ is a morphism of $\zz$-graded groupoids
from ${\mathcal G} \toto {\mathcal M}$ to 
  ${\mathcal G}[{\mathcal X}] \toto  {\mathcal X}$.
 Therefore, it induces a  morphism $\hat{\sigma}^*$ of the truncated  2-term 
complexes from
$ {\mathcal C}^\infty({\mathcal X})[1]\to {\mathcal Z}({\mathcal G} [{\mathcal X}])[1]$
to ${\mathcal C}^\infty({\mathcal M})[1]\to {\mathcal Z}({\mathcal G})[1]$.
Since $\Phi\smalcirc\, \hat{\sigma}=\id$, it follows that
$\hat{\sigma}^* \smalcirc\, \Phi^* =\id$.
Moreover, it is straightforward to check that 
$\Phi^*\smalcirc\, \hat{\sigma}^*$ is homotopic to the identity,
with  $\tau^*: {\mathcal Z}({{\mathcal G}}({\mathcal X}))[1]\to
 {\mathcal C}^\infty({\mathcal X})[1]$ being a homotopy map.

In general,  global sections $\sigma: {\mathcal M}\to {\mathcal X}$
may not exist. However, since   $\phi: {\mathcal X} \to {\mathcal M}$ 
is a surjective submersion, local sections always exist.
A standard argument using a partition of unity enable us
to construct a homotopy inverse of $\Phi^*$. 
More precisely, 
denote by $X$ and $M$
the base manifold of ${\mathcal X} $ and
${\mathcal M}$, respectively, and by $\varphi : X \to M$ 
the surjective submersion at the  level of base manifolds.
Choose a  nice open cover $(U_i)_{i \in S}  $  of
$M$.  Let $(\chi_i)_{i \in S}$ be a partition of unity
subject to the cover $(U_i)_{i \in S}$.
Denote by ${\mathcal U}_i$ the restriction of the graded manifold
${\mathcal M}$ to $U_i$,
and by $ \phi^{-1}({\mathcal U_i})$ the restriction of ${\mathcal X}$ to the 
open subset $\varphi^{-1}(U_i)$ of  $X$.
For each $i\in S$,
there exists a local section $\sigma_i : {\mathcal U}_i  \hookrightarrow  \phi^{-1}({\mathcal U}_i)$ of $\phi: {\mathcal X} \to {\mathcal M}$.
Let $\tau_i:  \phi^{-1}({\mathcal U}_i) \to  {\mathcal G}[{\mathcal X}]_{\phi^{\scalebox{0.7}-1}({\mathcal U}_i)}^{\phi^{\scalebox{0.7}-1}({\mathcal U}_i)}$
be the map  defined as  in Equation \eqref{eq:tau} with respect to the
 section $\sigma_i: {\mathcal U}_i  \hookrightarrow  \phi^{-1}({\mathcal U}_i)$.
Similar to Equation \eqref{eq:sigma}, 
for all $i_1, i_2\in S$,  denote by $\hat{\sigma}_{i_1, i_2}:
{\mathcal G}|_{{\mathcal U_{i_1}}}^{{\mathcal U_{i_2}}}\to
 {\mathcal G}[{\mathcal X}]_{\phi^{\scalebox{0.7}-1}({\mathcal U}_{i_1})}^{\phi^{\scalebox{0.7}-1}({\mathcal U}_{i_2})}$, the map  
\begin{equation}
\label{eq:sigma1}
  \hat{\sigma}_{i_1, i_2}=(\sigma_{i_1} \smalcirc\,  t, \id, \sigma_{i_2}	 \smalcirc\,  s),
\end{equation}
where ${\mathcal G}|_{{\mathcal U_{i_1}}}^{{\mathcal U_{i_2}}}
=s^{-1} ({{\mathcal U_{i_1}}})\cap t^{-1}({{\mathcal U_{i_2}}})$
with $s$ and $t$ being the source and target maps
of ${\mathcal G} \toto {\mathcal M}$; similarly for 
${\mathcal G}[{\mathcal X}]_{\phi^{-1}({\mathcal U}_{i_1})}^{\phi^{-1}({\mathcal U}_{i_2})}$.

Consider the maps
\begin{equation}
 \label{eq:I1}
I_1 :  {\mathcal C}^\infty( {{\mathcal G}}[{\mathcal X}])[1]
\to  {\mathcal C}^\infty({{\mathcal G}} )[1],  \ \ \ I_1= \sum_{i_1, i_2 \in S} 
(s^* \chi_{i_1})(t^* \chi_{i_2})\, \hat{\sigma}_{i_1, i_2}^*
\end{equation}

\begin{equation}
 \label{eq:I0}
I_0:  {\mathcal C}^\infty({\mathcal X})[1]\to {\mathcal C}^\infty({\mathcal M})[1],
\ \ \ I_0=\sum_{i \in S} \chi_i\, \sigma_i^*
\end{equation}

and

\begin{equation}
 \label{eq:HX}
 H: {\mathcal Z}({{\mathcal G}}({\mathcal X}))[1]  \to  {\mathcal C}^\infty({\mathcal X})[1] ,  \ \ \ 
H=\sum_{i \in S}  (\phi^*\chi_i)  \tau_i^* 
\end{equation}

The following proposition can be verified directly.

\begin{proposition}
\label{prop:inverseCandidateI}
Let ${\mathcal G} \toto {\mathcal M}$ be a $\zz$-graded
 groupoid, $ \phi :{\mathcal X} \to {\mathcal M}$ a surjective submersion,
and $ \Phi: {\mathcal G}[{\mathcal X}] \to {\mathcal G}$ the natural projection.
Then,
\begin{enumerate}
 \item[(i)] the pair $I:=(I_0,I_1)$ defines a morphism of truncated  2-term
 complexes from $ {\mathcal C}^\infty({\mathcal X})[1]\to
{\mathcal Z}({\mathcal G} [{\mathcal X}])[1]$
to ${\mathcal C}^\infty({\mathcal M})[1] \to {\mathcal Z}({\mathcal G})[1]$;
 \item[(ii)]   $I$ is a left inverse of $\Phi^*$;
 \item[(iii)] the composition $ \Phi^* \circ I $ is homotopic to the identity 
with $H$ being  a homotopy map.
\end{enumerate}
\end{proposition}
In summary, we have the following diagram:
 \begin{equation}
 \label{eq:diffGrouComplexMapsI}
\xymatrix{   {\mathcal Z}({\mathcal G})[1] &  {\mathcal Z}({\mathcal G}[{\mathcal X}])[1]
 \ar@/^/@[blue][dd]^{ H}
 \\
\ar@{}[r]^(.25){}="a"^(.75){}="b" \ar@<4pt>@[red]@{^{(}->}^{\Phi^*} "a";"b"  &
 \ar@{}[l]^(.25){}="a"^(.75){}="b" \ar@<4pt>@[red]@{_{}->}^{I} "a";"b"\\
\ar[uu]^{\delta}  {\mathcal C}^\infty({\mathcal M})[1]  &  {\mathcal C}^\infty({\mathcal X})[1] \ar[uu]^{\delta'}\, .}
 \end{equation}


\subsection{Proof of Proposition \ref{cor:superGroupoidsConclusion}}
\label{App:B2}

Every VB-groupoid $V\toto E$ defines a $\zz$-graded Lie groupoid $V_{[1]}\toto E_{[1]}$. The space of multiplicative functions 
${\mathcal Z}(V_{[1]})\subset\Gamma(\Lambda V^\vee)$ inherits
 the $\N$-grading with 
${\mathcal Z}^k(V_{[1]})\subset \Gamma(\Lambda^k V^\vee)$,  i.e. ${\mathcal Z}(V_{[1]})=\oplus_k {\mathcal Z}^k(V_{[1]})$. The following straightforward lemma gives an useful characterization.

\begin{lemma}\label{graded_grpd_cohomology}
Let $V\toto E$ be a VB-groupoid over $\Gamma\toto M$; any $P\in\Gamma(\Lambda^k V^\vee)$ is a multiplicative function, i.e. 
$P\in{\mathcal Z}^k(V_{[1]})$,
if and only if the function
$$F_P(\mu_1,\ldots,\mu_k)=\langle P,\mu_1\wedge\ldots\wedge\mu_k\rangle\; ,\;\;  \quad
(\mu_1,\ldots,\mu_k)\in V\times_\Gamma {\ldots} \times_\Gamma V$$ 
is a one cocycle of the Lie groupoid $V\times_\Gamma \ldots\times_\Gamma V\toto E\times_M \ldots \times_M E$ (considered as a subgroupoid of the 
direct product groupoid $V\times\ldots \times V\toto E\times\ldots\times E$).
\end{lemma}

For any Lie groupoid $\BaseGroupoid \toto M$ with
Lie algebroid $A$, the cotangent groupoid 
$T^\vee \BaseGroupoid \toto \BaseAlgebroid^\vee  $ is a VB-groupoid 
as  in Example \ref{eq:TGdual}.
Therefore, it gives rise to a $\zz$-graded Lie groupoid 
 $T^\vee\super \BaseGroupoid \toto \BaseAlgebroid^\vee\super$ \cite{Mehta, Gracia-Saz-Mehta}.
The following lemma follows from Lemma \ref{graded_grpd_cohomology} and the characterization of multiplicative polyvector fields given in 
Proposition 2.7 of \cite{IPLGX}.

\begin{lemma}
\label{lem:ConnectionMoritaMaps0} 
Let $\BaseGroupoid \toto M$  be a Lie groupoid. The truncated 
2-term complex of 
the $\zz$-graded groupoid $T^\vee\super \BaseGroupoid \toto \BaseAlgebroid^\vee\super$
coincides with the ${\mathbb Z}$-graded 2-term  complex
 $ \hSec{\bullet}{\BaseAlgebroid} \stackrel{\diff}{\to} \polymult{\bullet}{\BaseGroupoid}$  in Lemma  \ref{lem:rightAndleft}.
\end{lemma}

Now assume that $\varphi : X \to M$ is  a surjective submersion
and let $\hordis$ be an Ehresmann connection for $\varphi$.
It is simple to see that $(\Phi_\hordis, \phi_\hordis)$
 in  \eqref{eq:aa}
indeed defines a VB-groupoid morphism:

\begin{equation}
\label{eq:bba}
\begin{tikzcd}[row sep=scriptsize, column sep=scriptsize]
& T^\vee(\Gamma[X]) \arrow{dl}[swap]{\Phi_\nabla} \arrow[rr,shift left] \arrow[rr, shift right] \arrow[dd] & & A[X]^\vee \arrow{dl}{\phi_\nabla} \arrow[dd] \\
T^\vee\Gamma \arrow[rr, shift left, crossing over] \arrow[rr, shift right, crossing over] \arrow[dd] & & A^\vee\\
& \Gamma[X] \arrow[dl] \arrow[rr, shift left] \arrow[rr, shift right]& & 
X \arrow{dl}{\varphi} \\
\Gamma \arrow[rr, shift left] \arrow[rr, shift right] & & M \arrow[from= 2-3, crossing over] \\
\end{tikzcd}
\end{equation}
This in turn induces a morphism of  $\zz$-graded groupoids
from  $ T^\vee \super(\BaseGroupoid[X]) \toto (A[X])^\vee\super$
to $T^\vee \super\BaseGroupoid\toto A^\vee\super $.

The following result is just a rephrasing of the discussion 
preceding Proposition \ref{cor:superGroupoidsConclusion}
into the language of graded groupoids. For the proof,  it
 suffices to check that 
 $T^\vee (\BaseGroupoid[X])$ is isomorphic to the fibered product 
${A^\vee[X]} \times_{A^\vee} T^\vee \BaseGroupoid \times_{A^\vee} 
 {A^\vee [X]}$.

\begin{proposition}
\label{prop:ConnectionMoritaMaps}
Let $\BaseGroupoid \toto M$  be a Lie groupoid,
 $\varphi : X \to M$ a surjective submersion
and $\hordis$ an Ehresmann connection for $\varphi:X \to M$.
The pair $( \Phi_{\hordis}, \phi_\hordis)$ defined
in  \eqref{eq:aa}  is a Morita morphism of $\zz$-{\superG} 
from $ T^\vee\super (\BaseGroupoid[X]) \toto (A[X])^\vee\super $
 to $T^\vee\super \BaseGroupoid  \toto A^\vee\super $.
\end{proposition}

The following lemma can be  verified in a straightforward manner.

\begin{lemma}
\label{lem:ConnectionMoritaMaps1}
Let $\BaseGroupoid \toto M$  be a Lie groupoid, $\varphi : X \to M$ a surjective submersion
and $\hordis$ an Ehresmann connection for $\varphi$.
Under the identification as in  Lemma \ref{lem:ConnectionMoritaMaps0},
 the morphism of  2-term truncated complexes associated to the 
$\zz$-graded groupoid morphism $\Phi_\hordis$ as in
Proposition \ref{prop:ConnectionMoritaMaps} 
  coincides with the horizontal lift:
$$ 
\xymatrix{  \polymult{\bullet}{\BaseGroupoid} & \polymult{\bullet}{\BaseGroupoid[X]}   \\ 
\ar@{}[r]^(.25){}="a"^(.75){}="b" \ar@<4pt>@[red]@{^{(}->}^{\lambda_\hordis} "a";"b"  &  \\
\ar[uu]^{\diff} \hSec{\bullet}{\BaseAlgebroid}   & \hSec{\bullet}{\BaseAlgebroid[X]} \ar[uu]^{\diff'}. } 
$$
\end{lemma} 
\begin{proof}
It is straightforward  to see
that the dual of the maps $\Phi_\hordis$ and $ \phi_\hordis$ are the horizontal lifts $\lambda_\hordis$ defined in  (\ref{eq:hordis000}).
Since $(\Phi_\hordis, \phi_\hordis)$ is a $\zz$-graded groupoid morphism,
 its dual $\lambda_\hordis$ is a morphism of $\zz$-graded 2-term
 complexes.  This completes the proof.
\end{proof}

We are now ready to prove Proposition \ref{cor:superGroupoidsConclusion}.

\begin{proof}[Proof of Proposition \ref{cor:superGroupoidsConclusion}]
Applying Proposition \ref{prop:inverseCandidateI} to the Morita morphism
 described in Proposition \ref{prop:ConnectionMoritaMaps},
 we obtain a morphism $I=(I_0,I_1)$ of 2-term complexes 
from  $\hSec{\bullet}{\BaseAlgebroid[X]} \stackrel{\diff'}{\mapsto} \polymult{\bullet}{\BaseGroupoid[X]}$
to $ \hSec{\bullet}{\BaseAlgebroid} \stackrel{\diff}{\mapsto} \polymult{\bullet}{\BaseGroupoid}$ that is a left inverse to $\lambda_\hordis$, 
and a homotopy map $h_X :\polymult{\bullet}({\BaseGroupoid[X]}) \to\hSec{\bullet}{\BaseAlgebroid[X]}$. These maps depend on the choice of local sections of 
$\phi_\hordis: (A[X])^\vee_{[1]}\rightarrow A^\vee_{[1]}$.

Note that the image of  $\lambda_\hordis$ lies in 
$ \hSec{\bullet}{\BaseAlgebroid[X]}_{proj} \stackrel{\diff'}{\mapsto} 
\polymult{\bullet}({\BaseGroupoid[X]})_{proj}$.
In order to prove the first statement, we need  to show that it is possible to 
choose local sections of $\phi_\hordis: (A[X])^\vee_{[1]}\rightarrow A^\vee_{[1]}$
 so that: $(i)$ the restriction of $I$ to 
projectable elements is given by the
natural projection $\proj$;
$(ii)$ the restriction of the homotopy map $h_X$ to
$\polymult{\bullet}({\BaseGroupoid[X]})_{proj}$
yields a homotopy map $h_{\horlift}: \polymult{\bullet}({\BaseGroupoid[X]})_{proj}
\to \hSec{\bullet}{\BaseAlgebroid[X]}_{proj}$. 

Indeed, choose a  nice open cover $(U_i)_{i \in S}  $  of  $M$
so that  $\varphi: X\to M$ admits a family of local sections
$\sigma'_i: U_i\to \varphi^{-1} (U_i)$.
Denote by ${\mathcal U}_i$ the restriction of the graded manifold
$  \BaseAlgebroid^\vee \super$ to $U_i$,
and $ \varphi^{-1}({\mathcal U_i})$ the restriction of
 $ (\BaseAlgebroid [X])^\vee\super$
 to the open subset $\varphi^{-1}(U_i)\subset X$.
Therefore, for each $i\in S$,
there is an induced   local section
$\sigma_i : {\mathcal U}_i  \hookrightarrow  \varphi^{-1}({\mathcal U}_i)$ of
the submersion
$\phi_\nabla: (\BaseAlgebroid [X])^\vee\super \to \BaseAlgebroid^\vee\super$. It is now straightforward to check
that the maps $I$ and $h_X$ defined in the proof of
 Proposition \ref{prop:inverseCandidateI}
with these local sections do satisfy $(i)$ and $(ii)$.

For the second part of the proposition, let 
 $\RRR:= \lambda_\hordis \circ I $. By construction, $\RRR$ is a chain map
 from $\hSec{\bullet}{\BaseAlgebroid[X]} \stackrel{\diff'}{\mapsto}
 \polymult{\bullet}({\BaseGroupoid[X]})$
to $ \hSec{\bullet}{\BaseAlgebroid[X]}_{proj} \stackrel{\diff'}{\mapsto} 
\polymult{\bullet}({\BaseGroupoid[X]})_{proj}$.
According to Proposition \ref{prop:inverseCandidateI}, $\RRR$ is homotopic to
 the identity  map
 as a  chain map from $ \hSec{\bullet}{\BaseAlgebroid[X]} 
\stackrel{\diff'}{\mapsto} \polymult{\bullet}{\BaseGroupoid[X]}$ to itself,
where  the homotopy map is $h_X$, i.e.,
$$  {\mathfrak i} \circ \RRR =  {\rm id} + \diff \circ h_X+ h_X \circ \diff .$$
Also,
 $$ \RRR\smalcirc {\mathfrak i} = \lambda_\hordis \circ \proj = {\rm id} + \diff \circ h_{\horlift}+ h_{\horlift} \circ \diff  .$$
This concludes the proof.
\end{proof}

\newcommand{\namearticle}{}
\newcommand{\namejournal}{\emph}

\end{document}